\documentclass[11pt,a4paper]{amsart}

\usepackage[utf8]{inputenc}
\usepackage[T1]{fontenc}
\usepackage{graphicx}
\usepackage{hyperref}
\usepackage{amssymb}
\usepackage{amsmath}
\usepackage{enumerate}
\usepackage{tikz}
\usepackage{tikz-cd}
\usepackage{csquotes}
\usepackage{array}
\usepackage{multirow}
\usepackage{hhline}
\usepackage{mathdots}
\usepackage[capitalise]{cleveref}
\usepackage[normalem]{ulem}

\numberwithin{equation}{section}
\numberwithin{figure}{section}

\newtheorem{thm}{Theorem}[section]
\newtheorem{theorem}[thm]{Theorem}
\newtheorem{lemma}[thm]{Lemma}
\AddToHook{env/lemma/begin}{\crefalias{thm}{lemma}}
\newtheorem{prop}[thm]{Proposition}
\AddToHook{env/prop/begin}{\crefalias{thm}{prop}}
\newtheorem{proposition}[thm]{Proposition}
\AddToHook{env/proposition/begin}{\crefalias{thm}{proposition}}

\AddToHook{env/cor/begin}{\crefalias{thm}{cor}}
\newtheorem{corollary}[thm]{Corollary}
\AddToHook{env/corollary/begin}{\crefalias{thm}{corollary}}
\newtheorem{conjecture}[thm]{Conjecture}
\AddToHook{env/conjecture/begin}{\crefalias{thm}{conjecture}}

\AddToHook{env/question/begin}{\crefalias{thm}{question}}
\newtheorem{hypothesis}[thm]{Hypothesis}
\AddToHook{env/hypothesis/begin}{\crefalias{thm}{Hypothesis}}
\newtheorem{defi}[thm]{Definition}
\AddToHook{env/defi/begin}{\crefalias{thm}{defi}}
\newtheorem{definition}[thm]{Definition}
\AddToHook{env/definition/begin}{\crefalias{thm}{definition}}
\theoremstyle{definition}
\newtheorem{exa}[thm]{Example}
\AddToHook{env/exa/begin}{\crefalias{thm}{exa}}
\newtheorem{example}[thm]{Example}
\AddToHook{env/example/begin}{\crefalias{thm}{example}}

\AddToHook{env/rem/begin}{\crefalias{thm}{rem}}
\newtheorem{remark}[thm]{Remark}
\AddToHook{env/remark/begin}{\crefalias{thm}{remark}}
\newtheorem{setting}[thm]{Setting}
\AddToHook{env/setting/begin}{\crefalias{thm}{setting}}

\newcommand{\C}{\mathbb{C}}
\newcommand{\A}{\mathbb{A}}
\newcommand{\T}{\mathcal{T}}

\newcommand{\E}{\mathcal{E}}

\newcommand{\FF}{\mathbb{F}}
\renewcommand{\P}{\mathbb{P}}

\renewcommand{\epsilon}{\varepsilon}
\newcommand{\R}{\mathbb{R}}
\newcommand{\Q}{\mathbb{Q}}
\newcommand{\Z}{\mathbb{Z}}

\newcommand{\n}{\mathbf n}
\newcommand{\m}{\mathbf m}
\renewcommand{\d}{\mathbf d}
\renewcommand{\i}{\mathbf i}
\newcommand{\s}{\mathbf s}

\renewcommand{\u}{\mathbf u}

\newcommand{\sbar}{\overline{\s}}
\newcommand{\nbar}{\overline{\n}}
\newcommand{\mbar}{{\overline{\m}}}
\newcommand{\dbar}{\overline{\d}}

\newcommand{\N}{\mathbb{N}}
\newcommand{\D}{\mathcal{D}}
\newcommand{\X}{\mathcal{X}}

\newcommand{\OO}{\mathcal{O}}
\newcommand{\PP}{\mathcal{P}}
\newcommand{\x}{\underline{x}}
\newcommand{\p}{\mathbf{p}}
\newcommand{\q}{\mathbf{q}}
\newcommand{\Fields}{\mathbf{Fields}}

\newcommand{\Pic}{\operatorname{Pic}}
\newcommand{\disc}{\operatorname{disc}}
\newcommand{\CH}{\operatorname{CH}}
\newcommand{\spl}{\operatorname{spl}}
\newcommand{\Rings}{\mathbf{Rings}}
\newcommand{\Sets}{\mathbf{Sets}}

\newcommand{\U}{\mathcal U}

\newcommand{\cO}{\mathcal{O}}

\renewcommand{\geq}{\geqslant}
\renewcommand{\ge}{\geqslant}
\renewcommand{\leq}{\leqslant}
\renewcommand{\le}{\leqslant}

\newcommand{\Mbar}{\overline{\mathcal{M}}}

\newcommand{\Gal}{\operatorname{Gal}}

\newcommand{\Wel}{\operatorname{Wel}}

\newcommand{\Aut}{\operatorname{Aut}}

\newcommand{\Hom}{\operatorname{Hom}}

\newcommand{\Tr}{\operatorname{Tr}}
\newcommand{\tr}{\operatorname{tr}}

\newcommand{\GW}{\operatorname{GW}}

\newcommand{\lra}[1]{\langle #1 \rangle}
\newcommand{\Et}{\operatorname{Et}}
\newcommand{\Sym}{\operatorname{Sym}}
\newcommand{\Bl}{\operatorname{Bl}}
\newcommand{\Sq}{\operatorname{Sq}}
\newcommand{\Inv}{\operatorname{Inv}}
\newcommand{\UInv}{\operatorname{\beta Inv}}

\newcommand{\sgn}{\operatorname{sgn}}
\newcommand{\Nm}{\operatorname{nm}}
\newcommand{\nm}{\operatorname{nm}}
\newcommand{\Spec}{\operatorname{Spec}}
\newcommand{\Mat}{\operatorname{Mat}}
\newcommand{\PGL}{\operatorname{PGL}}
\newcommand{\WG}{\operatorname{\widehat{W}}}
\newcommand{\sWG}{\widehat{\mathcal{W}}}
\newcommand{\W}{\operatorname{W}}
\newcommand{\QF}{\operatorname{\widetilde{W}}}
\newcommand{\rk}{\operatorname{rk}}
\newcommand{\ev}{\operatorname{ev}}
\newcommand{\odp}{\operatorname{odp}}
\newcommand{\ve}{\operatorname{Ve}}
\newcommand{\ed}{\operatorname{Ed}}
\renewcommand{\div}{\operatorname{div}}
\NewDocumentCommand{\V}{}{V\!\!W}

\newcommand{\mbinom}[2]
{\begin{Bmatrix}
  #1 \\ #2
\end{Bmatrix}}

\begin{document}

\title{Welschinger--Witt invariants}

\author{Erwan Brugallé}
\address{Erwan Brugallé, Nantes Université, Laboratoire de
 Mathématiques Jean Leray, 2 rue de la Houssinière, F-44322 Nantes Cedex 3,
France}
\email{erwan.brugalle@math.cnrs.fr}

\author{Johannes Rau}
\address{Johannes Rau, 
Universidad de los Andes,
Departamento de Matemáticas, 
Carrera 1 No.\ 18A - 12,
Bogotá, Colombia}
\email{j.rau@uniandes.edu.co}

\author{Kirsten Wickelgren}
\address{Kirsten Wickelgren, Department of Mathematics, Duke University, 120 Science Drive
Room 117 Physics, Box 90320, Durham, NC 27708-0320, USA}
\email{kirsten.wickelgren@duke.edu}

\subjclass{Primary 14N35, 14P99,11E04; Secondary 53D45, 14N10, 14F42.}
\keywords{Weslchinger invariants, Quadratic Gromov--Witten invariants, Witt invariants, $\mathbb{A}^1$-homotopy theory}

\begin{abstract}
Welschinger invariants are signed counts of real rational curves satisfying contraints. Quadratic Gromov--Witten invariants give such counts over general fields of characteristic different from 2 and 3. For rational del Pezzo surfaces over a field, we propose a conjectural relationship between Welschinger and quadratic Gromov--Witten invariants. We construct multivariable unramified Witt invariants, in the sense of Serre, from Welschinger invariants and call them Welschinger--Witt invariants. We show that quadratic Gromov--Witten invariants are also Witt invariants and control their ramification. We then conjecture an equality between these Witt invariants, in particular giving a conjectural computation of all the quadratic Gromov--Witten invariants of $k$-rational  surfaces. We prove this conjecture for $k$-rational del Pezzo surfaces of degree at least 6. 
\end{abstract}

\maketitle
\tableofcontents

\section{Introduction}
We propose a relationship between enumerations of real rational curves in real rational surfaces and analogous enumerations over fields of characteristic different from 2 and 3. Welschinger introduced a signed count of the real rational ($J$-holomorphic) curves realizing a fixed class $D\in\Pic(X)(\R)$  in a real rational surface $X$, and interpolating a real configuration of points. 
More precisely, fixing the number  $s$ of pairs of complex conjugate points
in such a configuration, summing all  real rational curves of class $D$
with signs $+1$ and $-1$ depending on the number of solitary nodes of the curve
produces an integer $\Wel_{X}(D;s)$, independent of the choice of points \cite{Welschinger-invtsReal4mflds,Wel15,Bru18}.
These integers are a real analogue of genus 0 Gromov--Witten invariants called \emph{Welschinger invariants}. 

Given a field $k$ of characteristic not 2, we denote by $\WG(k)$ its \emph{Witt--Grothendieck ring}, that is, the ring completion of isomorphism classes of symmetric non-degenerate $k$-bilinear forms. We further denote by $\W(k)$ the \emph{Witt ring} of $k$ which is the quotient of $\WG(k)$ by the ideal generated by the hyperbolic form. See  \cref{sec:witt}.
Replacing the integers $\Z$ by $\WG(k)$  allows one to define an invariant count of rational curves with fixed class $D\in\Pic(X)(k)$ and through point constraints on a del Pezzo surface $X$ over any perfect field $k$ of characteristic not $2$ or $3$, under some technical restrictions \cite{degree}. For any 
finite étale $k$-algebra $A$ of the appropriate degree, this produces an arithmetically meaningful count $\widehat{Q}_{X,D,k}(A) \in \WG(k)$ of rational curves on $X$ through generic points $p_1,\ldots,p_r$ with $A = \prod_i k(p_i)$. Here, the algebra $A$ takes on the role of the value $s$ in the real case insofar as it specifies the field extensions over which the points in the configuration
live.
The count $\widehat{Q}_{X,D,k}(A) $ is again independent of the generally chosen points and is called a \emph{quadratic Gromov--Witten invariant}. Considering $\widehat{Q}_{X,D,k}(A) $ modulo the hyperbolic form gives an element $Q_{X,D,k}(A) \in\W(k)$.

Welschinger invariants of a $\R$-rational del Pezzo surface $X$  can be seen as particular instances of quadratic Gromov--Witten  invariants, since by construction  $\Wel_{X}(D;s)$ equals the signature of $Q_{X,D,\R}(\C^s\times \R^r)$, see \cite[Theorem 3]{degree}, \cite[Remark 2.5]{Levine-Welschinger}, or \cref{exRealQuadratic}.
This paper goes in the opposite direction by giving a conjectural expression of  quadratic Gromov--Witten invariants of $k$-rational surfaces in terms of Welschinger invariants, and proving it for $k$-rational del Pezzo surfaces of degree at least 6

In doing so, we give structure to both Welschinger and quadratic Gromov--Witten invariants. Enumerative problems often come in families, for example one may vary the parameter $s$ in Welschinger invariants. Beyond computing each individual number,
a major challenge is typically to understand the whole family of numbers 
in terms of structures imposed by the underlying geometry. For example
 \emph{Welschinger's formula}, see \cref{rem:wel formula} or \cite[Theorem 0.4]{Welschinger-invtsReal4mflds}, 
 describes how the invariant $\Wel_{\P^2_\R}(d;s)$ changes when $s$ is increased to $s+1$. We show that this implies that the collection of numbers $\Wel_{\P^2_\R}(d;s)$, for fixed $d$ but varying $s$, 
can be arranged to yield a \emph{Witt invariant} of étale algebras in the sense of Serre \cite{Garibaldi-Serre-Merkurjev}. 

Before discussing geometry, let us briefly review the concept of Witt invariants. We refer to \cref{sec:wittet} for more details.  View $\W$ as a functor $ \Fields/k \to \Sets$ from
the category $\Fields/k$ of field extensions $k \to K$ to the category of sets. Associating to $K$ the set $\Et_{n}(K)$ of isomorphism classes of étale $K$-algebras of degree $n$ defines a second functor. Witt invariants of étale algebras are morphisms of functors $\alpha:\Et_{n,k}\to \W$. In other words, a Witt invariant $\alpha$ consists of maps $\alpha_K \colon \Et_n(K) \to \W(K)$ for all $K \in \Fields/k$
such that for any field extension $K \to L$ the diagram
\[\begin{tikzcd}
	{\Et_n(K)} & {\W(K)} \\
	{\Et_n(L)} & {\W(L)}
	\arrow["{\alpha_K}", from=1-1, to=1-2]
	\arrow["{\otimes_K L}"', from=1-1, to=2-1]
	\arrow["{\otimes_K L}", from=1-2, to=2-2]
	\arrow["{\alpha_L}"', from=2-1, to=2-2]
\end{tikzcd}\]
commutes. Despite its apparent simplicity, this definition imposes strong constraints on the ring $\Inv_k(n)$ of such Witt invariants. 
 Serre  shows that  $\Inv_k(n)$ is a free $\W(k)$ module of rank $m+1$, with $m=\lfloor \frac{n}{2}\rfloor$, and that Witt invariants are determined by their values on the so-called multiquadratic algebras, meaning those finite \'etale $K$-algebras of the form 
\[
\E_{\delta_1} \times \E_{\delta_2} \times \ldots \times \E_{\delta_m} (\times K)
\] where 
$\E_\delta = K[x]/(x^2-\delta)$. (The parenthetical $K$ appears exactly when $n$ is odd.) There is a basis $(\beta_0,\ldots,\beta_m)$ of  $\Inv_k(n)$ whose values on multiquadratic algebras are given by the elementary symmetric polynomials applied to the  $m$-tuple of the trace forms of the $\E_{\delta_i}$, see \cref{thmBetaBasis1var}. 

Thus, fixing $d$ and using Welschinger's formula, one may encode the Welschinger invariants $\Wel_{\P^2_\R}(d;s)$
in a Witt invariant which we call \emph{Welschinger--Witt invariant}
and denote by  $\V_{d}$.
In fact, the quite simple following \enquote{triangle} recipe produces this Witt invariant.  List the integers $\Wel_{\P^2_\R}(d;s)$ for $s=m,m-1,\ldots,0$ in the top row of a triangle. For example the list $\Wel_{\P^2_\R}(4;s)$ is $ 0, 16, 40, 80, 144,240$. To form the next row of the triangle, subtract the element above from the element above and to the right and then divide by $2$. So for $d=4$, this yields 

	\[
    \begin{matrix}
		  0 & 16 & 40 & 80 & 144 & 240 \\
			8 & 12 & 20 & 32 &  48 &     \\
			2 &  4 &  6 &  8 &     &     \\
			1 &  1 &  1 &    &     &     \\
			0 &  0 &    &    &     &     \\
			0 &    &    &    &     &     
		\end{matrix}
	\] 
	The formula for $\V_d$ is then given by using the left column of the triangle as the coefficients for $\beta_0,\ldots,\beta_m$ reading from top to bottom. So, for $d=4$ we get
	\[
	\V_4 = 8\beta_1 + 2 \beta_2 + \beta_3.
	\] 
	Although an arbitrary Witt invariant in $\Inv_k(n)$ is of the form $\sum_{i=0}^m b_i \beta_i$ for $b_i$ in the Witt group $\W(k)$, it is remarkable that the coefficients of $\V_d$ in the $\beta$-basis
	are integers (and even Welschinger invariants themselves, but this time
	of a blow-up of $\P^2_\R$ at real points). We call such a Witt invariant {\em $\beta$-integral}. 
		Note that a $\beta$-integral Witt invariant  is determined by its restriction to $\Inv_\R(n)$ and defines a Witt invariant over any field of characteristic not $2$. In particular it is an element of the algebra $\Inv(n)$ of Witt invariants of $\Et_n, \W :\Fields\to \Sets$ where $\Fields$ is the category of \emph{all} fields  of characteristic not $2$.
		 The defining property of the Witt invariants $\V_d$ is that for fixed $d$, the Witt invariant $\V_d$ is the unique $\beta$-integral Witt invariant such that for $K= \R$ and for any $s=0,\ldots,\left\lfloor\frac{n}{2}\right\rfloor$,
\begin{equation}\label{eqn:Vd-defining}
\V_{d, \R}(\C^s\times \R^{n-2s})= \Wel_{\P^2_\R}(d;s)\in \W(\R)\cong \Z.
\end{equation} Here, 
the isomorphism $\W(\R) \cong \Z$ is given by 
the signature of quadratic forms over $\R$.
Note that this gives $m+1$ equations and $m+1$ unknowns in $\Z$ because the $b_i$ are integers by $\beta$-integrality. The corresponding matrix is invertible over $\Q$ but not over $\Z$  (\cref{lemMultirealMatrix dim1}), but we show that this system of equations has a (unique) integer solutions thanks to Welschinger's formula.
	
\medskip	
	We furthermore show that the quadratic Gromov--Witten invariants can also be given the structure of a Witt--invariant (\cref{thmQuadraticWitt}). Note that this require first  to extend the definition to quadratic Gromov--Witten invariants to non-perfect fields. We also control their ramification 
	away from $2$ and $3$ in terms of the primes of bad reduction of the surface.  See \cref{thm:QXDunramifiedWittawayS}. (Recall that the exclusion of $2$ and $3$ is due to the construction of quadratic Gromov--Witten invariants.) In particular, since $\P^2$ is smooth and proper over $\Z$, 
	the Witt invariants $Q_{\P^2, d}$ is unramified
	away from characteristic $2$ and $3$.	
	 Combining this result with the computation of $Q_{\P^2,d,k}$ for multiquadratic algebras from  \cite{PMPR-QuadraticallyEnrichedPlane}, we deduce that 
$\V_d$ 
and $Q_{\P^2,d}$ agree for fields of characteristic different from 2 and 3.
	
	\begin{theorem}\label{thm:intro:P2}
		For $K$ a field of characteristic not $2$ or $3$,  and $A\in\Et_{3d-1}(K)$, we have
	\[
	Q_{\P^2,d,K}(A) =\V_{d,K}(A).
	\]
	\end{theorem}
Since both Witt invariants have the same restriction to $\Inv_\R(3d-1)$, to prove the theorem it is enough to show that $Q_{\P^2,d}$ is $\beta$-integral.
The property of $\beta$-integrality is connected to ramification in the sense of \cref{def:unramified_inv}. Indeed, we show that for a finite set of primes $S$ containing $2$, a Witt invariant is unramified away from $S$ if and only if it is of the form $\sum_{i=0}^m b_i \beta_i$ with $b_i \in \W(\Z[S^{-1}])$. See \cref{thmUnramifiedQuasiIntegral}.  
This
already implies that the  coefficients of  $\Q_{\P^2,d}$ with respect
to the $\beta$-basis 
lie in $\Z[\lra{2}, \lra{3}] \subset \W(K)$. 
That is, these coefficients are $\Z$-linear combination of $\lra{1}, \lra{2},\lra{3}$, and $\lra{6}$, where $\lra{a}$ denotes the quadratic form 
$x\mapsto ax^2$. See \cref{thmUnramifiedQuasiIntegral} and \cref{thm:QXDunramifiedWittawayS}.
We  further prove that these coefficients eventually lie in $\Z$ using the floor diagram computation from  \cite{PMPR-QuadraticallyEnrichedPlane}.
It is then of course very tempting
to ask whether the ramification considerations
can be extended in the future to cover characteristic
$2$ and $3$. 

\medskip
While so far we restricted our exposition to $\P^2$ for simplicity,
 the full significance of our considerations only seems to 
become apparent in a more general setting.  Indeed, if quadratic Gromov--Witten invariants of a given algebraic del Pezzo surface $X$ provide Witt invariants by \cref{thmQuadraticWitt}, it is clear from the first non-trivial computations that these are not $\beta$-integral in general.  See  
\cite[Example 1.4]{degree} or \cref{exAnticanonical}. However we conjecture that one still gets a $\beta$-integral Witt invariant when varying not only the étale algebras of the interpolated points, but also the deformation class of $X$ in a suitable space, and that this Witt invariant is the restriction of a corresponding Welschinger--Witt invariant. We now make this precise in the case of 
rational surfaces which is the main setting of this paper. Recall that a $k$-rational surface over a field $k$ is either a blow-up of $\P^2_k$ or a quadric surface in $\P^3_k$ containing a $k$-rational point. For simplicity, we restrict to blow-ups of $\P^2_k$  in the rest of this introduction. 
Recall that 
\begin{itemize}
	\item 
	  Welschinger invariants are defined for any projective  real non-singular algebraic surface 
		and satisfy a generalization of Welschinger's formula called the
		real Abramovich--Bertram formula, see 
		 \cite[Proposition 2.3]{Bru18} and \ref{secWWelschinger},
	\item 
	  quadratic Gromov--Witten invariants are defined for any $\A^1$-connected del Pezzo surface of degree at least 3 
		and satisfy a quadratic version of the Abramovich--Bertram formula, 
		see \cite{Brugalle-WickelgrenABQ} and \cref{sec:quadGW}.
\end{itemize}
The main idea is to extend both Welschinger--Witt and quadratic Gromov--Witten invariants
to \emph{multivariable} Witt invariants, where the additional
variables control the $K$-structure 
of the blown-up points on $\P^2_k$. 
Here, a multivariable Witt invariant $\alpha \in \Inv(n_0, \dots, n_r)$ is a natural transformation 
\[
\Et_{n_0} \times \ldots \times \Et_{n_r} \to \W.
\] In other words, the invariant $\alpha$ is given by maps
\[
  \alpha_K \colon \Et_{n_0}(K) \times \dots \times \Et_{n_r}(K) \to \W(K)
\] 
for all $K \in \Fields$ satisfying a base change formula so that $\alpha$ is a morphism of functors.

\medskip
Let us explain our strategy on the simplest example after $\P^2$,
 starting with the Welschinger--Witt side. 
We fix $n_1 \in \N$ and 
denote by $X_{n_1,s_1}$ a blow-up of $\P^2_\R$ in a real configuration of $n_1$ points
containing exactly $s_1$ pairs of  $\Gal(\C:\R)$-conjugated points.
Then the real Abramovich--Bertram formula relates the Welschinger invariants
of $X_{n_1,s_1+1}$ and $X_{n_1, s_1}$. So,
in particular, 
it relates the invariants for two different real structures
on the same (deformation class of) symplectic surface. 
From the presentation of $X_{n_1,s_1}$ as a blow-up of $\P^2_\R$, the Picard group $\Pic(X_{n_1,s_1}) \cong \Z^{n_1+1}$ comes 
with a the canonical basis, independent of $s_1$, 
given by the pull-back of a line class $L$
and the exceptional divisors $E_1, \dots, E_{n_1}$.
We fix $\dbar = (d_0, d_1) \in \N^2$ and consider
the divisor $D= d_0 L - d_1(E_1 + \dots +  E_{n_1})$.
Note that $D$ lies in the subgroup of real divisor classes $\Pic(X_{n_1,s_1})(\R)$
no matter which $s_1$ we used. 
Generalizing the case of $\P^2$, fixing $n_1$ and $D$ but varying $(s_0,s_1)$, 
we arrange the 
Welschinger invariants $\Wel_{X_{n_1,s_1}}(D;s_0)$ thanks to  
the real Abramovich--Bertram formula 
 to yield a
two-variable Witt invariant
$\V_D$
in $\Inv(n_0, n_1)$, which is $\beta$-integral
with respect to the multivariable $\beta$-basis. Here  $n_0  = 3d_0 - n_1 d_1 -1$,
which we assume to be non-negative.

On the the quadratic Gromov--Witten invariants side,
choose $K$ a perfect fiel of characteristic not $2$ or $3$
and $n_1 \leq 6$.
Given $A_1 \in \Et_{n_1}(K)$ we denote 
by 
$X_{A_1}$
a blow-up of $\P^2_K$ along 
a generic $K$-configuration $\PP$ of points of length $n_1$ such that
$\PP = \Spec(A_1)$. 
Note that $D \in \Pic(X_{A_1})(K) \subset \Pic(X_{A_1}) \cong \Z^{n_1+1}$
and we can therefore consider the quadratic Gromov--Witten invariant
$Q_{X_{A_1}, D,K}(A_0)$.  We conjecture that 
 for all $(A_0,A_1)\in\Et_{n_0}(K)\times \Et_{n_1}(K)$
\[
  Q_{X_{A_1},D,K}(A_0)=\V_D(A_0,A_1).
\]
Here $A_0 \in \Et_{n_0}(K)$ controls the 
$K$-structure of the interpolated points while $A_1$ controls
the $K$-structure of the blown-up points.

\medskip
We just described the case where the divisor $D$ was chosen to be completely symmetric in the exceptional divisors. 
More generally, we take into account the symmetries allowed
by the chosen divisor $D$. We encode $D$ and  its symmetries by vectors
$\n\in \N^r$ and $\dbar \in \N^{r+1}$.
Here, we blow-up $|\n| = n_1 + \dots + n_r$ points in $\P^2$ and
\[
D = d_0 L - d_1(E_1 + \ldots + E_{n_1}) - d_2 (E_{n_1+1} + \ldots E_{n_1+n_2} ) - \ldots - d_r (E_{n_1 + \ldots + n_{r-1}+1} + \ldots + E_{n_1 + \ldots + n_r}).
\]
We define $n_0=3d_0-n_1d_1-\cdots - n_rd_r-1$ that we assume to be non-negative.
In this case, we show (\cref{thmWelschingerUniversal}) that there is a unique $\beta$-integral invariant $\V_D \in \Inv(n_0, \n)$
satisfying
\[
\V_{D,\R}(\C^{s_0} \times \R^{n_0 - 2 s_0}, \ldots, \C^{s_r} \times \R^{n_r - 2 s_r}) = \Wel_{X_{|\n|,|\s|}}(D;s_0)\in \W(\R)\cong \Z 
\]
for all $s\in\N^{r+1}$ such that $0\le n_i-2s_i$ for all $i$, where $|\n|=n_1+\cdots+n_r$ and  $|\s|=s_1+\cdots+s_r$.
We conjecture the following. 
\begin{conjecture}\label{conj:intro}
Let $K$ be a perfect field of 
	characteristic not $2$ or $3$. 
	Fix $\n \in \N^r$ such that $|\n|\le 6$ and $(|\n|,n_0)\ne (6,5)$, 
	and $(A_1, \dots, A_r) \in \Et_\n(K)$.
	Let 
	$X$ be a rational del Pezzo surface (of degree at least $3$)
	constructed 
	as the blow up of $\P^2_K$
	along the zero-dimensional subschemes $\p_1, \dots, \p_r \subset \P^2_K$
	such that $\p_i = \Spec A_i$. 
	Then
\[
  \V_D(A_0,A_1, \dots, A_r)=Q_{X,D,K}(A_0).
\] 
\end{conjecture}

See \cref{conjGeneral} in \cref{sec:quadGW}, which is a little more general. The limitations on the vector $(n_0,\n)$ in \cref{conj:intro} is again due to the construction of quadratic Gromov--Witten invariants.
Since the left hand side is defined
for any $\n$, it would be interesting to understand 
whether these values can also be described as quadratic invariants
(that is, as $\W(K)$-valued counts of rational curves or degrees of evaluation maps)
in the future.

We are able to prove Conjecture~\ref{conjGeneral} for del Pezzo surfaces of degree at least $6$, see \cref{thm:dP6}. In the special case discussed in this introduction, this reads as follows.

\begin{theorem}
	Conjecture~\ref{conj:intro} holds if $|\n|\le 3$.
\end{theorem}

 The proof proceeds in the following steps. 
As in the case of $\P^2$, we first show that  quadratic Gromov--Witten invariants for different choices of $A_0,\ldots,A_r$ yield (partial) unramified Witt invariants over fields of characteristic different from $2$ and $3$.
Since these Witt invariants take the same values over $\R$ than their corresponding Welschinger--Witt invariant, we reduce the conjecture to showing
that $Q_{X,D,k}(A_0)$ is the restriction of a $\beta$-integral Witt invariant. 
We prove $\beta$-integrality in two more steps. 
We first reduce to the case of toric surfaces using the quadratic Abramovich--Bertram formula from \cite{Brugalle-WickelgrenABQ}. We then use the 
 the floor diagram computation 
from \cite{PMPR-QuadraticallyEnrichedPlane} to conclude.

\medskip
Conjecture~\ref{conjGeneral} predicts that quadratic Gromov--Witten invariants of rational del Pezzo surfaces are determined by  Galois,  Witt, Gromov--Witten, and Welschinger theories in a precise manner. We believe that this opens to many exciting developments, that may be of interest both from the real and the $\A^1$-homotopy points of view.

\subsection*{Plan of the paper}
In Section \ref{secWittInvariants} we recall the notion of Witt invariants and extend it to multivariate Witt invariants. We also discuss the notions of $\beta$-integral and
unramified Witt invariants. 
We turn to the calculus of multireal values of $\beta$-integral Witt invariants (and their relation to torsion) in Section \ref{sec:multireal}.
This framework is then used in \cref{secWWelschinger} to construct Witt invariants out of Welschinger invariants which we illustrate with several examples. 
Section \ref{sec:quadGW} is devoted to quadratic Gromov--Witten invariants. After reminding their definition, we prove that they are unramified Witt invariants away from characteristic 2 and 3 for surfaces defined over $\Z$
and discuss the main conjecture \ref{conjGeneral}.
We continue to prove the conjecture for del Pezzo surfaces of degree at least 6 in Section \ref{sec:dP6}. 
We end the paper with two appendices.
We provide the explicit change of basis between the $\beta$-basis and Serre's $\lambda$-basis, for which computations may be easier,  in Appendix \ref{sec:basechange}. 
In Appendix \ref{WittInvarianceQuadratic}, we provide an alternative proof for the Witt-invariance of quadratic Gromov--Witten invariants using their enumerative description.

\subsection*{Acknowledgment} 
We warmly thank Benoît $\V$ Bertrand, Ilia Itenberg, Viatcheslav Kharlamov and Lionel Lang for their encouragement on an earlier stage of this project. Part of this work
has been achieved during a visit of JR, followed by a joint visit of JR and KW at Nantes Université.  EB is particularly grateful to JR and KW for making this possible. These visits have been funded by  the program {\it Missions Chercheurs Invités} of Nantes Université, the Fédération Henri Lebesgue (FR CNRS 2962), project Ambition Lebesgue Loire (Pays de Loire), and project  ANR-22-CE40-0014. We thank Laboratoire
de Mathématiques Jean Leray for excellent working conditions. 

EB was partially supported by l’Agence Nationale de la Recherche (ANR),
project ANR-22-CE40-0014. KW was partially supported by National Science Foundation Awards
DMS-2103838 and DMS-2405191. KW thanks Duke University for supporting research travel with
family. JR thanks Universidad de los Andes for granting a sabbatical semester (STAI) during 01--06/2025 which made
his participation in this project possible.
For the purpose of open access, the authors apply a CC-BY public copyright licence to any
Author Accepted Manuscript (AAM) version arising from this submission.

\section{Witt invariants and $\beta$-invariants} \label{secWittInvariants}

\subsection{Witt ring}\label{sec:witt}

Let $R$ be a Noetherian ring where $2$ is invertible. A unimodular quadratic form over $R$, or a unimodular $R$-quadratic form, is a projective $R$-module $P$ together with a map $q: P \to R$ such that
\begin{itemize}
\item $q(r x) = r^2 q(x) $ for all $r$ in $R$ and $x$ in $P$.
\item $B_q(x,y) = q(x+y) - q(x) -q(y)$ is bilinear.
\item The map $P \to \Hom(P,R)$ induced by the symmetric bilinear form $B_q$ is an isomorphism.
\end{itemize}
If $P$ is free, the last condition is equivalent to the determinant
of the Gram matrix of $B_q$ for some $R$-basis being a unit in $R$.
The {\em rank} of $q$ is the rank of the projective module $P$, which is equal to the dimension of the fiber $P \otimes R/\mathfrak{p}$ over $R/\mathfrak{p}$ for any prime ideal $\mathfrak{p}$, which is constant over the connected components of $\Spec R$. Given a second unimodular quadratic form $q' \colon P' \to R$, we can 
form the orthogonal sum $q\oplus q' \colon P \oplus P' \to R$
given by $q \oplus q'((x,x')) = q(x) + q'(x')$ as well as the tensor product
$q \otimes q' \colon P \otimes P' \to R$ given by 
$q \otimes q'(x \otimes x') = q(x) q(x')$, both of which are unimodular quadratic forms. 

Let $\QF(R)$ denote the set of isomorphism classes of unimodular quadratic forms over $R$. Since there are no additive inverses (except for the $0$ form),  the triple $(\QF(R),\oplus,\otimes)$
	is only a semiring. The associated Grothendieck group
	obtained by adding formal inverses is denoted by $\WG(R)$ and called
	the \emph{Witt-Grothendieck ring} of $R$.

A unimodular quadratic form $(P,q)$ is {\em metabolic} if there is a direct summand $U \subseteq P$ such that $U^{\perp} =U$. The {\em Witt ring} $\W(R)$ of $R$ is the quotient of $\WG(R)$ by the ideal generated by metabolic forms. Further references on Witt and Witt-Grothendieck rings include \cite{knus} \cite{lam05} \cite{milnor73}.

Given $a_1,\ldots,a_s\in R^*$, we denote by $\lra{a_1,\ldots,a_s}$ the diagonal quadratic form on $R^s$ defined by the $a_i$'s. We write $h$ for  the hyperbolic form $\lra{1,-1}$. We have
\begin{align*}
 	\lra{a_1,\ldots,a_s} \oplus \lra{b_1,\ldots,b_t}&=\lra{a_1,\ldots,a_s,b_1,\ldots,b_t},&
	\lra{a_1,\ldots,a_s}\otimes\lra{b_1,\ldots,b_t}&=
	\sum_{i,j = 1}^{s,t}
	\lra{a_i b_j}.
\end{align*} We note that the additive inverse of $[q]$ in $\W(R)$ is $[-q]$
since $q \oplus -q$ is isomorphic to  $n h$ where $n$ is the rank of $q$.
We will often denote the equivalence class of a quadratic
form $q$ by the same letter $q$ and write sum and product
as $q + q'$ and $qq'$ if no confusion is likely. 

We are particularly interested in the case where $R$ is a field. 
In this case $P$ is a $R$-vector space whose dimension is the rank of a quadratic form $q:P\to R$.
Denote by $\Fields$ the category of fields of characteristic different from $2$. More generally, given a finite set  $S$  of prime numbers containing $2$, 
let $\Fields_S \subset \Fields$ denote the full subcategory of fields whose characteristic is either $0$ or is positive but does not lie in $S$. For $k \in \Fields$, we denote by $\Fields/k$ the initial category over $k$, that is, 
the category of field extensions $k \to K$.

\begin{exa}
	One has
	\[
	\W(\C)\simeq\Z/2\Z \qquad\mbox{and}\qquad \W(\R)\simeq\Z,
	\]
	the first isomorphism being the reduction modulo 2 of the rank, and the second being the signature. 
	More generally, the Witt ring of any algebraically closed field is isomorphic to $\Z/2\Z$
	and the Witt ring of any real closed  field is isomorphic to $\Z$ by \cite[Proposition II.3.2 (1)]{lam05}.
	Also, the Witt ring of a finite field $K\in \Fields$ is given by
\[
\W(K)=\left\{\begin{array}{ll}
	(\Z/2\Z)[X]/(X^2+1) & \mbox{if }|K|=1\mod 4
	\\ \Z/4\Z  & \mbox{if }|K|=3\mod 4
\end{array}\right.,
\]
see \cite[Corollary II.3.6]{lam05}.
\end{exa}

\begin{remark} \label{remWGvsW} \label{lemIsomorphicSameClasses}
Since for $K \in \Fields$ the cancellation property
	holds in $\QF(K)$ by Witt's cancellation theorem \cite[Theorem I.4.2]{lam05}, 
	the map $\QF(K) \to \WG(K)$ is injective. 
  The map $(\rk, [\_]) \colon \WG(K) \to \Z \times \W(K)$ is injective. Indeed, two non-degenerate quadratic forms over $k\in \Fields$ are isomorphic if and only if they have the same rank and the same class in $\W(k)$.
	See \cite[Proposition II.1.4]{lam05}. 
	Moreover, the diagram
	\[\begin{tikzcd}
		{\WG(K)} & \Z \\
		{\W(K)} & {\Z/2\Z}
		\arrow["\rk", from=1-1, to=1-2]
		\arrow[from=1-1, to=2-1]
		\arrow[from=1-2, to=2-2]
		\arrow[from=2-1, to=2-2]
	\end{tikzcd}\]
	is cartesion, see \cite[Section VIII.27]{Garibaldi-Serre-Merkurjev}. In particular any metabolic form over a field is an integral multiple of the hyperbolic form $h$, that is, the ring  $\W(K)$ is the quotient of $\WG(K)$ by the ideal generated by $h$.

\end{remark}

Given a field extension $\varphi \colon K \to L$,   extension of scalars  provides a  map $\W(\varphi) = \otimes_K L \colon \W(K) \to \W(L)$. This yields a functor
$\W \colon \Fields \to \Sets$ (which factors through the category $\Rings$ of commutative rings
with unit).
By restriction, we also have functors $\W_S \colon \Fields_S \to \Sets$ and $\W_{k} \colon \Fields/k \to \Sets$ 
which by abuse of notation we also denote by $\W$ if no confusion is likely.

\subsection{Witt invariants of étale algebras}\label{sec:wittet}

Let $F \colon \Fields \to \Sets$ another functor. 
Following \cite[27.3]{Garibaldi-Serre-Merkurjev}, 
a \emph{Witt invariant of type $F$} is a morphism of functors $\alpha \colon F \to \W$. 
In other words, $\alpha$ consists of maps $\alpha_K \colon F(K) \to \W(K)$ for all $K \in \Fields$
such that for any field extension $\varphi \colon K \to L$ the diagram
\[\begin{tikzcd}
	{F(K)} & {\W(K)} \\
	{F(L)} & {\W(L)}
	\arrow["{\alpha_K}", from=1-1, to=1-2]
	\arrow["{F(\varphi)}"', from=1-1, to=2-1]
	\arrow["{\otimes_K L}", from=1-2, to=2-2]
	\arrow["{\alpha_L}"', from=2-1, to=2-2]
\end{tikzcd}\]
commutes. We denote the set of Witt invariants of type $F$ by $\Inv(F)$. 
Since $\W$ factors through $\Rings$, the set $\Inv(F)$ is carries a ring structure
by ordinary addition and multiplication of functions.
Similarly, 
given a functor $F \colon \Fields_S \to \Sets$ or $F \colon \Fields/k \to \Sets$, 
a morphism of functors from $F$ to $\W_S$ or $\W_k$ is 
called a Witt invariant \emph{away from $S$} or \emph{over $k$}, respectively, 
and the ring of such invariants is denoted by $\Inv_S(F)$ and $\Inv_{k}(F)$.
Note that $\Inv_{k}(F)$  is a $\W(k)$-algebra.
For a set of fields $\mathcal{K} \subset \Fields$, we set 
\[
\W(\mathcal{K}) := \prod_{k \in \mathcal{K}} \W(k).
\]
We denote by $\PP$ the set of all prime fields of characteristic different from $2$
and write $\PP \setminus S$ for the subset of prime fields whose
characteristic is not contained in $S$. 
Then clearly $\Inv_S(F) = \prod_{k \in \PP \setminus S} \Inv_{k}(F)$ and
$\Inv_S(F)$ is a $\W(\PP\setminus S)$-algebra.

We are mostly interested in the following choices for $F$. 
Given $n \in \N$, we denote by $\Et_n(K)$ the 
set of isomorphism classes of étale $K$-algebras $A$ of rank $n$. 
By definition any such algebra is isomorphic to the product $L_1 \times \dots \times L_\ell$
for finite and separable field extensions $L_1, \dots, L_\ell$ of $K$ such that
$[L_1:K] + \dots + [L_\ell:K] = n$. This decomposition is unique 
up to permutation of the factors (and isomorphisms of field extensions). 
Given an extension $\varphi \colon K \to L$, 
the extension of scalars $A \otimes_K L$ is an étale $L$-algebra of rank $n$
and this gives rise to a map $\Et_n(\varphi) = \otimes_K L \colon \Et_n(K) \to \Et_n(L)$.
Therefore, $\Et_n$ is a functor from $\Fields$ to $\Sets$ 
(and we denote by $\Et_{n,S}$ and $\Et_{n,k}$ its restriction
to $\Fields_S$ and $\Fields/k$, respectively, if necessary). 
The Witt invariants of type $\Et_{n}$ (respectively $\Et_{n,S}$, $\Et_{n,k}$) are called \emph{(étale) Witt invariants of degree $n$}
(away from $S$ and over $k$, respectively)
and form a ring denoted by $\Inv(n)$ (respectively $\Inv_S(n)$, $\Inv_k(n)$). 

We extend the definitions to the multivariable case: 
given $\n = (n_0, \dots, n_r) \in \N^{r+1}$, we set $\Et_\n = \Et_{n_0} \times \dots \times \Et_{n_r}$;
Witt invariants of type $\Et_\n$ are called  \emph{Witt invariants of (multi-)degree $\n$}, and  form a ring denoted by $\Inv(\n)$. 
We define $\Inv_S(\n)$ and $\Inv_k(\n)$ analogously. 

\begin{exa}\label{exa:trace}
	There is a particular Witt invariant of degree $n$, called the \emph{trace form}, from which all others can be derived (see
\cite{Garibaldi-Serre-Merkurjev}). 
Given $A \in \Et_n(K)$ and $x \in A$, 
we denote by $\tr_{A/K}(x)$ the trace
of the $K$-linear map $A \to A$, $y \mapsto xy$.
Since $\tr_{A/K} \colon A \to K$ is $K$-linear,
the map 
\[
\begin{array}{ccc}
A&\longrightarrow & K
\\ x&\longmapsto & 	\tr_{A / K}(x^2),
\end{array}
\]
is a quadratic form
over $K$ which we denote by $\Tr_K(A)$ or just $\Tr(A)$. 
Given a second étale $K$-algebra $B$, it is clear 
that $\Tr(A \times B) = \Tr(A) + \Tr(B)$.
It follows that the trace form $\Tr(A)$ is non-degenerate since it
is non-degenerate for any finite separable field extension,
see e.g.\ \cite[Exercise I.30]{lam05}.
As usual, we also denote by $\Tr(A) = \Tr_K(A)$ the associated class in $W(K)$.
Given a field extension $K \to L$ and $s \in L$, $x \in A$, we have
\begin{align*}
  \tr_{A \otimes L /L} (s \otimes x) &= s \tr_{A/L}(x) & \text{and hence} & &\Tr_L(A \otimes L) = \Tr_K(A) \otimes L.
\end{align*}
Therefore $\Tr \in \Inv(n)$.
\end{exa}

We continue by recalling some basic structure theorems
for $\Inv_k(n)$ from \cite{Garibaldi-Serre-Merkurjev}
and generalizing them to $\Inv_S(\n)$,
i.e.\ to the multivariable case with no fixed  base field. 
We consider the single variable case first. 

Let $K$ be a field of characteristic not $2$. Any algebra in $\Et_2(K)$ is of the form $\E_a = K[x]/(x^2-a)$
for some $a \in K^*$.
Note that $\Tr_K(\E_a) = \lra{2, 2a} \in W(K)$. Since 
 $\E_a \cong \E_b$ if and only if $a$ and $b$ define the same class in  $K^*/(K^*)^2$, there is 
a well-defined bijection $K^*/(K^*)^2 \to \Et_2(K)$ which by abuse of notation
we denote by $a \mapsto \E_a$.
Given $m \in \N$, we consider the $m$-th \emph{square classes} functor 
$\Sq_m \colon \Fields \to \Sets$ given by
$\Sq_m(K) = [K^*/(K^*)^2]^m$. 
For $2m \leq n$, we have that $\Sq_m$ is a subfunctor of $\Et_n$ via 
the natural transformation
$(a_1, \ldots, a_m) \mapsto \E_{a_1} \times \ldots \times \E_{a_m} \times K^{n-2m}$.
For $m = \lfloor \frac{n}{2} \rfloor$, 
the image of this map is the set of \emph{multiquadratic étale algebras of rank $n$}.
This yields a map $\Inv(n) \to \Inv(\Sq_m)$ which we denote by
$\alpha \mapsto \alpha|_{\Sq_m}$ and call \emph{restriction to multiquadratic algebras}. 

The basic structure theorem is the following one, which is merely a reformulation of results from \cite[Section 29]{Garibaldi-Serre-Merkurjev}.

\begin{theorem} \label{thmMultiquadraticFree}
  Given $k \in \Fields$, $S$ a set of primes and $n \in \N$, we denote $m = \lfloor \frac{n}{2} \rfloor$. 
	Then the following hold. 
	\begin{enumerate}
		\item 
		  The restriction maps to multiquadratic algebras $\Inv_k(n) \to \Inv_k(\Sq_m)$ and 
			$\Inv_S(n) \to \Inv_S(\Sq_m)$ are injective. 
		\item 
		  The ring $\Inv_k(n)$ is a free 
			$W(k)$-module of rank $m+1$. 
			The ring $\Inv_S(n)$ is a free $W(\PP\setminus S)$-module of rank $m+1$. 
	\end{enumerate}
\end{theorem}

\begin{proof}
	As mentioned above, clearly $\Fields_S = \bigsqcup_{k \in \PP\setminus S} \Fields/k$,
	and therefore 
	\[
	\Inv_S(n) = \prod_{k \in \PP \setminus S} \Inv_{k}(n).
	\] 
	It is hence sufficient to prove the statements for fixed $k \in \Fields$. 
  The map $\Inv_k(n) \to \Inv_k(\Sq_m)$ is injective
	since by 
	\cite[Theorem 29.1]{Garibaldi-Serre-Merkurjev} its kernel is trivial. 
	The fact that $\Inv_k(n)$ is free of rank $m+1$ is proven in 
	\cite[Theorem 29.2]{Garibaldi-Serre-Merkurjev}.
\end{proof}

In the following we 
present a few 
bases for the free module $\Inv(n)$
(which are thus universal in the sense that they
provide
bases for $\Inv_S(n)$ and $\Inv_k(n)$
for any choice of $S$ and $k$). 
In fact, it is often easier to describe their restrictions
to $\Inv(\Sq_m)$. 
We denote by $P_i^m$ the $i$-th symmetric polynomial on $m$ variables, that is, 
\[
  P_i^m(x_1, \dots, x_m) = \sum_{\substack{J\subset\{1,\ldots,m\}\\ |J|=i}} \prod_{j\in J} x_j.
\]
\begin{enumerate}
	\item 
	  Given a quadratic form $q \colon V \to K$ and $i \in \N$, one can define
		its $i$-th exterior power $\bigwedge^i q \colon \bigwedge^i V \to K$ such
		that on diagonal forms we have
		\[
			{\textstyle \bigwedge^i} \lra{a_1,\ldots,a_n}=P_i^n(\lra{a_1},\ldots,\lra{a_n}),
		\]
		see \cite[IX.1.2, Definition 12]{Bou-ElementsDeMathematique}.
		This gives rise to elements $\lambda_i^n \in \Inv(n)$ defined by 
		$\lambda_i^n(A) = \bigwedge^i \Tr_K( A) \in W(K)$. Its restriction to $\Sq_m$ is given by
		\[
			\lambda_i^n|_{\Sq_m} (a_1, \ldots, a_\ell) = 
		   P^n_i(\lra{2}, \ldots, \lra{2}, \lra{2a_1}, \ldots, \lra{2a_m}, (1)),
	   \]
	  where there are $m$ repetitions of $\lra{2}$ and the notation $(1)$ means that
	   there is an extra argument $ \lra{1}$ if $n$ is odd. 
	\item 
	  We define elements $\beta'^{m}_i\in \Inv(\Sq_m)$
		by setting 
		\[
		  \beta'^{m}_i(a_1, \ldots, a_m) = P^m_i(\Tr \E_{a_1}, \ldots, \Tr \E_{a_m})
			= P^m_i(\lra{2,2a_1}, \ldots, \lra{2,2a_m}).
		\]
	\item
	  We define elements $\alpha'^{m}_i  \in \Inv(\Sq_m)$
		by setting
		\[
		  \alpha'^{m}_i(a_1, \ldots, a_m) = P^m_i(\lra{a_1}, \ldots, \lra{a_m}).
		\]
\end{enumerate}

The symmetric group $\mathfrak S_m$ acts on $\Inv(\Sq_m)$
by permuting the $m$ factors of $K^*/(K^*)^2$, and
we denote by $\Inv(\Sq_m)^{\mathfrak S_m}$ the set 
of $\mathfrak S_m$-invariant elements. 

\begin{theorem} \label{thmBetaBasis1var}
  Fix $n \in \N$ and write $m = \lfloor \frac{n}{2} \rfloor$. 
	Then the following statements hold. 
	\begin{enumerate}
		\item 
		  The image of the restriction map 
			$\Inv(n) \to \Inv(\Sq_m)$ is $\Inv(\Sq_m)^{\mathfrak S_m}$. 
		\item
      The invariants $\lambda_0^n, \ldots, \lambda_m^n \in \Inv(n)$
			form a $W(\PP)$-basis for $\Inv(n)$. 
		\item
		  There exist unique invariants $\alpha_0^n, \ldots, \alpha_m^n$ 
			in $\Inv(n)$ such that $\alpha_i^n|_{\Sq_m} = \alpha'^{m}_i$ for all 
			$i = 0, \ldots, m$. These invariants form a $\W(\PP)$-basis for $\Inv(n)$. 
		\item
		  There exist unique invariants $\beta_0^n, \ldots, \beta_m^n$ 
			in $\Inv(n)$ such that $\beta_i^n|_{\Sq_m} = \beta'^{m}_i$ for all 
			$i = 0, \ldots, m$. These invariants form a $\W(\PP)$-basis for $\Inv(n)$. 
	\end{enumerate}
\end{theorem}

\begin{proof}
  The proof is essentially contained in the proof
	of \cite[Theorem 29.2]{Garibaldi-Serre-Merkurjev}, but since
	the argument is slightly convuluted we reproduce it here for completeness.
	As in the proof of \cref{thmMultiquadraticFree}, it is sufficient
	to prove the analogous statements for $\Inv_k(n)$ for $k \in \PP$. 

	Let us introduce an additional family of invariants.
	  For $J \subset \{1, \ldots, m\}$, we  define $\alpha'^{m}_J \in \Inv(\Sq_m)$
		by $\alpha'^{m}_J(a_1, \ldots, a_m) = \prod_{j \in J} \lra{a_j}$.
		By definition, $\alpha'^{m}_i = \sum_{|J|=i} \alpha'^{m}_J$. 
	By \cite[Theorem 27.15]{Garibaldi-Serre-Merkurjev}
	the elements $\alpha'^{m}_J$ with $J \subset \{1, \ldots, m\}$
	form a $\W(k)$-basis of $\Inv_k(\Sq_m)$. 
	It follows that the elements $\alpha'^{m}_i$ with $i = 0, \ldots, m$
	form a basis of $\Inv_k(\Sq_m)^{\mathfrak S_m}$. 
	
	Recall that $\Inv_k(n) \to \Inv_k(\Sq_m)$ is injective
	and that it clearly takes values in $\Inv_k(\Sq_m)^{\mathfrak S_m}$.
	When expressing $\lambda_i^{n}|_{\Sq_m}$ in terms of the basis
	$\alpha'^{m}_i$, we see easily that
	\[
	  \lambda_i^{n}|_{\Sq_m} = \lra{2^i} \alpha'^{m}_i \mod \langle \alpha'^{m}_j : j < i \rangle.
	\]
	Since $\lra{2^i}$ is a unit  in $\W(k)$, 
	it follows that $\lambda_0^{n}|_{\Sq_m}, \ldots, \lambda_m^{n}|_{\Sq_m}$
	is another basis for $\Inv_k(\Sq_m)^{\mathfrak S_m}$.
	This proves the first three items of the theorem.
	
	The last item follows easily from the previous items. First note that $\beta'^{m}_i\in \Inv_k(\Sq_m)^{\mathfrak S_m}$. Next, 
  expressing 
	$\beta'^{m}_i$ in terms of the basis
	$\alpha'^{m}_i$, we get
	\[
	  \beta'^{m}_i = \lra{2^i} \alpha'^{m}_i \mod \langle \alpha'^{m}_j : j < i \rangle
	\]
	as above, which implies again that $\beta_0^{n}, \ldots, \beta_m^{n}$ exist, are unique, and  form a basis of $\Inv_k(n)$.
\end{proof}

To lighten notation, we will use the notation $\lambda_i$, $\beta_i$, and $\alpha_i$ rather than $\lambda_i^n$, $\beta_i^n$, and $\alpha_i^n$ when no confusion is possible.

\begin{remark} \label{remBaseChange}
In  \cref{prop:beta basis}, we explicitly compute the
base changes for the three bases in \cref{thmBetaBasis1var}. 
Nevertheless, expanding the previous proof 
slightly we observe the following facts.
	\begin{enumerate}
		\item 
		  The change of basis matrices between the bases
			$(\lambda_i)_i$, $(\alpha_i)_i$ and $(\beta_i)_i$
			are triangular matrices with coefficients
			in $\Z[\lra{2}] = \Z\lra{1} + \Z\lra{2} \subset \W(\PP)$,
			see 
			\cref{remGeneralS} and 
			\cref{remGroupStr}
			for more details
			on this subring.
		\item 
		  The diagonals of these matrices are either 
			constantly $1$ (for $\lambda_i$ and $\beta_i$) 
			or alternate between $1$ and $\lra{2}$ (otherwise).
		\item
		  The change of basis
			between $(\lambda_i)$ and $(\beta_i)$
			is in general not an integer matrix. 
			For example, for $n = 2m$ 
			we have 
			$\lambda_2 = \beta_2 + \lra{2} \beta_1 - m \beta_0$, 
			and for $n = 2m+1$ we find
			$\lambda_2 = \beta_2 + (\lra{1}+\lra{2}) \beta_1 - m \beta_0$.
	\end{enumerate}
\end{remark}
\begin{remark} \label{remOddN}
	If $n$ is odd, the map 
	$A \mapsto A \times K$
	  induces an isomorphism 
	  $\Inv(n) \cong \Inv(n-1)$
	  which sends $\beta^{n}_i$ to $\beta^{n-1}_i$. 
  \end{remark}

It is 
easy to generalize the above discussion to the multivariate case. 
In the following, we focus on the $\beta$-basis which is of most interest for
our purposes. 

\begin{definition} \label{defMultivariateBases}
	Given $\i, \n \in \N^{r+1}$, we denote by $\beta_\i^\n  \in \Inv(\n)$ 
	the Witt invariant of degree $\n$
	given by 
	\[
	  \beta_\i^\n (A_0, \dots, A_r) = \prod_{j=0}^r \beta_{i_j}^{n_j}(A_j). 
	\]
\end{definition}

We emphasize that $\beta_\i^\n$ is indeed a Witt invariant
since $\W$ factors through $\Rings$, that is, 
the maps $\W(K) \to \W(L)$ for $K \to L$ are ring homomorphisms. 
In a completely analogous way, we can define Witt invariants
$\lambda_i^\n$ and $\alpha_\i^\n$ in $\Inv(\n)$. Again, to lighten notation we will use the notation $\lambda_\i$, $\beta_\i$, and $\alpha_\i$ rather than $\lambda_\i^\n$, $\beta_\i^\n$, and $\alpha_\i^\n$ when no confusion is possible.

Given $\n = (n_0, \dots, n_r) \in \N^{r+1}$, we set $\m = (\lfloor n_0/2 \rfloor, \dots, \lfloor n_r/2 \rfloor)$
throughout the following. We define
\[
  \N_\m := \{\i \in \N^{r} : 0 \leq i_j \leq m_j \text{ for all } j = 0, \dots, r\}. 
\]
Then clearly $\beta_\i^\n = 0$ for $\i \notin \N_\m$.

\begin{theorem} \label{thmBetaBasis}
  Given $k \in \Fields$ and $\n\in\N^{r+1}$, the ring $\Inv_{k}(\n)$ is a free $\W(k)$-module.
	Given a set of primes $S$ containing $2$, the ring $\Inv_S(\n)$ is a free $\W(\PP\setminus S)$-module with $(\beta_\i)_{\i \in \N_\m}$ as a basis. 
\end{theorem}

Of course $(\lambda_\i)_{\i \in \N_\m}$ and $(\alpha_\i)_{\i \in \N_\m}$ also yield bases of  $\Inv(\n)$. 

\begin{proof}
  Again, it is enough to prove the relative version for some $k \in \Fields$. We proceed by induction on $r$, the case $r=0$ being covered by  Theorem \ref{thmBetaBasis1var}. Let us assume that $r> 0$, and
	fix $\alpha \in \Inv_k(\n)$. 
	
	We denote by $\n' \in \N^{r}$
	the vector obtained from $\n$ by forgetting the last coordinate. 
	Given $k \to K$ and $T = (A_0, \ldots, A_{r-1}) \in \Et_{\n'}(K)$, 
	we denote by $\alpha(T) \in \Inv_K(n_r)$ the Witt invariant given 
	by 
	\[
	  \alpha(T)(A) = \alpha(A_0 \otimes_K L, \ldots, A_{r-1} \otimes_K L, A)
	\]
	for any $K \to L$ and $A \in \Et_{n_r}(L)$. By  Theorem \ref{thmBetaBasis1var} we have 
	$\alpha(T)=\sum_i \alpha(T)_i \beta^{n_r}_i$ for some uniquely defined 
	$\alpha(T)_i \in \W(K)$. A simple calculation shows that for all $K \to L$,
	we have 
	\[
	\alpha(T \otimes_K L) = \sum_i [\alpha(T)_i \otimes_K L] \beta^{n_r}_i
	\in \Inv_L(n_r).
	\]
	Hence 
	for each $i$ the assignment $T \mapsto \alpha(T)_i$
	is a Witt invariant of degree $\n'$. 
	The statement then follows from the  induction hypothesis for $r-1$ applied to
	these invariants. 
\end{proof}

\begin{remark} \label{remBetaBasis}
  Essentially equivalent to the theorem is the following statement.
	Set $|\m| = m_0 + \ldots + m_r$. Then
	$\Sq_{|\m|} = \Sq_{m_0} \times \ldots \times \Sq_{m_r}$ is a subfunctor
	of $\Et_\n$ and the corresponding restriction map
	$\Inv(\n) \to \Inv(\Sq_{|\m|})$ is injective. Moreover, 
	the product of symmetric groups $\mathfrak S_\m = \mathfrak S_{m_0} \times \ldots \times \mathfrak S_{m_r}$
	acts on $\Inv(\Sq_{|\m|})$ and the image of 
	$\Inv(\n) \to \Inv(\Sq_{|\m|})$ is equal to the set in invariant
	elements $\Inv(\Sq_{|\m|})^{\mathfrak S_\m}$.
	In particular, this remark allows to 
	verify equations in $\Inv(\n)$ 
	via restriction to $\Sq_{|\m|}$, that is, 
	to multiquadratic algebras. 
\end{remark}

\begin{corollary}
For any  $\n\in \N^{r+1}$ and $k\in \Fields$, one has 
\[
\Inv_S(\n)=\bigotimes_{j=0}^r\Inv_S(n_j)\qquad\mbox{and}\qquad \Inv_k(\n)=\bigotimes_{j=0}^r\Inv_k(n_j).
\]
\end{corollary}

Given $\alpha \in \Inv_S(\n)$, the unique coefficients
$b_\i = (b_\i^{k})_{k \in \PP\setminus S} \in \W(\PP\setminus S)$
such that $\alpha = \sum_\i b_\i \beta_\i$
are called the \emph{$\beta$-coefficients} of $\alpha$.
Note that a priori the restrictions of $\alpha$ for different $k \in \PP \setminus S$ are completely
unrelated.
The
following definition is motivated by the Welschinger--Witt invariants defined in Section 
\ref{secWelschinger} and their conjectural relations to quadratic invariants of 
del Pezzo surfaces, see \cref{conjGeneral}. 
Recall that $\Z$ is seen as a subring of $\W(\PP\setminus S)$ via the diagonal map 
(it is indeed a subring since the map $\Z \to \W(\Q)$, $n \mapsto n \cdot \lra{1}$, is injective.

\begin{definition}
  A Witt invariant $\alpha \in \Inv_S(\n)$ or $\alpha\in\Inv_k(\n)$ is called  \emph{$\beta$-integral}
	if there exists a tuple of integers $(b_\i) \in \Z^{\N_\m}$
	such that $\alpha = \sum_\i b_\i \beta_\i$.
	We denote the set of $\beta$-integral Witt invariants by $\UInv_S(\n)$ and $\UInv_k(\n)$. 
\end{definition}
In terms of coordinates, 
we may therefore write
\[
  \UInv_S(\n) \cong \Z^{\N_\m} \subset W(\PP\setminus S)^{\N_\m} \cong \Inv(\n). 
\]
Here, the isomorphisms are of (additive) groups or $\Z$- and $\W(\PP)$-modules, 
respectively, but not rings.

\begin{exa}
	If $k\in\Fields$ is a real closed field, then $\W(k)= \Z\lra{1}$. In this case any Witt invariant in $\Inv_k(\n)$ is  $\beta$-integral. 
\end{exa}

Clearly, the restriction to $\Fields/k$ of a  $\beta$-integral Witt invariant in $\UInv_S(\n)$ is in $\UInv_k(\n)$
 for any $k \in \Fields_S$. Note however that given the data of a $\beta$-integral Witt invariant $\alpha_k\in \UInv_k(\n)$ for each $k\in\PP$, there does not exist in general a   $\beta$-integral Witt invariant $\alpha\in \UInv(\n)$ restricting to $\alpha_k$ for each $k\in\PP$.

\begin{lemma}
			The set of $\beta$-integral Witt invariants $\UInv_S(\n)$ is a subring 
			of $\Inv_S(\n)$. 
\end{lemma}
\begin{proof}
  By definition $\UInv_S(\n)$ is the additive subgroup in $\Inv_S(\n)$ generated
	by the $\beta_\i$ with  $\i \in \N_\m$.
	To see that it is also a subring, we note that
	$\Tr(A)^2 = 2 \Tr(A)$ for $A \in \Et_2(K)$.
	It follows that for all $i_1, i_2$ the product 
	$\beta'_{i_1} \beta'_{i_2}$ is an integral linear combination
	of the $\beta'_i$. The claim then follows by 
	\cref{thmBetaBasis} or \cref{remBetaBasis}. 
\end{proof}

\subsection{Specialization in mixed characteristic}

As mentioned, a Witt invariant $\alpha \in \Inv_S(\n)$ is
given by a collection of Witt invariants over each prime field $k \in \PP \setminus S$, 
chosen completely independently.
For $p$ an odd prime or $p = \infty$, we write $\alpha_p$ for the corresponding Witt invariant
in $\Inv_{\FF_p}(\n)$ or $\Inv_\Q(\n)$, respectively. 
Conversely, if $\alpha \in \UInv(\n)$ then the Witt invariants $\alpha_p$
are all determined by $\alpha_\infty$ since $\Z \to \W(\Q)$ is injective.
In this section, we discuss
properties of $\alpha \in \Inv(\n)$ that involve $\alpha_p$ for different
$p$ and use them to describe subsets of $\Inv(\n)$ close to $\UInv(\n)$.

Let $R$ be a discrete valuation ring with fraction field $K$ and residue field $\kappa$. 
We denote by $\Et_n(R)$ the set of isomorphism classes of étale $R$-algebras of relative degree $n$. For $\n\in\N^{r+1}$ define $\Et_\n(R)$ in the obvious way. 
We have maps 
\[
\begin{array}{ccc}
	\Et_\n(R)& \longrightarrow& \Et_\n(K)
	\\A &\longmapsto & A_K = A \otimes_R K
\end{array}
\qquad \mbox{and}\qquad
\begin{array}{ccc}
	\Et_\n(R)& \longrightarrow& \Et_\n(\kappa)
	\\A &\longmapsto & A_\kappa = A \otimes_R \kappa
\end{array}
\]
called
generic and special fibre map, respectively. 
If $R$ is complete, the map $\Et_\n(R) \to \Et_\n(\kappa)$ is bijective, 
see \cite[\href{https://stacks.math.columbia.edu/tag/04GK}{Lemma 04GK}]{stacks-project}.
Analogously, we have ring homomorphisms
\[
\begin{array}{ccc}
	\W(R)& \longrightarrow& \W(K)
	\\q &\longmapsto & q_K 
\end{array}
\qquad \mbox{and}\qquad
\begin{array}{ccc}
	\W(R)& \longrightarrow& \W(\kappa)
	\\q &\longmapsto & q_\kappa 
\end{array}
\]
induced by $\otimes_R K$ and $\otimes_R \kappa$, respectively. 
The homomorphisms $\Sq(R) \to \Sq(\kappa)$ and $\W(R) \to \W(\kappa)$
are surjective and, if $R$ is complete, even isomorphisms, see \cite[VI.1.1 and VI.1.5]{lam05}.

\begin{definition} \label{def:unramified_inv}
  Let $\alpha \in \Inv_S(\n)$ (respectively $\alpha \in \Inv_k(\n)$) be a Witt-invariant and $R$
	a complete discrete valuation ring such that $\kappa \in \Fields_{S}$ (respectively such that $k \subset R$). 
	Then $\alpha$ is \emph{unramified over $R$}
	if the diagram
\[\begin{tikzcd}
	{\Et_\n(K)} & {\Et_\n(R)} & {\Et_\n(\kappa)} \\
	{W(K)} & {W(R)} & {W(\kappa)}
	\arrow["{\alpha_K}"', from=1-1, to=2-1]
	\arrow[from=1-2, to=1-1]
	\arrow["\cong"{marking, allow upside down}, draw=none, from=1-2, to=1-3]
	\arrow["{\alpha_\kappa}", from=1-3, to=2-3]
	\arrow[from=2-2, to=2-1]
	\arrow["\cong"{marking, allow upside down}, draw=none, from=2-2, to=2-3]
\end{tikzcd}\]
is commutative. 
  Moreover, $\alpha$ is \emph{unramified} if it is unramified over all complete discrete valuation rings $R$
	as above. 
\end{definition}

The case of discrete valuation rings of equal characteristic is understood. 

\begin{proposition} \label{prop:unramified-equichar}
  The following holds. 
	\begin{enumerate}
		\item 
		  Any $\alpha \in \Inv_k(\n)$ is unramified. 
		\item
		  Given $\alpha \in \Inv_{S}(\n)$ and a discrete valuation ring
			$R$ such that $K$ and $\kappa$ are of the same characteristic,
			then $\alpha$ is unramified over $R$.
	\end{enumerate}
\end{proposition}

\begin{proof}
  The first item is just \cite[Theorem 27.11]{Garibaldi-Serre-Merkurjev}.
	The second item follows from the first one 
	by restricting $\alpha$ to $\alpha_k \in \Inv_k(\n)$
	where $k$ is the prime field whose characteristic is equal to
	that of $K$ and $\kappa$. 
\end{proof}

Let $S$ be a finite set of primes containing $2$ and let $\Z[S^{-1}]$ denote the ring obtained by inverting the primes in $S$ in $\Z$. Note that the tensor product $\otimes_{\Z[S^{-1}]} K$ for $K$ in $\Fields_S$ induces a ring homomorphism 
\[
\W(\Z[S^{-1}]) \to \W(\PP\setminus S).
\] 
This map is injective since $\W(\Z[S^{-1}]) \to \W(\Q)$ is injective, see
\cref{remGroupStr} below.
We therefore consider $\W(\Z[S^{-1}])$ as a subring of $\W(\PP\setminus S)$.

\begin{definition} 
  A Witt invariant $\alpha \in \Inv_{S}(\n)$ is \emph{$S$-integral}
	if its $\beta$-coefficients lie in $\W(\Z[S^{-1}])$.
\end{definition}

In Lemma~\ref{lem:genW(Z[Sp0])}, we give explicit generators for $\W(\Z[S^{-1}])$ 
when $S = S(p_0)$ is the set of all primes lower or equal a given prime $p_0$.
By \cref{remBaseChange} the bases $\alpha_\i$, $\lambda_\i$ or $\beta_\i$
are related to each other by a base change
over $\Z[\lra{2}]$  which, since $2 \in S$, is contained in $\W(\Z[S^{-1}])$.
Therefore a Witt invariant $\alpha$ is $S$-integral
if and only if its coefficients with respect to any
of the bases $\alpha_\i$, $\lambda_\i$ or $\beta_\i$ lie in 
$\W(\Z[S^{-1}])$. 

\begin{theorem} \label{thmUnramifiedQuasiIntegral}
  Let $S$ be a finite set of primes containing $2$.
	A Witt-invariant $\alpha \in \Inv_S(\n)$ is unramified 
	if and only if it is $S$-integral. 
\end{theorem}

\begin{proof}
  \enquote{if}:
	Let $R$ be a complete discrete valuation ring. Since
	$\W(\kappa) \to \W(K)$ is a ring homomorphism, 
	sums and products of Witt invariants unramified over $R$ are again unramified over $R$.
	Since the constant Witt invariants represented by the elements in $\W(\Z[S^{-1}])$ are clearly unramified, 
	it suffices to show that e.g.\ the $\lambda$-basis is unramified.
	Since the map $\WG(\kappa) \to \WG(K)$ commutes with exterior powers,
	it is in fact sufficient to show that $\lambda_1 = \Tr$ is unramified.
	To see this, note that for $R \to A$ finite \'etale, there is a well-defined $R$-linear trace map
	$\tr_{A/R} \colon A \to R$ which is functorial, that is, 
	$\tr_{A/R} \otimes K = \tr_{A_K/K}$ and $\tr_{A/R} \otimes \kappa = \tr_{A_\kappa/\kappa}$,
	\cite[\href{https://stacks.math.columbia.edu/tag/02DU}{Tag 02DU}]{stacks-project}.
	Moreover, the $R$-quadratic form $\Tr(A) \colon a \mapsto \tr_{A/R}(a^2)$
	is unimodular. See \cite[Proposition 4.10]{GR-RevetementsEtalesEt}.
	Hence $\Tr(A) \in \W(R)$. By the above functoriality, we have
	$\Tr(A)_K = \Tr(A_K) \in \W(K)$ and $\Tr(A)_\kappa = \Tr(A_\kappa)$ which implies the claim.

\medskip
	\enquote{only if}:
	We denote by
	$a_0, \ldots, a_m \in \W(\Q)$ the $\alpha$-coefficients of $\alpha_\infty \in \Inv_\Q(n)$
	and by $a_{p,0}, \ldots, a_{p,m} \in \W(\FF_p)$ the $\alpha$-coefficients of 
	$\alpha_{p} \in \Inv_{\FF_p}(n)$ for $p$ not in $S$.

	In general, for a field of Laurent series $K((x_1, \ldots, x_m))$
	we have 
	\[ 
	\W(K((x_1, \ldots, x_m))) = \W(K)[t_1, \ldots, t_m]/(t_1^2-1,\ldots, t_m^2-1) \cong \W(K)[(\Z/2\Z)^m],
	\]
	with $t_i = \lra{x_i}$ by Springer's theorem \cite[IV Theorem 1.4]{lam05}. 
	For $p$ not in $S$, we have that $\alpha$ is unramified over 
	$\Z_p((x_1, \ldots, x_m))$, giving the commutative
	diagram:
\[\begin{tikzcd}
	{\Sq_m(\Q_p((x_1, \ldots, x_m)))} & {\Sq_m(\FF_p((x_1, \ldots, x_m)))} \\
	{W(\Q_p)[t_1, \ldots, t_m]} & {W(\FF_p)[t_1, \ldots, t_m]}
	\arrow["\alpha"', from=1-1, to=2-1]
	\arrow[from=1-2, to=1-1]
	\arrow["\alpha", from=1-2, to=2-2]
	\arrow[from=2-2, to=2-1]
\end{tikzcd}\]
	Evaluating on $(x_1, \ldots, x_m) \in \Sq_m(\FF_p((x_1, \ldots, x_m)))$, 
	it follows that
	\[
	  \sum_i (a_i \otimes_{\Q} \Q_p) P_i(t_1, \dots, t_m) = \sum_i \iota_p(a_{p,i}) P_i(t_1, \dots, t_m) \in W(\Q_p)[t_1, \ldots, t_m].
	\] where $\iota_p:\W(\FF_p) \cong \W(\Z_p)\to \W(\Q_p)$ denotes the ring homomorphism of Definition~\ref{def:unramified_inv} associated to the complete discrete valuation ring $\Z_p$.
	It follows that $(a_i \otimes_{\Q} \Q_p)= \iota_p(a_{p,i}) \in \W(\Q_p)$ for all $i$. By Springer's theorem, we have the short exact sequence
	\[
	0 \to \W(\FF_p) \xrightarrow{\iota_p} \W(\Q_p) \xrightarrow{d_p} \W(\FF_p) \to 0
	\] where $d_p$ is the second residue homomorphism. See for example  \cite[IV 1.6]{lam05}. In particular, $d_p\iota_p(a_{p,i}) =0$. Let $\delta_p$ denote the composition
	\[
	\delta_p: \W(\Q) \xrightarrow{\otimes_\Q \Q_p} \W(\Q_p) \xrightarrow{d_p} \W(\FF_p).
	\] It follows that $\delta_p(a_i) = d_p (a_i \otimes_{\Q} \Q_p) = d_p\iota_p(a_{p,i})=0$. By \cite[Corollary IV.3.3]{milnor73}, it follows that $a_i$ is in the image of $\W(\Z[S^{-1}])$. Since we have that $(a_i \otimes_{\Q} \Q_p)= \iota_p(a_{p,i}) $ as well, the theorem follows.
\end{proof}

\begin{remark} \label{remGroupStr}
  Expanding a bit the proof, we note that 
	by \cite[Theorem VI.4.1]{lam05} and \cite[Corollary IV.3.3]{milnor73}
	there is a (split) exact sequence
	\[
	  0 \to \W(\Z[S^{-1}]) \xrightarrow{\otimes \Q} \W(\Q) \xrightarrow{\prod \delta_p} 
		\prod_{p \in \PP \setminus S} \W(\FF_p) \to 0
	\]
	which is a rearrangement of the exact sequence 
	\[
	  0 \to \Z[\lra{2}] \xrightarrow{} \W(\Q) \xrightarrow{\prod \delta_p}
		\prod_{p \in \PP \setminus 2} \W(\FF_p) \to 0.
	\]
	It follows 
	that $\W(\Z[S^{-1}]) \to \W(\Q)$ is injective and
	that as a group $\W(\Z[S^{-1}])$ is isomorphic to 
	\[
	  \Z \oplus \Z/2\Z \oplus \bigoplus_{p \in S \setminus 2} \W(\FF_{p}).
	\]
	In particular, 
	$\W(\Z[2^{-1}]) \cong \Z \oplus \Z/2\Z$
	and $\W(\Z[2^{-1}, 3^{-1}] )\cong \Z \oplus \Z/2\Z \oplus \Z/4\Z$.
	For example, the isomorphism  $\Z \oplus \Z/2\Z \cong \W(\Z[2^{-1}])$
	is given by $(1,0) \mapsto \lra{1}$ and $(0,1) \mapsto \lra{1} - \lra{2}$.
\end{remark}

Let  $\Z[\lra{S}]$ denote the subring of $ \W(\PP\setminus S)$ generated by $\lra{p}$ for $p$ in $S$
\[
  \Z[\lra{S}] := \Z[\lra{p} : p \in S] \subset \W(\PP\setminus S).
\] 
Clearly, as a group $\Z[\lra{S}]$ is generated by the elements
$\lra{p_1 \cdots p_r}$ for (distinct) $p_1, \dots, p_r \in S$, 
and $\Z[\lra{S}] \subset \W(\Z[S^{-1}])$.
For a prime $p_0$, let $S(p_0)$ denote the set of primes lower or 
equal than $p_0$.

\begin{lemma}\label{lem:genW(Z[Sp0])}
  For any prime $p_0$, the subrings 
	$\W(\Z[S(p_0)^{-1}])$ and $  \Z[\lra{S(p_0)}] $
	of $\W(\PP\setminus S(p_0))$ (or $\W(\Q)$) 
	are equal. 
\end{lemma}

\begin{proof}
We prove this by induction on $p_0$. For $p_0=2$, the claim follows because $ \lra{1}$ and $ \lra{1} - \lra{2}$ generate  $\W(  \Z[S(p_0)^{-1}] )$ as a group by Remark~\ref{remGroupStr}. 
Now suppose the claim hold for all primes less than $p_0$. Then 
$\W(\Z[p^{-1}: p < p_0,p \text{ prime}]) \subset \Z[\lra{S(p_0)}]$ 
and we have an exact sequence 
\[
	0 \longrightarrow  \W(\Z[p^{-1}: p < p_0,p \text{ prime}])\longrightarrow \W(\Z[S(p_0)^{-1}])
	\longrightarrow \W(\FF_{p_0}) \to 0.
\]
But $\Z[\lra{S(p_0)}] \to \W(\FF_{p_0})$ is surjective since 
the elements $d_{p_0}\lra{ap_0} = \lra{a}$ for $0<a<p_0$ generate $\W(\FF_{p_0})$
as a group,
and the claim follows.
\end{proof}

\begin{remark} \label{remGeneralS}
  In general, 	
	the inclusion $\Z[\lra{S}] \subset \W(\Z[S^{-1}])$ is strict.
	For example, if $S = \{2,17\}$, then
	$\lra{3 \cdot 17} - \lra{3 \cdot 2} \notin \Z[\lra{2}, \lra{17}]$,
	but it is isomorphic to $q(x,y) = 3x^2 + 2 xy  -45 y^2 \in \W(\Z[1/34])$.
\end{remark}

\begin{corollary}
 Let $p_0$ be a prime and $S$ the set of primes lower or equal to $p_0$.
	A Witt-invariant $\alpha \in \Inv_S(\n)$ is unramified 
	if and only if  its $\beta$-coefficients lie in the image of $\Z[\lra{S}] $
\end{corollary}

\begin{proof}
Follows from Theorem~\ref{thmUnramifiedQuasiIntegral} and Remark~\ref{remGeneralS}.
\end{proof}

\section{Multireal values}\label{sec:multireal}

\subsection{Multireal values of $\beta$-integral Witt invariants}

In this section we exhibit  a finite set of values, called \emph{multireal values}, that completely characterize a $\beta$-integral Witt invariant. Recall that, given $K\in \Fields$ and $a\in K^*$, we have defined the degree two étale $K$-algebra $\E_a=K[X]/(X^2-a)$.
The multireal values are obtained by evaluating a Witt invariant in combinations
of $\E_{1} = K^2$ and $\E_{-1}$. 

\begin{definition}
	Given $K\in \Fields$ and $i, n \in \N$ such that $i \leq m = \lfloor n/2 \rfloor$, we call
\[
  R_s = R^n_s = (\E_{-1})^s \times K^{n-2s} \in \Et_n(K)
\]
 a \emph{multireal algebra} over $K$. 
Given $\n \in \N^{r+1}$ and $\s \in \N_\m$, we set 
\[
  R_\s = R^\n_\s = (R^{n_0}_{s_0}, \dots, R^{n_r}_{s_r}) \in \Et_\n(K). 
\]
Given a Witt invariant $\alpha \in \Inv_{k}(\n)$ over $k$, 
we call 
\[
  w_\s := \alpha_k(R_{\s}) \in \W(k), \s \in \N_\m,
\]
the \emph{multireal values} of $\alpha$.
For a Witt invariant $\alpha \in \Inv_S(\n)$, we define its multireal values
as the multireal values of its restriction to $\Inv_\Q(\n)$. 
\end{definition}

Given $k\in\Fields$ and $\n\in\N^{r+1}$, we define the map
\[
	\begin{array}{cccc}
		\mu_{k,\n}:&\Inv_k(\n) &\longrightarrow& \W(k)^{\N_\m}
		\\ &\alpha &\longmapsto& (w_{\s})_{\s \in \N_{\m}}
	\end{array}\cdot
\]
It is clearly a $\W(k)$-linear map, which furthermore decomposes as 
\[
	\mu_{k,\n}=\bigotimes_{j=0}^r \mu_{k,n_j}.
\]

In this subsection and the following one, we study to which extent 
the multireal values (or their signatures) 
determine a Witt invariant, i.e. we study the kernel of $\mu_{k,\n}$. Given $m\in \N$, we define the matrix $M_m\in \mathcal \Mat_{m+1,m+1}(\Z)$ by 
\[
  (M_m)_{s, i} = 2^i \cdot \binom{m-s}{i}, 
\]
that is 
\[
M_m=\left(\begin{array}{cccccc}
	1 & 2{m\choose 1}& 4{m\choose 2}& \cdots&\cdots & 2^m
	\\ 1 & 2{m-1\choose 1}& 4{m-1\choose 2}& \cdots& 2^{m-1}&0
	\\ \vdots &\vdots & \vdots&\iddots &\iddots &\vdots
	\\  \vdots &\vdots&4 &\iddots &&\vdots
	\\ \vdots  & 2 &0 &&&\vdots
	\\ 1&0 &0 &\ldots &\ldots&0
\end{array}\right)\cdot
\] 
Regarding binomial coefficients, our convention is that
$\binom{s}{i} = 0$ if $0 \leq i \leq s$ is not satisfied.

\begin{lemma} \label{lemMultirealMatrix dim1}
	Given $k\in\Fields$ and $n\ge 1$, the matrix $M_m$ is the matrix of the map $\mu_{k,n}$ in the base $(\beta_i)$ of $\Inv_k(n)$. In particular, 
		  \[
			 \det(\mu_{k,n}) = \pm 2^{\binom{m+1}{2}}. 
		  \] 
  \end{lemma}
  \begin{proof}
	Recall the identities $\Tr \E_{-1} = \lra{2} + \lra{-2} =
	0$
	  and $\Tr \E_{1} = \lra{2} + \lra{2} = \lra{1}+   \lra{1}= 2$ in $\W(k)$.
	  It follows that $\beta_i(R_{s}) = 2^i \binom{m-s}{i} \lra{1}$,
		which proves the statement. 
  \end{proof}

Given $\m\in\N^{r+1}$, we define the matrix $M_\m$ as the Kronecker product of the matrices $M_{m_0},\ldots,M_{m_r}$:
\[
M_\m=\bigotimes_{j=0}^r M_{m_j}.
\]
Hence elements of $M_\m$ are indexed by elements of $\N_\m$, and we have
\[
  (M_\m)_{\s, \i} = 2^\i \cdot \binom{\m-\s}{\i},
\]
where for $\i, \s \in \N^{r+1}$,
\begin{align*}
  2^\i &:= 2^{i_0 + \dots + i_r}, &
	\binom{\s}{\i} &:= \binom{s_0}{i_0} \cdots \binom{s_r}{i_r}.
\end{align*}
		The next lemma follows immediately from Lemma \ref{lemMultirealMatrix dim1}.
\begin{lemma} \label{lemMultirealMatrix}
	Given $k\in\Fields$ and $\n\in \N^{r+1}$, the matrix $M_\m$ is the matrix of the map $\mu_{k,\n}$ in the base $(\beta_\i)$ of $\Inv_k(\n)$. In particular $M_\m$ is upper-left-triangular, and 
		\[
		 \det(\mu_{k,\n}) = \pm 2^{\binom{m_0+1}{2} + \dots + \binom{m_r+1}{2}}. 
		\] 
\end{lemma}

Recall that a field $k$ is called \emph{formally real} if $-1$ is not a sum of squares in $k$. 
In particular, this implies that $k$ is of characteristic $0$.
Here is a list of well-known criteria for being formally real. 

\begin{proposition} \label{propFormallyReal}
  For $k \in \Fields$, the following statements are equivalent. 
	\begin{enumerate}
		\item The field $k$ is formally real. 
		\item The field $k$ can be ordered. 
		\item The group $\W(k)$ is not torsion. 
		\item The element $\lra{1} \in \W(k)$ is not torsion. 
	\end{enumerate}
\end{proposition}

\begin{proof}
  If $n\lra{1} = 0 \in \W(k)$ for some $n \in \N$, this means
	that $n \lra{1}$ is hyperbolic and in particular isotropic. 
	It follows that $-1$ is a sum of squares, e.g.\ \cite[Chapter I, Corollary 3.5]{lam05}.
	This implies (1) $\Rightarrow$ (4).
	The implication (4) $\Rightarrow$ (3) is trivial, (3) $\Rightarrow$ (2) is a
	consequence of \cite[Chapter VIII, Theorem 3.2]{lam05}
	and (2) $\Rightarrow$ (1) follows from \cite[Chapter VIII, Proposition 1.3]{lam05}
\end{proof}

In the following, when $k$ is a formally real field we identify $\Z$ with its image $\Z \lra{1}$ in $\W(k)$. 

\begin{corollary} \label{corMultirealValues}
  Let $k$ be a formally real field
	and $\alpha \in \UInv_k(\n)$.
	Then the multireal values of $\alpha$ lie in $\Z$ and determine $\alpha$ completely
	(that is, if all multireal values are zero, then $\alpha$ is zero). 
\end{corollary}

\begin{proof}
 By \cref{lemMultirealMatrix} the restriction of $ \mu_{k,\n} $ to $\Z^{\N_\m}$ has image in $\Z^{\N_\m}$, which proves the first claim.
	Since $\det M_\m=\det \mu_{k,\n} \neq 0$ still by \cref{lemMultirealMatrix} , this restriction 
	is injective. 
\end{proof}

\begin{corollary} \label{corUniversalMultireal}
  Let $\alpha\in \UInv_S(\n)$ be a $\beta$-integral Witt invariant. Then the following hold:
	\begin{enumerate}
		\item 
		  The multireal values of $\alpha$ lie in $\Z$ and determine $\alpha$ completely.
		\item 
		  The multireal values of $\alpha$ are equal to the multireal values 
			of the restriction of $\alpha$ to $\Inv_k(\n)$ for any formally real field $k$.
	\end{enumerate}
\end{corollary}

\begin{proof}
  Part (1) is just \cref{corMultirealValues} applied to the restriction of $\alpha$ to $\Inv_\Q(\n)$. 
	Part (2) follows from the fact that $W(\Q) \to W(k)$ is the identity on $\Z\lra{1} = \Z$ for 
	all formally real fields $k$.
\end{proof}

In view of Corollaries \ref{corMultirealValues} and \ref{corUniversalMultireal}, we define the set of \emph{$\beta$-integral multireal vectors} 
as $M_\m\Z^{\N_\m} \subset \Z^{\N_\m}$.

The matrix $M_\m$ is triangular and invertible over $\Q$, and we can compute 
$M_\m^{-1} w$ via the following simple recursion, which in turn allows to describe the set of $\beta$-integral multireal vectors.
Given $w=(w_\s) \in \Z^{\N_m}$ we define a \emph{multi-triangle} 
(referring to the shape of the index set $\i \leq \s$ in each coordinate) 
of rational numbers $(c_{\i}^\u)_{\u, \i, \u+\i \in \N_\m}$  by the following recursion. 
Here, $e_0, \dots, e_r$ denote the  vectors of the canonical basis of  $\Z^{r+1}$. 
\begin{enumerate}
	\item[init:] $c_\mathbf{0}^\u = w_{\m-\u}$. 
	\item[rec:] $c_{\i}^\u = \frac{1}{2} (c_{\i - e_j}^{\u+e_j} - c_{\i - e_j}^\u)$ for any $j = 0, \dots, r$
	            such that $i_j > 0$.
	\item[term:] $b_\i := c_\i^\mathbf{0}$. 
\end{enumerate}

We note that, given $w$, the recursion is indeed well-defined: Given $\i$ and $j, j'$ such that $i_j, i_{j'} > 0$, 
we have
\begin{align*}
  c_{\i - e_j}^{\u+e_j} - c_{\i - e_j}^\u
	&= \frac{1}{2}\left( c_{\i - e_j - e_{j'}}^{\u+e_j+e_{j'}} - c_{\i - e_j - e_{j'}}^{\u-e_j}
	   - c_{\i - e_j - e_{j'}}^{\u+e_{j'}} + c_{\i - e_j - e_{j'}}^{\u}\right) 
	= c_{\i - e_{j'}}^{\u+e_{j'}} - c_{\i - e_{j'}}^\u.
\end{align*}

The procedure is summarized in the following diagram displaying
the single variable case with the recursion formula indicated by the arrows. 

\begin{equation} \label{eq:RecursionTriangle}
  \begin{tikzcd}
	{\text{term\textbackslash init}} & {w_m} & {w_{m-1}} & {w_{1}} & {w_0} \\
	{b_0} & {c_0^0} & {c_0^{1}} & {c_0^{m-1}} & {c_0^m} \\
	{b_1} & {c_1^0} && {c_1^{m-1}} \\
	{b_{m-1}} & {c_{m-1}^0} & {c_{m-1}^{1}} \\
	{b_m} & {c_m^0}
	\arrow[dotted, no head, from=1-3, to=1-4]
	\arrow["{-1/2}"', from=2-2, to=3-2]
	\arrow[dotted, no head, from=2-3, to=2-4]
	\arrow["{1/2}", from=2-3, to=3-2]
	\arrow["{-1/2}"', from=2-4, to=3-4]
	\arrow["{1/2}", from=2-5, to=3-4]
	\arrow[dotted, no head, from=3-1, to=4-1]
	\arrow[dotted, no head, from=3-2, to=4-2]
	\arrow[dotted, no head, from=3-4, to=4-3]
	\arrow["{-1/2}"', from=4-2, to=5-2]
	\arrow["{1/2}", from=4-3, to=5-2]
\end{tikzcd}
\end{equation}

\begin{proposition} \label{propMultirealTriangle}
  Given $w=(w_\s) \in \Z^{\N_m}$, the vector $b=(b_\s) \in \Q^{\N_m}$ obtained by the
	above recursion is  $M_\m^{-1} w$. 
	Moreover the three following properties are equivalent.
	\begin{enumerate}
		\item $w$ is a $\beta$-integral multireal vector;
		\item $b$ is is an integral vector;
		\item $c = (c_{\i}^\u)_{\u, \i, \u+\i \in \N_\m}$ is integral. 
	\end{enumerate}
\end{proposition}

\begin{proof}
  The inverse process of the given recursion is the one that exchanges
	initialization and termination and uses the recursion identity
	$c^{\u}_{\i} = c^{\u-e_j}_\i + 2 c^{\u - e_j}_{\i + e_j}$ illustrated in the following
	diagram.
\[\begin{tikzcd}
	{c_0^0} & {c_0^{1}} & {c_0^{m-1}} & {c_0^m} \\
	{c_1^0} && {c_1^{m-1}} \\
	{c_{m-1}^0} & {c_{m-1}^{1}} \\
	{c_m^0}
	\arrow["1", from=1-1, to=1-2]
	\arrow[dotted, no head, from=1-2, to=1-3]
	\arrow["1", from=1-3, to=1-4]
	\arrow["2"', from=2-1, to=1-2]
	\arrow[dotted, no head, from=2-1, to=3-1]
	\arrow["2"', from=2-3, to=1-4]
	\arrow[dotted, no head, from=2-3, to=3-2]
	\arrow["1", from=3-1, to=3-2]
	\arrow["2"', from=4-1, to=3-2]
\end{tikzcd}\]
It therefore remains to show that this inverse recursion corresponds to $w = M_\m b$. 
Since $M_\m=\otimes M_{m_j}$, 
it suffices to prove this statement in the single variable case. 
Note furthermore that, since this inverse recursion has integer coefficients, there is no need in what follows to assume that $b$ is an integral vector. Hence let us consider $(b_i)_{i\in 0,\ldots,m}\in \W(\PP)^{m+1}$ and the Witt invariant $\alpha = \sum_i b_i \beta_i \in \Inv(n)$. 

For $0 \leq u \leq m$, we may consider the morphism $\Et_{n -2u} \to \Et_n$
given by $A \mapsto A \oplus k^{2u}$, and the corresponding map
$\Inv(n) \to \Inv(n-2u)$. We denote the image of $\alpha$ under this map by
$\alpha^{(u)}$. 
Combining the identities $P_i^{m} = P_i^{m-1} + x_m P_{i-1}^{m-1}$
and $\Tr \E_1 = 2 \lra{1}$, we obtain the relation
\[
  (\beta^n_i)^{(1)} = \beta^{n-2}_i + 2 \beta^{n-2}_{i-1}.
\]
Since  $\alpha = \alpha^{(0)}$, this implies that 
the $c_i^u$ are the unique integers such that
\[
  \alpha^{(u)} = \sum_{i=0}^{m-u} c_i^{u} \beta_i \in \Inv(n-2u)
\]
for all $u = 0, \dots, m$.
Finally since $\Tr \E_{-1} = 0$ in $\W(k)$ for any $k\in\Fields$, we deduce that $\beta^n_i(R_m) = 0$ for $i > 0$, i.e.
\[
  c^{u}_0 = \alpha^{(u)}_\Q(R_{m-u}) = \alpha_\Q(R_{m-u}) = w_{m-u}.
\]
Hence the multireal values of $\alpha$ are the values $(w_s)$ defined by the termination of the inverse recursion. 

If $\alpha\in \UInv(n)$, that is, if $b\in \Z^{\N_\m}$, then we deduce from \cref{corUniversalMultireal} that $w=M_\m b$ and so $b=M_\m^{-1}w$. 
This proves the first part of statement.

The second part is now straightforward. By definition $w$ is a $\beta$-integral multireal vector if and only if $b = M_\m^{-1}w$ is integer. 
Moreover, since the inverse recursion rule has integer coefficients, the vector $b$ is integral 
if and only if $c$ is integral. 
\end{proof}

For later purposes, we expand a bit on the main 
constructions used in the previous proofs. 
Given $\alpha=\sum_{\i\in\N_\m}b_i\beta_i \in\Inv(\n)$, we define its \emph{multireal triangle} as the multi-triangle  $c$ obtained from $b\in \W(\PP)^{\N_\m}$ by the inverse recursion described in the proof of \cref{propMultirealTriangle}.

\begin{definition} \label{defCuttingHoms}
  Given two integers $n$ and $u$ such that $n \geq 2u\ge 0$, we define three homomorphisms
	of $\W(\PP)$-modules $\Inv(n) \to \Inv(n-2u)$ given for $\alpha \in \Inv(n)$
	and $A \in \Et_{n-2u}(K)$ 
	by
	\begin{align*}
		\alpha^{(u)}(A) &= \alpha(A  \oplus \E_1^u), \\
		\alpha^{[u]}(A) &= \alpha(A  \oplus \E_{-1}^u), \\
		(\sum_{i=0}^m b_i \beta^n_i)^{\{u\}} &\mapsto \sum_{i=0}^{m-u} b_{u+i} \beta^{n-2u}_i. 
	\end{align*}
	Alternatively, the last homomorphism is given on the level of the $\beta$-basis by 
	$(\beta^n_i)^{\{u\}} = \beta^{n-2u}_{i-u}$,  declaring $\beta_i= 0$ for $i < 0$. 
	
	Given $\n \in \N^{r+1 }$ and $\u \in \N_\m$, we denote the obvious multivariate versions
	$\Inv(\n) \to \Inv(\n - 2 \u)$ by $\alpha^{(\u)}$, 
	$\alpha^{[\u]}$ and $\alpha^{\{\u\}}$.
\end{definition}

The motivation behind these definitions is implicit in the proof of
\cref{propMultirealTriangle}, we summarize it here for convenience
in the single variable case. 

\begin{corollary} \label{corCuttingHoms}
  Consider $\alpha \in \Inv(n)$ with multireal triangle $(c^u_i)$
	and fix $s \leq m$. Then the following statements hold. 
	\begin{enumerate}
		\item 
		  The multireal triangle of $\alpha^{(s)}$
			is $(c^u_i)_{u \geq s}$, that is, 
			the triangle obtained by removing $s$ times the left hand vertical side of the triangle. 
		\item 
		  The multireal triangle of $\alpha^{[s]}$
			is $(c^u_i)_{u + i \leq m-s}$, that is, 
			the triangle obtained by removing $s$ times the bottom diagonal side of the triangle. 
		\item
		  The multireal triangle of $\alpha^{\{s\}}$
			is $(c^u_i)_{i \geq s}$, that is, 
			the triangle obtained by removing $s$ times the top horizontal side of the triangle. 
	\end{enumerate}
\end{corollary}

\begin{proof}
  The first two items follow from 
	$P_i^{m} = P_i^{m-1} + x_m P_{i-1}^{m-1}$ and $\Tr \E_1 = 2\lra{1}$
	and $\Tr \E_{-1} = 0$ in $\W(K)$.
	The last item is by construction. 
\end{proof}

This construction implies the following, which may be seen as a \emph{coarse Abramovich--Bertram formula} for Witt invariants. In particular, it explains the factor $\lra{2}-\lra{2d}=2\lra{1}-\Tr(\E_d)$  in the quadratic Abramovich--Bertram formula \cite[Theorem 1.1]{Brugalle-WickelgrenABQ}. 
Given $\alpha \in \Inv(\n)$ and $j = 0, \ldots, r$, we define 
$\spl_j \alpha \in \Inv(\n - 2e_j, 2)$ 
by 
\[
\spl_j \alpha(\ldots, A_j, \ldots, \E) = \alpha(\ldots, A_j \times \E, \ldots).
\]
\begin{prop}\label{prop:ABtypeformula}
	Given $\alpha \in \Inv(\n)$ and $j = 0, \ldots, r$, one has 
 \begin{align*}
   \spl_j \alpha(A, \E) &= \alpha^{[e_j]}(A) + \Tr(\E) \alpha^{\{e_j\}}(A) \\
 	               &= \alpha^{(e_j)}(A) + (\Tr(\E)-2\lra{1}) \alpha^{\{e_j\}}(A).
 \end{align*}		
	\end{prop}
	
	\begin{proof}
		 This just follows from doing the calculus on the bases: We have
	\begin{align*}
	  \spl \beta^n_i = \beta^{n-2}_i + \Tr(\E) \beta^{n-2}_{i-1}, \\
		(\beta^n_i)^{[1]} = \beta^{n-2}_i, \\
		(\beta^n_i)^{\{1\}} = \beta^{n-2}_{i-1}, \\
		(\beta^n_i)^{(1)} = \beta^{n-2}_i + 2 \beta^{n-2}_{i-1},
	\end{align*}
	which all follow from $P^m_i = P^{m-1}_i + x_m P^{m-1}_{i-1}$.
	\end{proof}

\subsection{Multireal values of Witt invariants}
We proved in the previous section that a $\beta$-integral Witt invariant in $\UInv(\n)$ or $\UInv_k(\n)$, with $k$ formally real, is 
 determined by its multireal values.
In general, the  multireal values of a Witt invariant is only determined up to 
certain torsion elements.

\begin{proposition} \label{propValuesTorsion}
  Given $k \in \Fields$ and $\alpha = \sum_{\i\in \N_\m} b_\i \beta_\i \in \Inv_k(\n)$, all multireal values of $\alpha$ are zero
	if and only if	
	each $b_\i\in\W(k)$ is a $2^{\m-\i}$-torsion element. 
\end{proposition}

\begin{proof}
We prove by induction on $\m$ that if $c=(c_\i^\u)_{\u,\i,\u+\i\in\N_\m}$ is the multireal triangle of $\alpha$, then $\alpha$ has only zero multireal values if and only if $c_\i^\u$ is a $2^\i$-torsion element of $\W(k)$ for all $\u,\i$, the case $\m=0$ being true by assumption.
Let now $\m$ be such that $\m-e_j\in \N^{r+1}$. By \cref{corCuttingHoms}, the multireal values of $\alpha$ and of $\alpha^{(e_j)}$ coincide, and it follows from the proof of \cref{propMultirealTriangle} that
\[
	\alpha^{(e_j)}=\sum_{\i\in\N_{\m-e_j}} c_\i^{\m-e_j}\beta_i.
\]
Since $c_\i^\u=c_\i^{\m-e_j}+2c_{\i+e_j}^{\m-e_j}$, the result follows.
\end{proof}

We now turn to a version of \cref{propValuesTorsion} involving integer numbers rather than element of Witt rings.
We denote by $\OO$ the set of orderings of the field $k \in \Fields$. 
Recall that $k$ is formally real if and only if $\OO \neq \emptyset$, 
see \cref{propFormallyReal}.
Given an ordering $P \in \OO$, we have an associated signature map
$\sgn_P \colon \W(k) \to \Z$ defined as the number of positive minus the number
of negative elements (with respect to $P$) in a diagonalization of an element of $\W(k)$,
see \cite[Chapter VIII, Section 3]{lam05}.
The \emph{total signature} is the map defined  by 
\[
\begin{array}{cccc}
	\sgn \colon& \W(k) &\longrightarrow& \Z^\OO 
	\\ & a &\longmapsto &(\sgn_P(a))_{P \in \OO}
\end{array}\cdot
\]
We recall that the torsion subgroup of $\W(k)$ is equal to the kernel 
of $\sgn$, see \cite[Chapter VIII, Theorem 3.2]{lam05}.
Given $\alpha \in \Inv_k(\n)$, we call $\sgn(w_\s)$ the \emph{multireal signatures}
of $\alpha$. The multireal signatures determine $\alpha$ up to torsion. 

\begin{proposition} \label{propSignatureTorsion}
  Given $k \in \Fields$ and $\alpha \in \Inv_k(\n)$, 
	the multireal signatures of $\alpha$ are all zero if and only if $\alpha$ lies in the torsion submodule
	of $\Inv_k(\n)$.
\end{proposition}

\begin{proof}
 Let $c=(c_\i^\u)_{\u,\i,\u+\i\in\N_\m}$ be the multireal triangle of $\alpha$. The triangle 
	$(\sgn(c^\s_\i))_{\u,\i,\u+\i\in\N_\m}$ has values in $\Z^\OO$. 
	Since $\Z^\OO$ is torsionfree, the top line is zero if and only if the left row is zero. 
	The statement follows. 
\end{proof}

We now consider the case $k = \Q$. Given $\alpha \in \Inv_\Q(\n)$, we denote
by $\alpha_\R$ the restriction of $\alpha$ to $\Inv_\R(\n)$. 
Note that since $\W(\R) = \Z$, any Witt invariants in $\Inv_\R(\n)$ is $\beta$-integral and has integer multireal values. 

\begin{proposition} \label{propSignatureTorsionQ}
  Given $\alpha\in \Inv_\Q(\n)$ , the multireal values of $\alpha_\R$ are all zero if and only if $\alpha$ lies in the torsion submodule
	of $\Inv_\Q(\n)$.
\end{proposition}
\begin{proof}
  The only possible ordering for $\Q$ is $P = \Q_{\geq 0}$
	and using see \cite[Chapter II, Proposition 3.2]{lam05} we have the diagram:
	\[\begin{tikzcd}
		{W(\Q)} & {W(\R)} \\
		& \Z
		\arrow["{\otimes_\Q \R}", from=1-1, to=1-2]
		\arrow["\sgn"', from=1-1, to=2-2]
		\arrow["\cong", from=1-2, to=2-2]
	\end{tikzcd}\]
  Therefore, the multireal signatures of $\alpha$ agree with the 
	multireal values of $\alpha_\R$. 
	The statement then follows from \cref{propSignatureTorsion}. 
\end{proof}

We emphasize again the meaning of this statement: even 
if $\alpha \in \Inv_\Q(\n)$ is not $\beta$-integral, the (integer) multireal
values of $\alpha_\R$ agree with the multireal signatures of $\alpha$
and hence determine $\alpha$ up to torsion. 
\begin{remark}
One can consider any real closed field $k$ rather than $\R$ in	the above discussion.
\end{remark}

\section{Welschinger--Witt invariants} \label{secWWelschinger}

In this section, we show how Welschinger invariants can be arranged in certain  $\beta$-integral Witt invariants which we call
Welschinger--Witt invariants.

\subsection{Welschinger invariants} \label{secWelschinger}
Here we briefly recall the definition of Welschinger invariants and recast their properties that are of interest in this paper. We refer to 
\cite{Welschinger-invtsReal4mflds,Welschinger-invCountHoloDisk,Bru18} for more details. 
A {\em real symplectic manifold} $(X,\omega,\tau)$ is a symplectic
manifold $(X,\omega)$ equipped with an anti-symplectic involution
$\tau$. An almost complex
structure $J$ on $X$ is called \emph{$\tau$-compatible} if it is
tamed by $\omega$, and if $\tau$ is $J$-anti-holomorphic.

Let $(X,\omega,\tau)$ be a compact  real symplectic 4-manifold, and denote by
$H_2^{-\tau}(X;\Z)$ the space of $\tau$-anti-invariant classes. 
The fixed point set of $\tau$ is denoted by $X^{\tau}$, and is assumed to be non-empty. 
Choose a class
$D\in H_2^{-\tau}(X;\Z)$ 
and
$s\in\N$ such that
\[c_1(X)\cdot D - 1- 2s =: r \in \N.
\]
Choose a configuration $\x$ made of $r$ points in $X^{\tau}$ and  $s$ 
 pairs of $\tau$-conjugated 
points in $X\setminus X^{\tau}$.
Given a   $\tau$-compatible almost complex structure $J$ tamed by $\omega$,
we denote by $\mathcal C(D,\x,J)$
 the set of real rational
$J$-holomorphic curves  in $X$ realizing the class $D$, and passing
 through $\x$. Then we define the integer
 \[
 \Wel_{(X,\omega,\tau)}(D;s)= \sum_{C\in\mathcal C(D,\x,J)}(-1)^{m(C)},
 \]
where
$m(C)$ is the number of nodes of  $C$ in $X^\tau$ with two
$\tau$-conjugated branches.
For a generic choice of $J$, the set  $\mathcal C(D,\x,J)$ is finite, and
 $W_{(X,\omega,\tau)}(D;s)$
depends
 neither on $\x$, $J$, nor on the deformation class of the triple
 $(X,\omega,\tau)$, see \cite{Welschinger-invtsReal4mflds,Wel15,Bru18}.
 We call these numbers the \emph{Welschinger invariants of $(X,\omega,\tau)$}.
 When $c_1(X)\cdot D - 1=0$, implying $s=0$, we use the shorthand
 \[
 \Wel_{(X,\omega,\tau)}(D) = \Wel_{(X,\omega,\tau)}(D;s).
\]
 
 \medskip

A real algebraic projective surface $X$ equipped with a real Kähler form $\omega$ 
induces a 
real symplectic 4-manifold $(X(\C),\omega,\tau)$, where $\tau$ is the action of $\Gal(\C:\R)$ on $X(\C)$. 
We note that we also get an induced $\tau$-compatible complex structure $J$ 
on $(X(\C),\omega,\tau)$ tamed by $\omega$, which
however may not be generic in the above sense.
The first Chern class composed with Poincaré duality  maps a divisor class $D\in \Pic(X)$ to an element
in $H_2(X;\Z)$, still denoted by $D$.
A $\Gal(\C:\R)$-invariant class is mapped to an element of $H_2^{-\tau}(X;\Z)$. Hence we can define Welschinger invariants $\Wel_X(D;s)$ for any class $D\in\Pic(X)(\R)$. We denote by $\pi:X_{n,s}\to X$  a blow-up of $X$ at a real configuration of $n$ distinct points containing exactly $s$ pairs of $\Gal(\C:\R)$-conjugated points.  Note that if $X(\R)$ is not connected and $n-2s\ge 1$, there exists several deformation classes of real blow-ups $\pi:X_{n,s}\to X$ depending on the distribution of the blown-up real points on $X(\R)$. However, it follows from  \cite[Theorem 1.3]{Bru18} that the Welschinger invariants of   $X_{n,s}$ do not depend on the choice of a particular deformation class. Subsequently we do not record the blown-up configuration in the notation $X_{n,s}$.
Denoting by $E_1, \ldots , E_n$
the  exceptional classes,  the map $(D, a_1,\ldots,a_n) \mapsto \pi^* D -a_1 E_1 -\cdots - a_n E_n$ identifies $\Pic (X_{n,s})$ with $ \Pic (X) \times \Z^n$, and we tacitly use this identification throughout the following. 
 Note that under this identification we have
$\Pic (X_{n,0})(\R) = \Pic (X) (\R) \times \Z^n$, 
and $\Pic (X_{2,1})(\R) = \Pic (X) (\R) \times \Delta$ where $\Delta$ denotes the 
diagonal in $\Z^2$. 
By   \cite[Theorem 1.3]{Bru18} combined for example with \cite[Corollary 1.3]{DingHu18}, we have
\begin{equation}\label{eq:W n=0}
\Wel_X(D;s) = \Wel_{X_{n,s}}(\pi^* D - E_1-\cdots-E_n),
\end{equation}
where  $n=c_1(X)\cdot D - 1$. 
Moreover, it is easy to check that $\Wel_X(D;s) = \Wel_{X_{n,s'}}(\pi^* D; s)$.

The next result is the key observation that relates Welschinger invariants to $\beta$-integral Witt invariants.

\begin{thm}[{\cite[Proposition 2.3]{Bru18}}] \label{thm:ABW}
	Let $X$ be a real algebraic projective surface, $D \in \Pic(X)(\R)$ and $d \in \N$.
	Then one has
	\[
    \Wel_{X_{2,1}}((D,d,d);s) -  \Wel_{X_{2,0}}((D,d,d);s) 
		  = 2 \sum_{\ell=1}^d  (-1)^{\ell} \Wel_{X_{2,0}}((D, d+\ell, d-\ell); s).
  \]
\end{thm}

\begin{remark} \label{rem:wel formula}
  The particular case $d = 1$ of the formula (in combination with \cref{eq:W n=0})
	gives
	\emph{Welschinger's formula} \cite[Theorem 0.4]{Welschinger-invtsReal4mflds}
  \[
    \Wel_X(D;s) -  \Wel_X(D;s+1) = 2 \Wel_{X_{1,0}}(\pi^*D - 2E_1;s).
  \] 
\end{remark}

In what follows, the algebraic surface 
  $X$ will always be $\R$-rational, that is, it will be either
 $Q(1)=\P_\R^1\times \P_\R^1$, 
the quadric ellipsoid $Q(-1)$ in $\P^3_\R$, 
or  $\P^2_{\R, n,s}$, 
equipped with any real Kähler form. Welschinger invariants of $\R$-rational surfaces can be all computed via Solomon's open WDVV equations, see \cite{HorevSolomon,ChenZinger21}.

\subsection{Welschinger--Witt invariants} \label{subsecWelschingerWitt}

In this subsection, we associate $\beta$-integral Witt invariants to complex deformation classes of marked rational surfaces. We first treat  the case of blow-ups of $\P^2_\C$.

\emph{Convention:}
Throughout the following, a bold face letter like $\n$ denotes a vector in $\Z^r$, often
$\N^r$ (where $r \in \N$ is fixed). 
Its coordinates are labelled as $\n = (n_1, \dots, n_r)$. 
A bold face letter with bar like $\nbar$ denotes vectors in $\Z^{r+1}$.
In this case, the first coordinate is denoted $n_0$ and the last $r$ coordinates
form a vector $\n$, so that $\nbar = (n_0, \n) \in \Z \times \Z^r$. 
Given $\n=(n_1,\ldots,n_r)\in \N^r$, we set
\[
  \m=\left(\left\lfloor \frac{n_1}2 \right\rfloor,\ldots ,\left\lfloor \frac{n_r}2 \right\rfloor\right)\in \N^{r}.
 \] 
 Given $\dbar = (d_0, \d) \in \N^{r+1}$, 
 we set 
\[
  n_0 := 3 d_0 - \sum_{j=1}^r d_{j}n_j -1
\] 
and $m_0 = \lfloor n_0/2 \rfloor$. In the cases of interest to us, we have
$n_0 \geq 0$. We set $\nbar = (n_0, \n)$ and $\mbar = (m_0, \m)$.
 
 Given $\s=(s_1,\ldots,s_r) \leq \m$, i.e. $s_j\le m_j$ for all $j=1,\ldots,r$,
we denote by $X_{\n,\s}$ the blow-up of $\P^2_\R$ 
 at $\PP=\PP_1\sqcup\cdots\sqcup\PP_r\subset \P_\R^2$, where each $\PP_j=\{p_{j,1},\ldots,p_{j,n_j}\}$ is a real configuration of $n_j$ points 
such that $p_{j,1}, \dots, p_{j,n_j-2s_j}$ are real points
and the remaining $2s_j$ points consist of $s_j$ pairs of $\Gal(\C:\R)$-conjugated points.
 We denote by  $L \in \Pic(X_{\n,\s})$  the pullback of the class of a line under the blow-up map, 
 and by $E_{j,t} \in \Pic(X_{\n,\s})$ the class in of the exceptional divisor corresponding to $p_{j,t}$.
 We also use the shorthand
 $E_j = E_{j,1} + \dots + E_{j,n_j}$ for all $j = 1, \dots, r$.
When $r=0$, we have $\n=()$ and $X_{\n,\s}=\P^2_\R$. 

Throughout the following, we consider $\Z^{r+1}$ as a subset of
$\Pic(X_{\n,\s})$, for any $\s$,
via the map
$\Delta \colon (d_0,d_1,\ldots,d_r)\mapsto d_0 L - d_1 E_1 - \dots - d_r E_r$. 
Note that under this identification $\Z^{r+1}$ lies in 
the subset $\Pic(X_{\n,\s})(\R)$
of $\Gal(\C:\R)$-invariant classes. 
We usually suppress the map $\Delta$ when there is no risk of confusion. For example, we write
 $\Wel_{X_{\n, \s}}(\dbar; s_0)$ for the associated
Welschinger invariant $\Wel_{X_{\n, \s}}(\Delta(\dbar); s_0)$.

\medskip
The main theorem of this section shows that the Welschinger invariants $(\Wel_{X_{\n, \s}}(D; s_0))_{\overline \s\in\N_\mbar}$ are the multireal values of a $\beta$-integral Witt invariant. In order to describe the multireal triangle of this latter, we need to introduce some further notation. 
Let us assume $n_0 = 0$, which is no significant restriction thanks to Identity \eqref{eq:W n=0}. Note that $\N_\mbar=\N_\m$ in this case.
Let $\s,\i\in\N_\m$ such that $\s+\i\in \N_\m$. 
We consider the surface $X_{\n,\s}$ and recall
that by our conventions, since $\s + \i \leq \m$, 
the classes $E_{j,1},\ldots E_{j,2i_j}$
 are real for all $j = 1, \dots, r$.
	We define $\gamma_{j,t}=E_{j,2t-1}- E_{j,2t}\in \Pic(X_{\n,\s})(\R)$.
	We set $\N^\i = \N^{i_1} \times \dots \times \N^{i_r}$
	and, for  $\ell=(\ell_1,\ldots,\ell_r)\in \N^\i$, 
  \[
|\ell|=\sum_{\substack{j=1,\ldots,r \\ t=1,\ldots,i_j}}\ell_{j,t}, \qquad \mbox{and}\qquad 
\Gamma_{\ell}=\sum_{\substack{j=1,\ldots,r \\ t=1,\ldots,i_j}}(\ell_{j,t}+1)\gamma_{j,t}.
\]
Finally, we define
\[
\Wel^\i_{X_{\n, \s}}(\dbar)=\sum_{\ell\in \N^\i} (-1)^{|\ell|}\Wel_{X_{\n, \s}}
\left(\Delta(\dbar)-\Gamma_\ell\right).
\]
Note that $\Wel^{0}_{X_{\n, \s}}(\dbar)=\Wel_{X_{\n, \s}}(\dbar)$.

\begin{theorem} \label{thmWelschingerUniversal}
	For $\n\in \N^r$ and $\dbar \in  \Z^{r+1}$ such that $n_0 \geq 0$, 
	the vector $w = (w_{\sbar})_{\sbar \leq \mbar} \in \Z^{\N_\mbar}$ given by
	\[
	  w_{\sbar} = \Wel_{X_{\n,\s}}(\dbar; s_0)
	\]
	is $\beta$-integral. In particular there exists a unique $\beta$-integral Witt invariant $\V_{\n,\dbar}\in \beta\Inv(\nbar)$ with multireal values $w$.
	Furthermore, if $n_0 = 0$
	then the multireal triangle of 
	$\V_{\n, \dbar} \in \UInv(\n)$ is given by
	\[
	  c_{\i}^{\u} = \Wel^{\i}_{X_{\n, \m-\u-\i}}(\dbar).
	\]
\end{theorem}

\begin{proof}
It suffices to consider the case $n_0 = 0$ by Identity \eqref{eq:W n=0}. 
In view of \cref{corUniversalMultireal} and \cref{propMultirealTriangle}, 
it remains to prove that, taking $c_{\i}^{\u} = \Wel^{\i}_{X_{\n, \m-\u-\i}}(\dbar)$
as a definition, this multitriangle satisfies
$c_\mathbf{0}^\u = w_{\m-\u}$ and 
$c_{\i+e_j}^{\u-e_j} = \frac{1}{2} (c_{\i}^{\u} - c_{\i}^{\u-e_j})$ for any $j = 0, \dots, r$
	            such that $u_j > 0$.
We already explained the first part: we have 
$c_\mathbf{0}^\u = \Wel^{0}_{X_{\n, \m-\u}}(\dbar) = \Wel_{X_{\n, \m - \u}}(\dbar)
= w_{\m-\u}$. 
The second part is an application of \cref{thm:ABW} (using the shorthand $D = \Delta(\dbar)$):
\begin{align*}
	c_{\i}^{\u} - c_{\i}^{\u-e_j} 
	&=\sum_{\ell\in \N^\i} (-1)^{|\ell|} \left(\Wel_{X_{\n, \m-\u-\i}}\left(D-\Gamma_\ell\right) - \Wel_{X_{\n, \m-\u-\i+e_j}}\left(D-\Gamma_\ell\right)\right)
	\\& = 2 \sum_{\ell\in \N^\i} (-1)^{|\ell|}\sum_{l\ge 1} (-1)^{l+1}\Wel_{X_{\n, \m-\u-\i}}\left(D-\Gamma_\ell - l\gamma\right)
	\\& = 2 \sum_{\ell\in \N^{\i+e_j}} (-1)^{|\ell|}\Wel_{X_{\n, \m-\u-\i}}\left(D-\Gamma_\ell\right)
	\\&=  2 \Wel^{\i+e_j}_{X_{\n, \m-\u-\i}}(D) 
	\\&= 2 c_{\i+e_j}^{\u-e_j}.
\end{align*}
Here, $\gamma$ is the difference 
$\gamma_{j,i_j+1} = E_{j,2i_j +1} - E_{j, 2i_j + 2}$
which is $\Gal(\C:\R)$-invariant since $p_{j, 2i_j+1}$ and $p_{j, 2i_j +2}$ are real points. 
	This proves the statement. 
\end{proof}

\begin{remark}
	There is geometry of the Welschinger invariants being encoded in the fact that they form a $\beta$-integral Witt invariant. Conversely, a hypothetical alternate proof of the $\beta$-integrality of a Witt invariant encoding the Welschinger invariants as in \eqref{eqn:Vd-defining} would impose congruence conditions on Welschinger invariants themselves, including Welschinger's formula and a version of the Abramovich--Bertram formula, see \cref{prop:ABtypeformula}. 
\end{remark}

Let $Y_{\eta}$ the blow-up of $\P^2_\C$ at a configuration
of $\eta$ labelled points $p_1,\ldots,p_\eta$ in $\P^2_\C$. 
This description of $Y_\eta$ provides a canonical identification of $\Pic(Y_\p)$ with $\Z^{\eta+1}$: a class $D$ maps to $(d_0,a_1,\ldots,a_\eta)$, where $d_0$ is the intersection number of $D$ with the class of a line, and $a_i$ is the intersection number of $D$ with the exceptional divisor $E_i$ corresponding to the point $p_i$. Let $D=(d_0,a_1,\ldots,a_\eta)\in \Pic(Y_\eta)=\Z^{\eta+1}$, and consider the partition $\sqcup_{j=1}^rI_j$ of the set $\{1,\ldots,\eta\}$ induced by the values of the $a_i$'s. That is  $i,i'\in I_j\Leftrightarrow a_i=a_{i'}$. We define $d_j$ to be the value of the $a_i$'s on $I_j$, and $n_j=|I_j|$. Finally we set 
 \[
 \V_D=\V_{\n,\dbar}.
 \]

\begin{definition} \label{defWelschingerWittInvariants}
	Let $\n \in \N^r$ and $\dbar \in \Z^{r+1}$ such that $n_0 \geq 0$.
	The $\beta$-integral Witt invariant $\V_{\n, \dbar}\in \UInv(\nbar)$ from \cref{thmWelschingerUniversal} is called 
	the \emph{Welschinger--Witt invariant} associated to $\n$ and $\dbar$.
	
	Given $D\in \Pic(Y_\eta)$, the $\beta$-integral Witt invariant $\V_{D}\in \UInv(\nbar)$  is called 
	the \emph{Welschinger--Witt invariant} associated to $D$.
\end{definition}

When $r=0$, and so $\n=()$ and $\dbar=d\in\N$, we rather use the notation $\V_{d}$,  and we call it a Welschinger--Witt invariant of $\P^2$.

\begin{remark}\label{rem:WG inv}
	There exists a unique lift of the Welschinger--Witt invariant $\V_D$ to a Witt invariant
	with value in the Witt--Grothendieck functor $\WG$ and with fixed rank.  
Indeed, recall that a quadratic form $q\in\WG(k)$ is determined by its rank and its class in $\W(k)$. Let $\widehat \beta_i:\Et_n\to\WG$ be the lift of $\beta_i$ of rank $2^i$, and $\widehat \beta_\i$ be the obvious multivariable version. Then if $\V_D=\sum_\i b_\i\beta_\i$, the Witt invariant 
\[
\widehat \V_D=\sum_\i b_\i\widehat \beta_\i + \frac{\GW_{Y_\eta}(D)-\Wel_{X_{\eta,0}}(D;0)}{2}h
\]
reduces to $\V_D$ in $\W$ and has rank the genus $0$ Gromov--Witten invariant $\GW_{Y_\eta}(D)$ for the class $D$ in the surface $Y_\eta$.
\end{remark}

So far we associated Witt $\beta$-invariants to all (deformation classes of) complex algebraic rational (marked) surfaces but $\P^1\times \P^1$. To include the latter, recall that 
$\Pic(\P^1\times \P^1)\cong \Z^2$ with $F_1=\{p\}\times \P^1(\C)$ and $F_2 =\P^1(\C)\times\{p\}$ corresponding to the canonical basis of $\Z^2$. 
Up to swapping $F_1$ and $F_2$, there is a canonical identification of $\Pic(\P^1\times \P^1)$ with $(\pi^*L-E_1-E_2)^\perp$ in $\Pic(Y_2)$. We then
define we define the Welschinger--Witt invariant of $\P^1\times \P^1$
as
\[
\V_{\P^1\times \P^1,(1,1),(d_1,d_2)}= \V_{(1,1),(d_1+d_2,d_1,d_2)} \qquad\mbox{and}\qquad \V_{\P^1\times \P^1,(2),(d)}= \V_{(2),(2d,d)}.
\]
Note that this is consistent with identifications at the level of Welschinger invariants since
\[
\Wel_{Q(1)}((a,b);s)= \Wel_{X_{2,0}}((a+b,a,b);s)\qquad\mbox{and}\qquad \Wel_{Q(-1)}((a,a);s)= \Wel_{X_{2,1}}((2a,a,a);s).
\]
Analogously, given $D = (d_1, d_2) \in\Pic(\P^1_\C\times\P^1_\C)$, we define 
\[
\V_{\P^1\times \P^1,(d_1, d_2)}=\V_{(d_1+d_2, d_1, d_2)}.
\]

\begin{exa}[Welschinger--Witt invariants of $\P^2$] \label{exa:WWP2}
	Let $d \in \Z \cong \Pic(\P^2)$. Combining Theorem \ref{thmWelschingerUniversal} with Remark \ref{rem:wel formula} we get 
	\[
	\V_{d}=\sum_{i=0}^{\left\lfloor \frac{3d-1}{2} \right\rfloor} \Wel_{X_{i, 0}}\left(d L-2(E_{1}+\ldots  E_{i});\left\lfloor \frac{3d-1}{2} \right\rfloor-i\right)\beta_{i}.
	\]

  In \cref{tab:WWP2}, we list the first values of 
	the Welschinger--Witt invariants $\V_{d}$
	in terms of the $\beta$-basis. 
	To do so, we use the Welschinger invariants $\Wel_{\P^2}(d ; s)$
	as computed for example in \cite{Br8}. 
	For example, for $d=4$ we find that 
	$\Wel_{\P^2}(4 ; s) = 240, 144, 80, 40, 16, 0$ for $s = 0, 1, 2, 3, 4$.
	The associated triangles is:
	\[
    \begin{matrix}
		  0 & 16 & 40 & 80 & 144 & 240 \\
			8 & 12 & 20 & 32 &  48 &     \\
			2 &  4 &  6 &  8 &     &     \\
			1 &  1 &  1 &    &     &     \\
			0 &  0 &    &    &     &     \\
			0 &    &    &    &     &     
		\end{matrix}
	\]
	Therefore, we find $\V_{4} = 8 \beta_1 + 2 \beta_2 +  \beta_3$.
	We also note that, while the full triangle is completely encoded by
	both the multireal values $w \in \Z^{m+1}$
	or the $\beta$-coefficients $b \in \Z^{m+1}$, the latter
	seem to be
	composed of smaller numbers.
\begin{table}[!ht]
  \[
    {\setlength{\extrarowheight}{3pt}
    \begin{array}{|c|c|}
     \hline  d & \V_{d}
    \\  \hhline{|=|=|} 1 & \beta_0
    \\ \hline 2 & \beta_0
    \\ \hline   3 & \beta_1
    \\ \hline  4& 8\beta_1 + 2\beta_2+\beta_3
    \\ \hline  5& 64\beta_0 + 46\beta_2 + 16\beta_3+12\beta_4+4\beta_5+\beta_6 
    \\ \hline  6& 1024\beta_0 + 256\beta_1 +1088\beta_2 + 848\beta_3+  728\beta_4 + 480\beta_5 + 288\beta_6+  132\beta_7   +46\beta_8 

    \\ \hline  \multirow{2}{*}{7}& -14336\beta_0+  13056 \beta_1+  4096\beta_2+  16978\beta_3+  16512\beta_4+  18088\beta_5
    \\ & +  16240\beta_6+  13491\beta_7+   9832\beta_8+   6238\beta_9+ 3336\beta_{10} 
    \\ \hline  \multirow{2}{*}{8}& -280576 \beta_0 + 390144 \beta_1 + 356352 \beta_2 + 913408 \beta_3 + 1300160 \beta_4 + 1719968 \beta_5 + 2029008 \beta_6
    \\ & + 2213368 \beta_7 + 2217016 \beta_8 + 2037884 \beta_9 + 1704276 \beta_{10} + 1285806 \beta_{11}
    \\ \hline  
  \end{array}
    }
  \]
  \caption{Welschinger--Witt invariants of $\P^2$ in $\beta$-basis}
  \label{tab:WWP2}
\end{table}
						
	Thanks to Proposition \ref{prop:beta basis}, we can also express 
	the invariants in terms of the $\lambda$-basis. 
	The results are displayed in \cref{tab:WWP2 lambda} 
	Interestingly, in all computed examples the product of the signature of two 
	consecutive coefficients in the $\lambda$-basis is non-positive.

\begin{table}[!ht]
  \[
  {\setlength{\extrarowheight}{3pt}
  \begin{array}{|c|c|}
    \hline  d & \V_{d}
    \\ \hhline{|=|=|} 1 & \lambda_0
    \\ \hline 2 & \lambda_0
    \\  \hline  3 & \lambda_1
    \\  \hline 4& -13\lambda_0 +13\lambda_1 - \lambda_2+\lambda_3
    \\ \hline  5& 589 \lambda_0+  (\lra{2} - 110)\lambda_1 +    109\lambda_2  +(\lra{2} - 14)\lambda_3     + 13\lambda_4 +  (\lra{2} - 2)\lambda_5       +\lambda_6  
    \\ \hline  6& 196500\lambda_0 -96160\lambda_1 + 49110\lambda_2 -21068\lambda_3   +9186\lambda_4  -3176\lambda_5+   1066\lambda_6   -236\lambda_7+     46\lambda_8
    \\  \hline  \multirow{2}{*}{7}& 116803576\lambda_0 -63115170\lambda_1+  32807172\lambda_2 -16003434\lambda_3+   7374736\lambda_4  -3105703\lambda_5
    \\ & +   1196494\lambda_6   -398753\lambda_7+    113384\lambda_8    -23786 \lambda_9    + 3336\lambda_{10}
    \\  \hline  \multirow{3}{*}{8}& -409568889748\lambda_{0} +  209980086324\lambda_{1} -102839510628\lambda_{2}  + 47794430388\lambda_{3}  -20878902720 \lambda_{4}
	\\ &   + 8478699840\lambda_{5}   -3148076928 \lambda_{6} +  1046510240 \lambda_{7}   -300864590  \lambda_{8}   + 71144126    \lambda_{9} 
	\\&-12439590   \lambda_{10}+    1285806\lambda_{11}
    \\  \hline 
  \end{array}
  }
  \]
  \caption{Welschinger--Witt invariants of $\P^2$ in $\lambda$-basis}
	  \label{tab:WWP2 lambda}
\end{table}
\end{exa}

\begin{exa}\label{exa=r=1}
	Similarly to the case of $\P^2$,  for 
	$\n=(1,\ldots,1)\in \N^r$,
	we obtain 
\[
\V_{\n,\dbar}=\sum_{i=0}^{m_0} \Wel_{X_{r+ i, s}}(d_0 L - d_1 E_1 - \dots - d_r E_r - 2(E_{r+1}+\ldots  E_{r+i});m_0-i)\beta_{i}.
\]
For example, using tables of Welschinger invariants in \cite{ChenZinger21},  we find
\begin{align*}
	 \V_{(6,3)}&= 
		224 \beta_1 + 92 \beta_2 + 78 \beta_3 + 40 \beta_4 + 
			20 \beta_5 + 6 \beta_6 + \beta_7
			\\& = -749\lambda_0  +749\lambda_1 -109\lambda_2+  109\lambda_3  -13 \lambda_4+  13 \lambda_5  -\lambda_6+    \lambda_7.
\end{align*}
\end{exa}

\begin{exa}[Welschinger--Witt invariants of $\P^1\times \P^1$]\label{exa:WWP1P1}

	Recall that we denote by $F_1$ and $F_2$ the classes of the two rulings of $\P^1_\C\times \P^1_\C$. 
	By Example \ref{exa=r=1}, we have
 \[
\V_{\P^1\times \P^1,(1,1),(a,b)}=\sum_{i=0}^{a+b-1} \Wel_{Q(1)_{i, 0}}(a F_1 + b F_2 -2(E_1+\dots+E_i)\emph{\emph{}};a+b-1-i)\beta_{i}.
\]
Moreover,  
we have
\begin{align*}
	\V_{\P^1\times \P^1,(2),(a)}&= \sum_{i=0}^{2a-1}\left( \Wel_{Q(-1)_{i, 0}}(D_{i}; 2a-1-i)\beta_{(i,0)} \right.
	\\
&\left.   \qquad \quad + \frac{\Wel_{Q(1)_{i, 0}}(D_{i}; 2a-1-i)-\Wel_{Q(-1)_{i, 0}}(D_{i}; 2a-1-i)}{2}\beta_{(i,1)} \right).
\end{align*}
where $D_i = a (F_1 + F_2) -2(E_1+\dots+E_i)$, or equivalently (c.f.\ \cref{thm:ABW})
\begin{align*}
  \V_{\P^1\times \P^1,(2),(a)}= \sum_{i=0}^{2a-1}\Big( & \Wel_{Q(-1)_{i, 0}}(D_{i}; 2a-1-i)\beta_{(i,0)} \Big. \\
  & \Big. + \sum_{\ell\ge 1}(-1)^\ell \Wel_{Q(1)_{i,0}}(D_{i} - \ell(F_1-F_2); 2a-1-i)\beta_{(i,1)} \Big) .
\end{align*}

We give a few values of  $\V_{\P^1\times \P^1}$, both in the $\beta$ and $\lambda$-basis, in Tables \ref{tab:WWP1P1 more} and \ref{tab:WWP1P1}. 
The needed values of Welschinger invariants can be found for example in \cite{ChenZinger21}.  The specific case of $\V_{\P^1 \times \P^1,(1,1), (a,2)}$ is treated  in Example \ref{exa:(2,a)}.

\begin{table}[!h]
	\[
	{\setlength{\extrarowheight}{3pt}
	\begin{array}{|c|ccc|}
	\hline  (a,b) &  \multicolumn{3}{|c|}{\V_{\P^1\times \P^1,(1,1),(a,b)}} 
		\\ \hhline{|=|===|}  
  (a,1) & \beta_0 & = & \lambda_0
		\\ \hline
	  \multirow{2}{*}{(3,4)}  & 224 \beta_0 + 92 \beta_1 + 78 \beta_2 & \multirow{2}{*}{=}&(\lra{2} + 639)\lambda_0  + (\lra{2}- 1)\lambda_1 +   (\lra{2} + 95)\lambda_2  
		\\ & + 40 \beta_3 + 20 \beta_4 + 6 \beta_5 + \beta_6 & & +(\lra{2} - 1)\lambda_3 +  (\lra{2} + 11)\lambda_4  + (\lra{2} - 1)\lambda_5       +\lambda_6
		  
		  \\  \hline  
		  \multirow{2}{*}{(3,5)} & 991 \beta_0 + 448 \beta_1 + 408 \beta_2 + 
			248 \beta_3 & \multirow{2}{*}{=}&
			-2595\lambda_0 + 3432\lambda_1 -1076\lambda_2+  792\lambda_3
			\\ &+ 158 \beta_4  + 80 \beta_5 + 32 \beta_6 + 8 \beta_7 && 
			-234\lambda_4+   120 \lambda_5 -24  \lambda_6+  8\lambda_7
		\\  \hline  
		\multirow{3}{*}{(4,5)} & 13056 \beta_0 + 7552 \beta_1 + 8128 \beta_2  & \multirow{3}{*}{=}&
		4395560 \lambda_0 - 2232960\lambda_1+   1113240\lambda_2 
		  \\&  + 
		  7248 \beta_3+ 6376 \beta_4 + 4864 \beta_5&&  
			-494672\lambda_3+   208472 \lambda_4  -74240\lambda_5   
		  \\ & + 3328 \beta_6 + 1920 \beta_7 + 912 \beta_8  && 
			23632 \lambda_6   -5376 \lambda_7+     912\lambda_8
		\\  \hline 
	\end{array}
	}
	\]
	\caption{
	Welschinger--Witt invariants of $\P^1 \times \P^1$
	}
	\label{tab:WWP1P1 more}
\end{table}

\begin{table}[!h] 
  \[
  {\setlength{\extrarowheight}{3pt}
  \begin{array}{|c|ccc|}
   \hline  a & \multicolumn{3}{|c|}{\V_{\P^1\times \P^1,(2),(a)}} 
    \\ \hhline{|=|===|}  1 & \beta_{(0,0)} & = & \lambda_{(0,0)}
    \\  \hline   2 & \beta_{(1,0)} +\beta_{(0,1)}& = & -\lambda_{(0,0)} +\lambda_{(1,0)} +\lambda_{(0,1)} 
    \\  \hline   \multirow{3}{*}{3} & 16 \beta_{(0,0)}  + 8 \beta_{(2,0)}+  2 \beta_{(3,0)}  + \beta_{(4,0)}  &  \multirow{3}{*}{=}&
    (\lra{2} + 91)\lambda_{(0,0)} + (\lra{2} - 27)\lambda_{(1,0)}
    \\ & + 15 \beta_{(0,1)} +8 \beta_{(1,1)}+   2 \beta_{(2,1)}   + \beta_{(3,1)} &&  +  (\lra{2} + 13)\lambda_{(2,0)}+(\lra{2} - 3)\lambda_{(3,0)}+ \lambda_{(4,0)} 
    \\ &&& +     2\lambda_{(0,1)} + 13\lambda_{(1,1)}   -\lambda_{(2,1)}  +\lambda_{(3,1)} 
    \\ \hline 
  \end{array}
  }
  \]
  \caption{Welschinger--Witt invariants of $\P^1\times \P^1$ (note that the product of the signature of two consecutive coefficients of the element $\lambda_{(i,0)}$ are non-positive)}
	 \label{tab:WWP1P1}
\end{table}

\end{exa}

\begin{exa}[The invariants  $\V_{\P^1\times \P^1,(1,1),(a,2)}$]\label{exa:(2,a)}
Note that one has 
\begin{equation}\label{eq:dehn}
\Wel_{Q(1)_{i,0}}(D_{a,i};s)=\Wel_{Q(1)}((a-i)F_1+2F_2;s),
\end{equation}
where $D_{a,i}=aF_1+2F_2 -2E_1-\cdots-2E_i$.
Indeed, by \cref{eq:W n=0} we have
\[
\Wel_{Q(1)_{i,0}} (D_{a,i};s)=\Wel_{Q(1)_{i+1,0}}(D_{a,i}-E_{i+1};s).
\]
Since Welschinger invariants are invariant under real Dehn twists, see \cite[Remark 1.3]{Bru18-surgery}, 
applying a Dehn twist with respect to the class $F_1-E_{i}-E_{i+1}$ gives
\begin{align*}
\Wel_{Q(1)_{i+1,0}}( D_{a,i}-E_{i+1};s)&=\Wel_{Q(1)_{i+1,0}}(D_{a-1,i-1}-E_i;s)
\\&=\Wel_{Q(1)_{i-1,0}}(D_{a-1,i-1});s).
\end{align*}
Identity \eqref{eq:dehn} follows by induction.

From \cref{exa:WWP1P1} and Identity \eqref{eq:dehn}, we conclude that  one has
\begin{align*} 
\V_{\P^1\times \P^1,(1,1),(a,2)}= \sum_{i=0}^{a-1}  \Wel_{Q(1)}((a-i)F_1+2F_2;a-i+1) \beta_i
\end{align*}
(note that $\beta_a$ and $\beta_{a+1}$ do not appear since 
$ \Wel_{Q(1)}( 2F_2;1)= \Wel_{Q(1)}(-F_1+2F_2;0)=0$). 
It is a nice exercise, left to the reader, to show that
\[
\Wel_{Q(1)}( aF_1+2F_2;a+1) = \left\lfloor\frac{a+1}{2}\right\rfloor 2^{a-1}=
\left\{\begin{array}{ll}
a 2^{a-2}&\mbox{if $a$ is even,}
\\ (a+1)2^{a-2}   &\mbox{if $a$ is odd.}
\end{array}  \right.
\]
Therefore, the first terms of $\V_{\P^1\times \P^1,(1,1),(a,2)}$ are  
\[
\V_{\P^1\times \P^1,(1,1),(a,2)} = \beta_{a-1} +  2\beta_{a-2} + 8\beta_{a-3} +16\beta_{a-4} +48\beta_{a-5}+ 
96\beta_{a-6}   \ldots 
\]
We give the decomposition of the invariants $\V_{\P^1\times \P^1,(1,1),(a,2)}$ in the 
$\lambda$-basis in Table \ref{tab:WWP1P1 lambda}.
\begin{table}[!ht]
	\[
	{\setlength{\extrarowheight}{3pt}
	\begin{array}{|c|c|}
	\hline  a & \V_{\P^1 \times \P^1,(1,1),(a,2)}
		\\ \hhline{|=|=|}  1 & \lambda_0
		\\  \hline  2 & \lambda_0  + \lambda_1 
		\\  \hline  3 & (\lra{2} + 11)\lambda_0  + (\lra{2} - 1)\lambda_1 +     \lambda_2
		\\ \hline  4& 3\lambda_0 +13\lambda_1 -1\lambda_2 +\lambda_3   
		\\ \hline  5&   123\lambda_0    -12\lambda_1+ (\lra{2} + 14)\lambda_2 + (\lra{2} - 3)\lambda_3  +\lambda_4   
		\\ \hline  6&  -99\lambda_0+ 153\lambda_1 -30\lambda_2+  18\lambda_3  -3\lambda_4  +\lambda_5  
		\\ \hline   7&   1336 \lambda_0   -304\lambda_1+ (\lra{2} + 204)\lambda_2  + (\lra{2} - 53)\lambda_3+ (\lra{2} + 21)\lambda_4 + (\lra{2} - 5)\lambda_5 +\lambda_6 
		\\  \hline 
	\end{array}
	}
	\]
	\caption{Welschinger--Witt invariants $\V_{\P^1\times \P^1,(1,1),(a,2)}$ 
	in $\lambda$-basis}
	\label{tab:WWP1P1 lambda}
\end{table}

\end{exa}

\cref{thm:ABW} and \cref{rem:wel formula} combined with \cref{prop:ABtypeformula} give the following Abramovich--Bertram type formula for Welschinger--Witt invariants. Recall that we identify $\Inv(\n-2e_j)$ with $\Inv(\n-2e_j, 1,1)$ if $\n-2e_j\ge 0$.
\begin{prop}\label{prop:ABWW}
	Let  $\n\in \N^r$ and $\dbar \in  \Z^{r+1}$, such that $n_0 \geq 0$ and $\n-2e_j\ge 0$. Then one has 
	\begin{align*}
		j > 0: &&   \spl_j \V_{\n, \dbar}(\cdot, \E) &= \V_{{(\n-2e_j, 1, 1)}, (\dbar, d_j, d_j)} + (2\lra{1}-\Tr(\E) ) \sum_{\ell \geq 1} (-1)^{\ell} \V_{(\n-2e_j, 1, 1), (\dbar, d_j+\ell,d_j-\ell)}, \\
		j = 0: &&   \spl_0 \V_{\n, \dbar}(\cdot, \E) &= \V_{{(\n, 1, 1)}, (\dbar, 1, 1)} + (\Tr(\E) - 2\lra{1}) \V_{(\n, 1), (\dbar, 2)}.
	\end{align*}
\end{prop}
\begin{proof}
	By \cref{prop:ABtypeformula}, one has 
	\begin{align*} 
		\spl_j \V_{\n, \dbar}(\cdot, \E) &= \V_{\n, \dbar}^{(e_j)} + ( \Tr(\E)-2\lra{1}) \V_{\n, \dbar}^{\{e_j\}}.
	  \end{align*}
	  Performing $\V_{\n, \dbar} \to \V_{\n, \dbar}^{(e_j)}$
	  corresponds to declaring two points
	  in $\PP_j$ \emph{real}, that is
	  \begin{align*} 
		  j > 0: &&   \V_{\n, \dbar}^{(e_j)} &= \V_{{(\n-2e_j, 1, 1)}, (\dbar, d_j, d_j)}, \\
		  j = 0: &&   \V_{\n, \dbar}^{(e_0)} &= \V_{{(\n, 1, 1)}, (\dbar, 1, 1)}.
	  \end{align*}
	  Furthermore by \cref{thmWelschingerUniversal} and \cref{corCuttingHoms}, one has 
	  \begin{align*} 
		  j > 0: &&   \V_{\n, \dbar}^{\{e_j\}} &= - \sum_{\ell \geq 1} (-1)^{\ell} \V_{(\n-2e_j, 1, 1), (\dbar, d_j+\ell,d_j-\ell)}, \\
		  j = 0: &&   \V_{\n, \dbar}^{\{e_0\}} &= \V_{(\n, 1), (\dbar, 2)}.
	  \end{align*}
		This proves the claim.
\end{proof}

\begin{remark}
	Weslchinger invariants are also defined for some real algebraic $3$-folds and real symplectict $6$-folds, see for example \cite{Welschinger3d,Wel0605,Solomon-thesis,PanSolWal08}. For simplicity we restrict here to the case of $\P^3_\R$. In this case, the number  $\Wel_{\P^3_\R}(d;s)$ is an invariant count of real rational curves of degree $d$ in $\P^3_\R$ interpolating a generic real configuration of $2d$ points, containing exactly $s$ pairs on complex conjugated points. We refer to \cite{Welschinger3d} for the precise definition of the sign of a real rational curve in $\P^3_\R$. The invariants $\Wel_{\P^3_\R}(d;s)$ enjoy  an analog of Weslchinger's formula (\cite[Theorem 0.3]{Welschinger3d}), and 
	 it follows from \cite[Theorem 1]{BruGeo16} that the $\beta$-integral  Witt invariant
$\V_{\P^3_\R,d}:\Et_{2d-2}\to\W$ defined by
	 \[
	 \V_{\P^3,d} = \sum_{\substack{d_1+d_2=d\\ 0\le d_1<d_2}} \V_{\P^1\times\P^1,(1,1),(d_1,d_2)}
	 \]
	 has multireal values
	 \[
		\V_{\P^3,d,\R}(\C^s\times \R^{2d-2s}) =  \Wel_{\P^3_\R}(d;s).
	 \]
	Analogously to \cref{rem:WG inv}, if $\V_{\P^1\times\P^1,(1,1),(d_1,d_2)}=\sum_i b_{d_1,d_2,i}\beta_i$, the Witt invariant
	 \[
	 \widehat \V_{\P^3,d} =\sum_i \sum_{\substack{d_1+d_2=d\\ 0\le d_1<d_2}} b_{d_1,d_2,i}\widehat \beta_i + \frac{\GW_{\P^3_\C}(d)-\Wel_{\P^3_\R}(d;0)}{2}h
	 \]
	 with value in $\WG$ reduces to $ \V_{\P^3,d}$ in $\W$ and  has rank the genus $0$  degree $d$ Gromov--Witten invariant $\GW_{\P^3_\C}(d)$ of $\P^3_\C$. See \cite{NguyenAnh25} for generalization of this discussion to other real algabraic Fano $3$-folds of index 2.
\end{remark}

\section{Quadratic Gromov--Witten invariants}\label{sec:quadGW}

In \cite{KLSW-relor,degree}, Kass, Levine, Solomon and the third author generalize the enumeration of (real) rational curves in (real) surfaces to enumeration of rational curves in $\A^1$-connected del Pezzo surfaces over a large class of fields $k$. These invariants no longer take values in $\Z$, but in the 
Witt-Grothendieck ring $\WG(k)$ (at least for $k$-rational surfaces over an infinite field, 
they in fact always lie in $\QF(k)$). 
For this reason, we call the invariants defined in \cite{KLSW-relor,degree} \emph{quadratic Gromov--Witten invariants}.

We recall their construction in \cref{sec:constr quadGW}. In fact we do slightly more: we show in \cref{thmQuadraticWitt-nonperfect} that they are Witt invariants, under certain hypotheses. To this end, we first need to extend the definition of quadratic Gromov--Witten invariants from \cite{KLSW-relor,degree} to include enumerations of curves through points over non-perfect fields (of characteristic not $2$ and $3$). See \cref{def:quadGW}.
Note that for the reader's convenience, we also provide in \cref{WittInvarianceQuadratic} 
an alternative version of parts of \cref{sec:constr quadGW} in terms of the enumerative description of these invariants. 
We then study specialization properties in mixed characteristic of these quadratic invariants in \cref{sec:sp mixed qGW}. In particular, we prove in \cref{thm:QXDunramifiedWittawayS} that quadratic Gromov--Witten invariants of surfaces defined over $\Z[1/6]$,
such as toric del Pezzo surfaces, are unramified away from characteristic 2 and 3.
The exclusion of characteristic $2$ and $3$ is due to the lack of definition of the invariants in these cases.
We  elaborate on relations between Welschinger--Witt and quadratic Gromov--Witten invariants in \cref{sec:conj}.

\subsection{Quadratic Gromov--Witten invariants as Witt invariants}\label{sec:constr quadGW}

Let $k\in\Fields$ and let $X$ be a del Pezzo surface over $k$. 
Let $D \in \Pic(X)(k)$ be an effective divisor class  such that $n = -K_X \cdot D -1 \geq 0$.
Let $\Mbar_{0,n}(X,D)$ denote the moduli stack representing tuples $(u: C \to X, (p_1,\ldots,p_n))$ consisting of a stable map $u: C \to X$ with $C$ a genus $0$ nodal curve over a base $B$ such that $u_*[C] = D$ on geometric fibers, and $p_i:B \to C$ are sections landing in the smooth locus of $C$. See \cite{Abramovich--Oort-mixed_char}. The total evaluation map 
\[
\ev=\ev_{X,D} : \Mbar_{0,n}(X,D) \to X^n
\] 
is defined by sending  $(u: C \to X, (p_1,\ldots,p_n))$ to $(u(p_1),\ldots, u(p_n))$. 
There is an open subscheme $\Mbar_{0,n}(X,D)^{\odp} \subset \Mbar_{0,n}(X,D)$ representing those $(u: C \to X, (p_1,\ldots,p_n))$ such that $C$ is smooth, the map $u: C \to u(C)$ is unramified (which is equivalent to the surjectivity of the map $u^* T^* X \to T^* C$ on cotangent spaces), and for every geometric point of the base $B$, the singularities of the image curves are only ordinary double points. See for example \cite[Lemma 2.14]{KLSW-relor}.

Let us now assume that $k$ is perfect of characteristic not $2$ and $3$
and that $X$ is $\A^1$-connected. 
Recall that the degree of a del Pezzo surface $X$ is the integer $K_X^2$.
Suppose that the following hypothesis is satisfied, see \cite[Hypothesis 1]{degree}. 

\begin{hypothesis}\label{deghyp1-part}
Either $X$ is of degree at least 4, or
$X$ has degree 3 and $n \neq 5$.
 \end{hypothesis} 
 
By \cite[Corollary 3.15]{KLSW-relor} there is a dense open subset $U \subseteq X^n$ such that $\ev^{-1}(U) \subset \Mbar_{0,n}(X,D)^{\odp}$ (in the language of \cite{Brugalle-WickelgrenABQ}, this is to say that $X$ is enumerative).
In \cite{KLSW-relor}, the class $D$ was additionally restricted to exclude the case of an $m$-fold multiple of a $-1$-curve for $m>1$. In this case, however, we can choose an open dense subset $U \subseteq X^n$ such that $\ev^{-1}(U) = \emptyset$ and the claimed results are trivial. 
By for example \cite[Lemma 2.27]{KLSW-relor}, the restriction of the evaluation map 
\[
\ev(U):=\ev^{-1}\vert_{\ev^{-1}(U)}: \ev^{-1}(U) \to U
\] 
is étale. It follows that the relative canonical bundle $\omega_{\ev(U)}$
is trivial. More explicitly, the determinant of 
$d\ev(U): \ev(U)^*\Omega_{U} \to \Omega_{\ev^{-1}(U)}$ 
gives an isomorphism
\[
\det d\ev(U):\cO_{\ev^{-1}(U)} \to \omega_{\ev(U)}.
\] 
The map $\ev(U)$
can be equipped with an orientation in the sense of fixing a square root of $\omega_{\ev(U)}$
(see \cite[Definition 2.2]{degree}) as follows.
There is a finite étale map $\pi: \mathcal{D} \to  \ev^{-1}(U)$ 
from the (functorial) {\em double point locus} $\mathcal{D}$. 
See \cite[Lemma 5.4, Section 6]{KLSW-relor}. 
The discriminant $\disc_{\pi}$ of the trace form section of $\pi$ 
determines an isomorphism 
\[
\disc_{\pi}:  \cO_{\ev^{-1}(U)} \to (\det \pi_* \cO_{\mathcal{D}})^{-2}.
\] 
 In \cite{KLSW-relor} the {\em double point orientation} of $\ev(U)$ is defined to be the composition $(\det d\ev(U)) \circ \disc_{\pi}^{-1} :  (\det \pi_* \cO_{\mathcal{D}})^{-2} \to \omega_{\ev^{-1}(U)} $.

In \cite{degree}, an $\A^1$-degree of oriented maps such as $\ev(U)$ is defined. This $\A^1$-degree takes values in the sections $\sWG(U)$ of a Witt--Grothendieck sheaf $\sWG$. The definition of $\sWG$ can be found in \cite[Section 2.3]{degree}, where it is denoted $\mathcal{GW}$. 
For us, it is sufficient to recall that 
the degree of $\ev(U)$ lies in the image of  
the canonical injective map $\WG(k) = \sWG(\Spec(k)) \to \sWG(X^n) \to \sWG(U)$,
see \cite[Sections 2.3 and 2.4]{degree}.
Our goal in the remainder of this section 
is to extend the construction of this degree
to the situation after base change
with an arbitrary (not necessarily perfect)
field $K$.

Let $k \to K$ be an arbitrary (not necessarily perfect)
field extension. Our general rule
is that any object over $K$ obtained by base change $\otimes_k K$
from a corresponding object over $k$ is decorated with the index $K$. 
For example, we set $X_K = X \otimes_k K$, $D_K = D \otimes_k K \in \Pic(X_K)(K)$ and
\[
  \ev_K = \ev_{X_K,D_K} = \ev \otimes_k K
	\colon 
	\Mbar_{0,n}(X_K,D_K) \to X_K^n.
\]
Over the dense open subset $U_K = U \otimes_{k_0} k \subset X_K^n$,
the map $\ev(U_K) := \ev_K\vert_{\ev_K^{-1}(U_K)} = \ev(U) \otimes k$
can be oriented by the pull back of the double point orientation
(note that by \cite[Tag 068E,08QL]{stacks-project}, 
there is a canonical isomorphism $\omega_{\ev(U_K) } \cong \omega_{\ev(U)} \otimes_{k} K$). 

Given $A \in \Et_n(K)$, 
we can study a twisted  evaluation map
\[
  \ev^A = \ev_K^A: \Mbar_{0,n}(X_K,D_K)^A \to (X_K^n)^A.
\] 
defined in \cite[Section 5]{degree}. 
The twist of a map oriented over a dense open subset of the base is likewise oriented over a dense open subset of the base. We denote by $U^A$ a dense open subset of $(X_K^n)^A$ such that 
\[
\ev(U^A):=\ev^A\vert_{\ev^{A, -1}(U^A)}
\colon \ev^{A, -1}(U^A) \to U^A
\] 
is oriented by the double point orientation.

\begin{lemma}\label{lem:k0kA-Q}
The canonical map $\WG(K) \to \sWG((X_K^n)^A) \to \sWG(U^A)$ is injective
and the $\A^1$-degree $\deg (ev(U^A)) \in \sWG(U^A)$ lies in the image of this map.
\end{lemma}

\begin{proof}
Since $X$ is $\A^1$-connected over $k$, the surface $X_K$ is $\A^1$-connected over $K$.
By \cite[Proposition 2.37]{degree}, it follows that $(X_K^n)^A$ is $\A^1$-connected. 
This implies that 
$\WG(K) \to \sWG((X_K^n)^A)$ is an isomorphism, 
see for example \cite[Proposition 2.26]{degree}.
Since $\sWG$ is an unramified sheaf by \cite{PaninOjanguren}, 
any restriction map to a dense open is injective. 
Hence the composed map $\WG(K) \to \sWG(U^A)$ is injective. 

For the second part, suppose first that  $k$ is of characteristic $0$. Then by \cite[Theorem 4.9]{degree}, the double point orientation of $\ev(U)$ extends to an orientation of the pullback of $\ev$ to an open subscheme of $X^n$ whose complement has codimension $\geq 2$. Thus the orientation of $\ev(U_K)$ extends to an orientation of the pullback of $\ev_K$ to an open subscheme of $X_K^n$ whose complement has codimension $\geq 2$. It follows that the orientation of $\ev(U^A)$ extends to an orientation of $\ev^A$ to an open subscheme $V$ of $(X_K^n)^A$ whose complement has codimension $\geq 2$. Thus $\deg(U^A)$ extends to a unique section in $\sWG(V)$. Moreover, using again that $\sWG$ is unramified, the restriction map $\sWG((X_K^n)^A) \to \sWG(V)$ is an isomorphism. 
 This proves the lemma for $k$ characteristic $0$.

Now suppose $k$ is of positive characteristic greater than $3$. Since $\ev(U)$ is oriented, we say the map $\ev$ is oriented away from codimension $1$. Moreover, $\ev$ can be equipped with lifting data in the sense of \cite[Assumption 2.16]{degree} by \cite[Section 5.2]{degree}. 
Since this lifting data is compatible with base change and twisting, the same is true 
for $\ev_{K}$ and even $\ev^A$. Thus the degree of $\ev(U^A)$ is the restriction of a section
from $\WG((X_K^n)^A)$ by \cite[Section 2.4]{degree}. This finishes the proof. 
\end{proof}

\begin{defi}\label{def:quadGW}
  Let $X$ be a $\A^1$-connected del Pezzo surface 
	over a perfect field $k$ of characteristic not $2$ or $3$, 
	and let $D$ be an effective divisor class such that 
	$(X,D)$ satisfies Hypothesis~\ref{deghyp1-part}. 
	Let $k \to K$ be a field extension and $A \in \Et_n(K)$.
	The \emph{quadratic Gromov--Witten invariant} 
	$\widehat Q_{X,D,K}(A)\in \WG(K)$ 
	is the unique element in $\WG(K)$ whose image
	in $\sWG(U^A)$ agrees with the $\A^1$-degree of $\ev(U^A)$
	as constructed in \cref{lem:k0kA-Q}.
\end{defi}
  
	Hence, given $(X,D)$ and $k \to K$ as above, we have a map 
	$\widehat Q_{X,D,K} \colon \Et_n(K) \to \WG(K)$. 
  By \cite[Proposition 5.26]{Brugalle-WickelgrenABQ}, for any field extensions $K \to L$, we have a commutative diagram
 	\[\begin{tikzcd}
		{\Et_n(K)} && {\WG(K)} \\
		{\Et_n(L)} && {\WG(L).}
		\arrow["{\widehat Q_{X,D,K}}", from=1-1, to=1-3]
		\arrow["{\otimes_K L}"', from=1-1, to=2-1]
		\arrow["{\otimes_K L}", from=1-3, to=2-3]
		\arrow["{\widehat Q_{X,D,L}}"', from=2-1, to=2-3]
	\end{tikzcd}\] 
  We have thus constructed a natural transformation $\widehat Q_{X,D} \colon \Et_n \to \WG(K)$ 
	over $\Fields/k$. 
	We denote by $Q_{X,D,K}$ and $Q_{X,D}$ 
	the composition of $\widehat Q_{X,D,K}$ and $\widehat Q_{X,D}$,  respectively, with the natural transformation $\WG \to \W$ 
	given by the quotient maps $\WG(K) \to \W(K)$. 
  In summary, we have the following.

 \begin{theorem}\label{thmQuadraticWitt}\label{thmQuadraticWitt-nonperfect}
Suppose $X$ is an $\A^1$-connected del Pezzo surface over a perfect field $k$ of characteristic not $2$ or $3$. Suppose $D \in \Pic(X)(k)$ is an effective divisor class such that 
Hypothesis~\ref{deghyp1-part} is satisfied. Set $n = -K_X \cdot D -1 \geq 0$.
Then $Q_{X,D}$ is a Witt invariant of degree $n$ over $k$.
 \end{theorem}
 
 For convenience, we give a more explicit proof under more restrictive hypotheses via the enumerative interpretation of $Q_{X,D}$ in \cref{WittInvarianceQuadratic}.

\begin{remark} \label{remChangePerfectField}
  We finish by mentioning that the value $\widehat Q_{X, D,K}(A)$ only depends on the tuple $(X_K,D_K)$.
	More precisely, let $k' \to K$ another perfect subfield of $K$
	and $(X', D')$ a del Pezzo surface and divisor as above
	such that $(X'_K, D'_K) \cong (X_K, D_K)$. 
	Then $\widehat Q_{X, D, K}(A) = \widehat Q_{X', D', K}(A)$.
This follows from the fact that 
the double point locus and discriminant are stable under base change, and hence the base change 
of the double point orientation is itself defined by a composition 
\[
(\det d\ev(U_K)) \circ \disc_{\pi_K}:  (\det (\pi_K)_* \cO_{\mathcal{D}_K})^{2} \to \cO_{\ev_K^{-1}(U_K)}.
\] 
Since this latter map can be defined without reference to $k$, 
it agrees on the overlap with the corresponding map obtained
from $k'$ over some $U'_K$. 
Thus the orientations in codimension $1$ of $\ev_K$ as well as the twist 
$\ev^A$ do not depend on this choice. 
It the follows from \cref{lem:k0kA-Q} that also the degree as element in $\WG(K)$
does not depend on the choice. 
\end{remark}

\subsection{Specialization in mixed characteristic}\label{sec:sp mixed qGW}

Next we prove that quadratic Gromov--Witten invariants are unramified Witt invariants under appropriate hypotheses. We refer to Definition~\ref{def:unramified_inv} for the meaning of an unramified Witt invariant.
The next proposition is closely related to how the quadratic Gromov--Witten invariants in positive characteristic are constructed in \cite{degree}.

\begin{proposition}\label{propQuadraticUnramified} 
  Let $X_0$ be a del Pezzo surface over a mixed characteristic complete discrete valuation ring
	$R_0$  with fraction field $K_0$ and perfect residue field $\kappa_0$. Let $D \in (\Pic X_0)(R_0)$ be a relative divisor class. Suppose that the characteristic of $\kappa_0$ is not $2$ or $3$,  $X_{\kappa_0}$ is $\kappa_0$-rational,  
	$( X_{\kappa_0}, D_{\kappa_0})$ satisfies Hypothesis~\ref{deghyp1-part}, and   $X_{K_0}$ is $\A^1$-connected. Let $R_0 \to R $ be a map of complete discrete valuation rings. Let $\kappa$ and $K$ denote the residue field and fraction field of $R$ respectively.
	Then $( X_K, D_K)$ satisfies Hypothesis~\ref{deghyp1-part}
	 and moreover the diagram
\[\begin{tikzcd}
	{\Et_n(K)} & {\Et_n(R)} & {\Et_n(\kappa)} \\
	{W(K)} & {W(R)} & {W(\kappa)}
	\arrow["{Q_{X,D,K}}"', from=1-1, to=2-1]
	\arrow[from=1-2, to=1-1]
	\arrow["\cong"{marking, allow upside down}, draw=none, from=1-2, to=1-3]
	\arrow["{Q_{X,D,\kappa}}", from=1-3, to=2-3]
	\arrow[from=2-2, to=2-1]
	\arrow["\cong"{marking, allow upside down}, draw=none, from=2-2, to=2-3]
\end{tikzcd}\]
commutes. 
\end{proposition}

\begin{proof}
Since $X_{\kappa}$ is the basechange of $X_{\kappa_0}$, we have that $(X_{\kappa},D_{\kappa})$ satisfies Hypothesis~\ref{deghyp1-part}. Because the intersection products $\deg(-K_X \cdot D), \deg(K_X \cdot K_X) \in \CH^0(\Spec R)$ are locally constant \cite[B18]{Kleinman_The_Picard_Scheme}, they are integers. Thus Hypothesis~\ref{deghyp1-part} is satisfied for $(X_K, D_K)$ because it is for $( X_\kappa, D_\kappa)$. See also \cite[Lemma 9.4]{KLSW-relor}.

 The evaluation map
\[
\ev_0: \Mbar_{0,n}(X_0, D_0) \to X_0^n
\] can be pulled back to an open subscheme $\U_0 \subset X_0^n$
such that $\U_{K_0} \subset X_{K_0}^n$ has complement of codimension at least $2$, 
$\U_0 \otimes \kappa_0 \subset X^n_{\kappa_0}$ is dense and the evaluation map carries a \enquote{double point} orientation, see \cite[Theorem 9.15 (2), Theorem 9.15 (4)]{KLSW-relor}
whose induced orientations on the special and general fibers
coincide with the double point orientations from \cref{sec:constr quadGW}.

Let $A \in \Et_n(\kappa)$. We may view $A$ as a finite \'etale extension of $R$ as well as a finite \'etale extension of $\kappa$ by the bijection $\Et_n(R) \cong \Et_n(\kappa)$. 
We use the shorthands $Y$, $Y_K$ and $Y_\kappa$ for the twists
$(X^n)^A$, $(X^n)^A_K = (X_K^n)^{A_K}$ and $(X^n)^A_\kappa = (X_\kappa^n)^A$, 
respectively. 
The pullback of $\ev_0$ to $R$ given by
\[
\ev: \Mbar_{0,n}(X, D) \to X^n
\] 
can be twisted by $A$. The open subset $U_0$ gives rise to an open subset $\U \subset Y$ such that  $\U_K \subset Y_K$ has complement of codimension at least $2$, 
$\U_\kappa \subset Y_\kappa$ is dense and the twisted evaluation map 
\[
\ev^A \colon M \to \U
\] 
carries the double point orientation
whose induced orientations on $\ev^A_K \colon M_K \to \U_K$ and $\ev^A_\kappa \colon M_\kappa \to \U_\kappa$
coincide with the double point orientations from \cref{sec:constr quadGW}.
Moreover, there is a well-defined $\A^1$-degree $\deg(\ev^A) \in \sWG(\U)$
by \cite[Section 5.2]{degree}.
Since $X_\kappa$ is $\kappa$-rational and hence $\A^1$-connected, 
using \cite[Proposition 2.21]{degree} we can assume that the associated
degrees $\deg(\ev^A_K) \in \sWG(\U_K)$ and $\deg(\ev^A_\kappa) \in \sWG(\U_\kappa)$
lift to $\sWG(Y_K)$ and $\sWG(Y_\kappa)$
where they are equal to the images of $\widehat Q_{X_K, D_K}(A_K) \in \WG(K)$ and 
$\widehat Q_{X_\kappa, D_\kappa}(A) \in \WG(\kappa)$, respectively.

Since $\U_\kappa$ is a dense open subset of the special fiber which is a rational surface, we can choose a point $p$ of $\U_{\kappa}$ with $\kappa \subseteq k(p)$ finite and odd degree, so that the degree is not a multiple of the characteristic of $\kappa$, by an argument similar to \cite[Proposition 2.5]{Brugalle-WickelgrenABQ}.
The extension $\kappa \subseteq k(p)$ is a separable because its degree is not a multiple of the characteristic of $\kappa$. 
Since  $\Et_n(R) \cong \Et_n(\kappa)$, the extension $\kappa \subseteq k(p)$ corresponds to a finite \'etale extension $R \subseteq S$. Let $t$ denote a uniformizer of $R$ (which is also a uniformizer of $S$). Since $Y = (X^n)^A \to \Spec R$ is smooth, it is formally smooth at $p$.
 By formal smoothness, 
starting with $q_0 = p \colon \Spec S/(t) = \Spec k(p) \to Y$,
any solid diagram \eqref{eq:formal_smoothness} 
\begin{equation}\label{eq:formal_smoothness}
\begin{tikzcd}
\Spec S/(t^{n}) \arrow[r,"q_n"] \arrow[d] &  Y \arrow[d] \\
\Spec S/(t^{n+1}) \arrow[r] \arrow[ru, dashed] & \Spec R
\end{tikzcd}
\end{equation} 
admits a lift $q_{n+1}:\Spec S/(t^{n+1}) \to Y$.
Since $R$ is complete, so is $S$. 
The maps $q_n$ thus determine a map $q: \Spec S \to Y$ such that $\Spec S \to Y \to \Spec R$ corresponds to the chosen extension $R \subseteq S$. Since $q_{k(p)} = p \in \U_\kappa$,
we have $q \in \U$.   
Let $L$ denote the fraction field of $S$ and $\tilde{p} = q_L \colon \Spec L \to Y_K$.
We can summarize the construction in the following commutative diagram:
\[
\begin{tikzcd}
	 {\WG(L)} & {\WG(S)}&{\WG(k(p))}\\
	{\sWG(\U_K)} & {\sWG (\U)} & {\sWG (\U_\kappa)} \\
	{\WG(K)} & {\WG(R)} & {\WG(\kappa)}
	\arrow["\tilde{p}",from=2-1, to=1-1]
	\arrow["q",from=2-2, to=1-2]
	\arrow["p",from=2-3, to=1-3]
	\arrow[from=1-2, to=1-1]
	\arrow["\cong",from=1-2, to=1-3]
	\arrow[from=3-1, to=2-1]
	\arrow[from=2-2, to=2-1]
	\arrow[from=2-2, to=2-3]
	\arrow[from=3-3, to=2-3]
	\arrow[from=3-2, to=3-1]
	\arrow[from=3-2, to=2-2]
	\arrow["\cong", from=3-2, to=3-3]
\end{tikzcd}
\begin{tikzcd}
	 {\bullet} & {\bullet}&{\bullet}\\
	{\deg(\ev^A_K)} & {\deg(\ev^A)} & {\deg(\ev^A_\kappa)} \\
	{\widehat Q_{X_K, D_K}(A_K)} & {} & {\widehat Q_{X_\kappa, D_\kappa}(A)}
	\arrow[maps to, from=2-1, to=1-1]
	\arrow[maps to,from=2-2, to=1-2]
	\arrow[maps to,from=2-3, to=1-3]
	\arrow[maps to,from=1-2, to=1-1]
	\arrow[leftrightarrow,from=1-2, to=1-3]
	\arrow[maps to,from=3-1, to=2-1]
	\arrow[maps to,from=2-2, to=2-1]
	\arrow[maps to,from=2-2, to=2-3]
	\arrow[maps to,from=3-3, to=2-3]
	\arrow[dashed, maps to,from=3-3, to=3-1]
\end{tikzcd}
\] 
Since the extension $K \subset L$ 
is of odd degree, the vertical composition 
$\WG(K) \to \WG(L)$
is injective, see  \cite[Theorem VII.2.7]{lam05}.
Hence the claim follows. 
\end{proof}

Let $S$ denote a finite set of primes containing $2,3$ and let $\sigma$ denote their product. 
Let $X \to \Spec \Z[1/\sigma]$ be a smooth, proper relative del Pezzo surface.
Let $D$ be an effective relative divisor class of $X$. 
As argued in \cref{propQuadraticUnramified}, the 
intersection numbers $K_X^2$ and $n = -K_X \cdot D -1$
are constant and hence integers. 
Suppose that Hypothesis~\ref{deghyp1-part} is satisfied. 
Suppose that for each prime field $k \in \PP \setminus S$
the surface $X_k$ is $k$-rational.
Examples are $\P^2_\Z$ and 
$\P^1_\Z \times \P^1_\Z$ (one may also consider
the blow-up of $\P^2_\Z$ in a subset of the points
$(1:0:0)$, $(0:1:0)$, $(0:0:1)$ and $(1:1:1)$).
Then, according to \cref{thmQuadraticWitt-nonperfect},
we have a Witt invariant $Q_{X,D,k} \in \Inv_k(n)$  for each
$k \in \PP \setminus S$ which we can package together
as a single Witt invariant $Q_{X,D} \in \Inv_S(n)$. 
 
  \begin{theorem}\label{thm:QXDunramifiedWittawayS}
  In the above situation, the Witt invariant $Q_{X,D} \in \Inv_S(n)$
	is unramified. 
  \end{theorem}

\begin{proof}
In equal characteristic, this follows from Proposition~\ref{prop:unramified-equichar}. 
So let us assume $R$ is as in Definition~\ref{def:unramified_inv}
and of mixed characteristic. 
Since $\sigma$ is invertible in $R$, there is a unique map $\Z[1/\sigma] \to R$. 
Let $R_0$ be the completion of $\Z[1/\sigma]$ at the inverse image of the maximal ideal of $R$. 
Let $K_0$ and $\kappa_0$ be the fraction field and residue field of $R_0$ respectively. 
Then $\kappa_0$ is a prime field and hence perfect. The claim follows from Proposition~\ref{propQuadraticUnramified}.
\end{proof}

\begin{remark}
Results analogous to Theorem~\ref{thm:QXDunramifiedWittawayS} hold for $X$ smooth and proper over a Dedekind domain, but we omit the statement for now to simplify.
\end{remark}

\begin{example} 
  The combination of \cref{thmQuadraticWitt}, 
	\cref{thmUnramifiedQuasiIntegral} and
	\cref{propQuadraticUnramified} shows that 
	for $d > 0$, $n = 3d-1$ 
	and $m = \lfloor n/2 \rfloor$,
	we have 
	$Q_{\P^2, d} = b_0 \beta_0^n + \dots + b_m \beta_m^n \in \Inv_{2,3}(n)$ with 
	$b_i \in \mathbb{Z}[\langle 2 \rangle, \langle 3 \rangle]$.
	In fact, in \cref{thm:dP6} we show that $Q_{\P^2, d}$ is even $\beta$-integral, that
	is, $b_0, \dots, b_m \in \Z$. 
\end{example}

\subsection{Quadratic Gromov--Witten and Welschinger--Witt invariants}\label{sec:conj}

Let $(X,D)$ be a tuple over a perfect field $k$
and let $k \to K$ be a field extension
as in \cref{def:quadGW}. 
We also have a Gromov--Witten invariant $\GW_{X_K}(D_K) \in \N$ which can be defined as the (ordinary)
degree of the  evaluation map $\overline{M}_{0,n}(X_K,D_K) \to X_K^n$.
See \cite{BM-StacksStableMaps,Beh-GromovWittenInvariants,BLRT}.

\begin{lemma} \label{corRankGromovWitten}
  For any $A \in \Et_n(K)$, the rank of $\widehat Q_{X,D,K}(A)$ is equal to the corresponding
	Gromov--Witten invariant $\GW_{X_K}(D_K) \in \N$.
\end{lemma}

\begin{proof}
  Since both $\rk(\widehat Q_{X,D,K}(A))$ and
	$\GW_{X_K}(D_K)$ are invariant under base change, 
	we may assume that $K$ is algebraically closed.
	In this case, the map $\rk \colon \QF(K) \to \N$ is an isomorphism
	and the $\A^1$-degree of the evaluation map used to define
	$\widehat Q_{X,D,K}(K^n)$ agrees with the ordinary degree 
	used to define $\GW_{X_K}(D_K)$, see
	\cite[Subsection 2.4]{KLSW-relor}.
\end{proof}

Since a quadratic form is determined by its rank and its class in the Witt ring, we obtain the following.
\begin{corollary}
	The quadratic Gromov--Witten invariant $\widehat Q_{X,D,k}$
	is completely determined by the Gromov--Witten invariant $\GW_X(D)$ and the 
	associated Witt invariant $Q_{X,D,k}$.
\end{corollary}
For convenience, let us summarize
the exact setting for the following
discussion:

\begin{setting} \label{settingConj}
  Let $k$ be a perfect field of 
	characteristic not $2$ or $3$. 
	Fix $\n \in \N^r$ and $(A_1, \dots, A_r) \in \Et_\n(k)$
	and suppose that 
	$X$ is a rational del Pezzo surface 
	constructed 
	as the blow-up of $\P^2_k$
	along the zero-dimensional subschemes $\p_1, \dots, \p_r \subset \P^2_k$
	such that $\p_i = \Spec A_i$.
	Let $D$ be an effective divisor class in $\Pic(X)(k)$
	such that $(X,D)$ satisfies Hypothesis~\ref{deghyp1-part}. 
	Suppose that $D$ corresponds to $\dbar \in \N^{r+1}$
  in the sense that the coefficient of $L$ in $D$ is $d_0$
  and the coefficient of each exceptional divisor
  projecting to $\p_i$ is $-d_i$. 
	We set $n_0 = n = -K_X \cdot D -1$ and $\nbar = (n_0, \n)$.
	Then we have associated Witt invariants
	$Q_{X,D} \in \Inv_k(n_0)$ 
	from \cref{thmQuadraticWitt-nonperfect}
	and $\V_{\n,\dbar} \in \Inv_k(\nbar)$
	from \cref{defWelschingerWittInvariants}.
	We define $\V_{X,D} \in \Inv_k(n_0)$ by 
	\[
		\V_{X,D}(A_0) = \V_{\n,\dbar}(A_0, A_1 \otimes K, \dots, A_r \otimes K)
	\]
	for any extension $k \to K$ and $A_0 \in \Et_{n_0}(K)$. 
\end{setting}

\begin{lemma} \label{lem:equaltors}
  Supposing \cref{settingConj}
  and  $k = \Q$, the
	Witt invariants $Q_{X,D,k}, \V_{X,D} \in \Inv_k(n_0)$ have the same multireal	signatures. 
	In particular, they only differ by torsion. 
\end{lemma}

\begin{proof}
  Since the multireal signatures are given by evaluating
	the invariants for $K = \R$, the first statement
	follows from the 
	enumerative description of quadratic invariants
	\cite[Theorem 3]{degree} \cite{Levine-Welschinger}, see also \cref{WittInvarianceQuadratic},
	and the relationship to Welschinger signs, see \cite[Remark 2.5]{Levine-Welschinger}
	\cref{exRealQuadratic}.
	The second statement follows from \cref{propSignatureTorsionQ}.
\end{proof}

Based on existing computations, on their multireal values, and on the Abramovich--Bertram formulas satisfied by both $Q_{X,D}$ and $\V_D$, we conjecture that quadratic Gromov--Witten invariants of rational del Pezzo surfaces are given by Welschinger--Witt invariants.

\begin{conjecture} \label{conjGeneral}
	Supposing \cref{settingConj}, 
	we have
		\[
		Q_{X,D,k}(A_0) = \V_{X,D}(A_0). 
		\]
  \end{conjecture}

	We may rephrase the conjecture as stating
that whenever the quadratic Gomov--Witten invariant of a rational surface
$X$ is defined, it agrees with the corresponding
evaluation of $\V_{\n,\dbar}$.
\begin{exa}\label{exa:real pyt}
	It follows from \cref{lem:equaltors} that \cref{conjGeneral} holds if $\W(k)$ has no torsion. According to \cite[Theorem VIII.4.1]{lam05}, this is equivalent to the condition that $k$ if formally real (i.e. $-1$ is not a sum of square) and Pythagorean (i.e. every sum of squares is a square). Note that such a field has characteristic 0.
\end{exa}

	Note that \cref{conjGeneral} also encompasses the case of smooth quadrics in $\P^3_k$ with a rational point, 
	since the blow-up at a rational point of any such quadric is isomorphic to the blow-up $\P^2_k$ at reduced zero-dimensional scheme $\p$  of length 2.
	
  \begin{remark}
\cref{conjGeneral} would also imply that quadratic Gromov--Witten invariants of rational surfaces in
  positive characteristic are essentially determined by their behavior 
  in characteristic $0$. This is a property which is expected to hold
  by the construction of the positive characteristic invariants
  and observed in the computed examples. 
\end{remark}

Since the right hand side $\V_{X,D}$ of \cref{conjGeneral}
is the specialization of
a multivariable Witt invariant $\V_{\n,\dbar}$, we conjecture 
an equality of multivariable Witt invariants
for a suitable defined $Q_{\n, \dbar}$. 
Currently, there are three obstructions to making this precise:
\begin{enumerate}
	\item 
	  It is currently not known whether the quadratic Gromov--Witten invariant
		$Q_{X,D}$ is a \emph{deformation invariant}
		in the sense that it only depends on the étale $k$-algebras
		$(A_1, \dots, A_r) \in \Et_\n(k)$ but not on the 
		(generic) subschemes $\p_1, \dots, \p_r \in \P^2_k$ from
		\cref{settingConj}. Note that
		our conjecture implies deformation invariance in 
		\cref{settingConj}.
	\item
	  Given $(A_1, \dots, A_r) \in \Et_\n(k)$, 
		to construct $X$ as in \cref{settingConj}
		we need to find a \emph{general} point configuration
		in $\P^2_k$ with étale $k$-algebra $A_1 \times \dots \times A_r$. 
		This is always possible in characteristic $0$,
		but not necessarily in positive characteristic. 
	\item
	  Currently, the invariant $Q_{X,D}$ is only defined
		for (base changes of) surfaces $X$ defined over a perfect field $k$
		as in \cref{def:quadGW}. This corresponds to evaluating
		the right hand side $\V_{\n,\dbar}$ only in tuples
		$(A_0, A_1 \otimes K, \dots, A_r \otimes K)$
		where $A_0 \in \Et_{n_0}(K)$ but 
		$(A_1, \dots, A_r) \in \Et_\n(k)$
		(for some field extension $k \to K$).
\end{enumerate}
Therefore, up these limitations
with respect to the definition
of the quadratic invariants, 
we conjecture an equality 
of multivariable Witt invariants
$Q_{\n,\dbar} = \V_{\n,\dbar}$
for a suitable quadratic counterpart $Q_{\n,\dbar}$.
To illustrate the general idea and give some evidence
for the conjecture, 
in the following we prove 
the case of rational del Pezzo surfaces of degree at least $6$.
In this case, the first two obstructions
can be removed and the conjecture can 
be proven as explained now.

\begin{lemma} \label{lemdelPezzo6}
  Let $k$ be a perfect field of 
	characteristic not $2$ or $3$. 
	Fix $\eta \leq 4$ and $A \in \Et_\eta(k)$.
	Then there exist a closed immersion
	$\p: \Spec A \hookrightarrow \P^2_k$ such that
	the blow-up $X$ of $\P^2_k$ along $\p$
	is a del Pezzo surface. 
	Moreover, the surface $X$ is unique up to isomorphism over $k$.
\end{lemma}

\begin{proof}
	Let $\overline k$ be an algebraic closure of $k$. 
	The blow-up of $\P^2$ along $\p$ gives a del Pezzo surface if and only if the geometric
	points $\p(\overline k)$ are in general position in the sense
	that no three are collinear. 
	To construct such $\p$, choose $\p_0: \Spec A \hookrightarrow \P^1_k$
	and take $\p$ as the image of $\p_0$ under the Veronese map $(s:t) \mapsto (s^2:st:t^2)$.
	To show the resulting blow-ups are isomorphic, assume that $\p': \Spec A \hookrightarrow \P^2_k$ is another such point configuration. 
	Since $\vert \Spec A(\overline{k})\vert \leq 4$ there is a projective 
	coordinate change  $\phi \in \PGL_3(\overline k)$ such that  $\phi \circ (\p \otimes \overline{k}) = (\p' \otimes \overline{k})$. Moreover, fixing an additional $4 - \vert \eta \vert$ rational points on $\P^2_k$, we may assume $\phi$ is unique. Thus $\phi$ is Galois invariant and hence $\phi \in \PGL_3(k)$. 
\end{proof}

For $K \in \Fields$ and $\n \in \N^r$, let $\widetilde{\Et}_{\n}(K)$ denote the subset of $\Et_{\n}(K)$ consisting of elements of the form $(A_1 \otimes_k K, \ldots, A_r \otimes_k K)$ with $k$ a perfect subfield of $K$ such that there exists a closed immersion $\p:\Spec \prod_i A_i \hookrightarrow \P^2_k$ such that the blow-up of $\P^2_k$ along $\p$ is del Pezzo.
Note that $\widetilde \Et_\n$ is a subfunctor of $\Et_\n$, that is, the subsets are compatible with base change.
Suppose that quadratic Gromov--Witten invariants are deformation invariant
(for $\eta =n_1 + \cdots + n_r \leq 4$, this follows from \cref{lemdelPezzo6}). 
For $D \in \Z^{1+ \eta}$, $\n \in \N^r$, $\dbar \in \N^{r+1}$
as in \cref{settingConj}, 
we can then define the \emph{partially defined} Witt invariant $Q_{\n,\dbar}$
given by the collection of maps  
\begin{align*}
  Q_{\n,\dbar,K} :  \Et_{n_0} (K) \times \widetilde{\Et}_{\n}(K) &\to  W(K), \\
	               (A_0, \dots, A_r) &\mapsto Q_{X,D}(A_0),
\end{align*}
for all $K \in \Fields_{2,3}$, 
where $X$ denotes the blow-up of $\P^2_K$ in $A_1, \dots, A_r$ as in \cref{settingConj}. 
Note $Q_{\n,\dbar}$ indeed satisfies the Witt invariance
property in all variables, that is, for all $K \to L$ 
we have 
\begin{align*} 
    Q_{\n,\dbar,K}(A_0, \dots, A_r) \otimes L = Q_{\n,\dbar,L}(A_0\otimes L, \dots, A_r\otimes L)
		& & \forall \; (A_0, \dots, A_r) \in  \Et_{n_0} (K) \times \widetilde{\Et}_{\n}(K).
\end{align*}
This follows immediately from the construction in \cref{sec:constr quadGW}
and the fact that $X \otimes L$ is the blow-up of $\P^2_L$ in
$A_1\otimes L, \dots, A_r\otimes L$.
Analogous to \cref{defWelschingerWittInvariants}, we also write $Q_D$ for $Q_{\n, \dbar}$ where
$\n$ is the coarsest partition compatible with $D$. 

By \cref{thmWelschingerUniversal}, \cref{conjGeneral} is equivalent to the statement that quadratic Gromov--Witten invariants of rational del Pezzo surfaces are deformation invariant, and extend to a unique $\beta$-integral Witt invariant.

\begin{conjecture}[Equivalent to \cref{conjGeneral}]\label{conj2}
	  For $n_1+\cdots+n_r\le 6$ and any $\dbar \in \N^{r+1}$ with $n_0 \neq 5$, quadratic Gromov--Witten invariants are deformation invariant and
	the partially defined Witt invariant $Q_{\n,\dbar}$
	is the restriction of a $\beta$-integral Witt invariant from
	$ \Et_{n_0}  \times {\Et}_{\n}$ to  $\Et_{n_0} \times \widetilde{\Et}_{\n}$.
\end{conjecture}

\begin{remark}
	A consequence of \cref{conj2} can be informally stated as follows: quadratic Gromov--Witten invariants $Q_{X,D,k} \colon \Et_{n_0}(k) \to \W(k)$ 
	of a del Pezzo surface $X$ of degree a least 3 over $k\in\Fields_{2,3} $ only depend on the lattice $\Pic(X)(\overline{k})$ equipped with its intersection form and the action of $\Gal(\overline{k}:k)$. 
	See also \cite[Remark 1.4]{Bru18} for a similar remark regarding Welschinger invariants.
\end{remark}

Before proving \cref{thm:dP6}
we revisit the 
\enquote{classical} computation of $Q_{X,-K_X}$.
\begin{example} \label{exAnticanonical}
   In \cite[Example 1.4]{degree}, 
	it is shown that for $D = -K_X$ one has
	\[   
	\widehat Q_{X,D,k}( A_0) = - \chi^{\A^1}(X) + 1 + \Tr(A_0),
	\]
	where $\chi^{\A^1}(X) \in \WG(k)$ denotes the $\A^1$-Euler characteristic.
	The latter behaves similiar to the ordinary Euler characteristic.
	Hence if $X$ is the blow-up of $\P^2_k$ at $\p_1$ with finite étale $k$-algebra $A_1$
	of degree $n_1$, then
	\[  
	\chi^{\A^1}(X) = \chi^{\A^1}(\P^2_k) + \Tr(A_1)(\chi^{\A^1}(\P^1_k) -1)
	= 1 + (1+\Tr(A_1))h - \Tr(A_1).
	\]
	Therefore 
	\[
	Q_{X,D} (A_0) = \Tr(A_0) + \Tr(A_1),
	\]
	which agrees with 
	$\V_{(n_1),(3,1)}=\beta_{1,0}+\beta_{0,1} \in \Inv(n_0, n_1)$, the splitting
	of $\V_3 = \beta_1 \in \Inv(8)$, as predicted by \cref{conjGeneral}.
\end{example}

\section{The case of del Pezzo surfaces of degree at least 6}\label{sec:dP6}

This section is devoted to the proof of \cref{conjGeneral} (equivalently \cref{conj2} ) when $n_1+\cdots+n_r\le 3$. 
\begin{thm}\label{thm:dP6}
  For $n_1+\cdots+n_r\le 3$ and any $\dbar \in \N^{r+1}$,
	the partially defined Witt invariant $Q_{\n,\dbar}$
	is the restriction of $\V_{\n,\d}$ from
	$ \Et_{n_0}  \times {\Et}_{\n}$ to  $\Et_{n_0} \times \widetilde{\Et}_{\n}$.
\end{thm}

The proof of \cref{thm:dP6} goes in two steps. First,  thanks to the quadratic Abramovich--Bertram formula from  \cite{Brugalle-WickelgrenABQ},  we reduce to the study of $X$ a toric del Pezzo surface over $k$, that is, to the specialization 
$Q'_{\n, \dbar} := Q_{\n, \dbar}^{((0,\m))} \in \Inv_{2,3}(n_0)$
of $Q_{\n, \dbar}$ for $(A_1, \dots, A_r) = (k^{n_1}, \dots, k^{n_r})$. (We remind the reader the notation $\alpha^{(\i)}$ from \cref{defCuttingHoms} for a Witt invariant $\alpha$.)
Next, we prove  that $Q'_{\n, \dbar}$ is $\beta$-integral 
using the tropical computation of quadratic Gromov--Witten invariants from  \cite{PMPR-QuadraticallyEnrichedPlane}. 
Since our setup is slightly different from \cite{PMPR-QuadraticallyEnrichedPlane}, we start by recalling in \cref{sec:fd} the floor diagram algorithm to compute quadratic Gromov--Witten invariants of toric del Pezzo surfaces.
The proof of \cref{thm:dP6} is given in \cref{sec:proofDP6}.

\subsection{Floor diagrams}\label{sec:fd}
We briefly recall the definition of floor diagrams and their markings following \cite{BruMik09}, and define their multiquadratic multiplicities following \cite{PMPR-QuadraticallyEnrichedPlane}. We refer to these two papers for more details and examples.
A (finite, oriented) \emph{graph} $G$ is the data of a finite set $\ve(G)$, the \emph{vertices} of $G$, a subset $\ed(G)$ of  $\ve(G)\times \ve(G)$, the \emph{bounded (oriented) edges} of $G$, and two finite multisets $\ed^{-\infty}(G)$ and $\ed^{+\infty}(G)$ made of elements of $\ve(G)$, the \emph{sources} and \emph{sinks} of $G$, respectively. 
Elements of $\ed(G)\cup \ed^{\pm\infty}(G)$ are referred to \emph{edges} of $G$.
We consider a bounded edge $(v_1,v_2)$ as oriented from $v_1$ to $v_2$ and a source or a sink adjacent to the vertex $v$ as oriented towards or outwards from $v$, respectively. 
A graph is called a \emph{tree} if it is connected and $|\ve(G)|-|\ed(G)|=1$.
A \emph{weighted graph} is a graph $G$ together with a \emph{weight function}
$\omega:\ed(G)\rightarrow \N \setminus \{0\}$. 
In addition, 
we declare the weight of a all sources and sinks to be $1$. 
We define the \emph{divergence} $\div(v)$ of a vertex $v \in \ve(G)$ as the sum of the 
weights of all its incoming edges
minus the sum of the weights of all its outgoing edges.

\begin{defi}\label{def fd}
    Let $d_0>d_1\ge d_2\ge d_3\in \N$ such that $d_0-d_1-d_2\ge 0$. 
A  floor diagram
 $\D$  of class $(d_0,d_1,d_2,d_3)$
is the data of a weighted oriented tree
$G$ and a map $\theta : \ve(G)\to \{0,1\}$ which satisfy
the
  following conditions

\begin{itemize}
\item $|\ve(\D)|=d_0-d_1$;
\item there are exactly $d_0-d_2-d_3$ and $d_1$ edges in
  $\ed^{-\infty}(G)$ and $\ed^{+\infty}(G)$, respectively;
\item the function $\theta$ takes value $0$ exactly $d_0-d_1-d_3$ times, and value $1$ exactly $d_3$ times;
  \item the function $\theta+\div$ takes value $0$ exactly $d_0-d_1-d_2$ times, and value $1$ exactly $d_2$ times.
\end{itemize}
\end{defi}

\begin{remark}\label{rem:poly to surface}
In the terminology of \cite{BruMik09}, a floor diagram of class $(d_0,d_1,d_2,d_3)$ is a floor diagram of genus 0 and Newton polygon depicted in Figure \ref{fig:poly}. 
\begin{figure}[h!]
\centering
\includegraphics[height=3.5cm, angle=0]{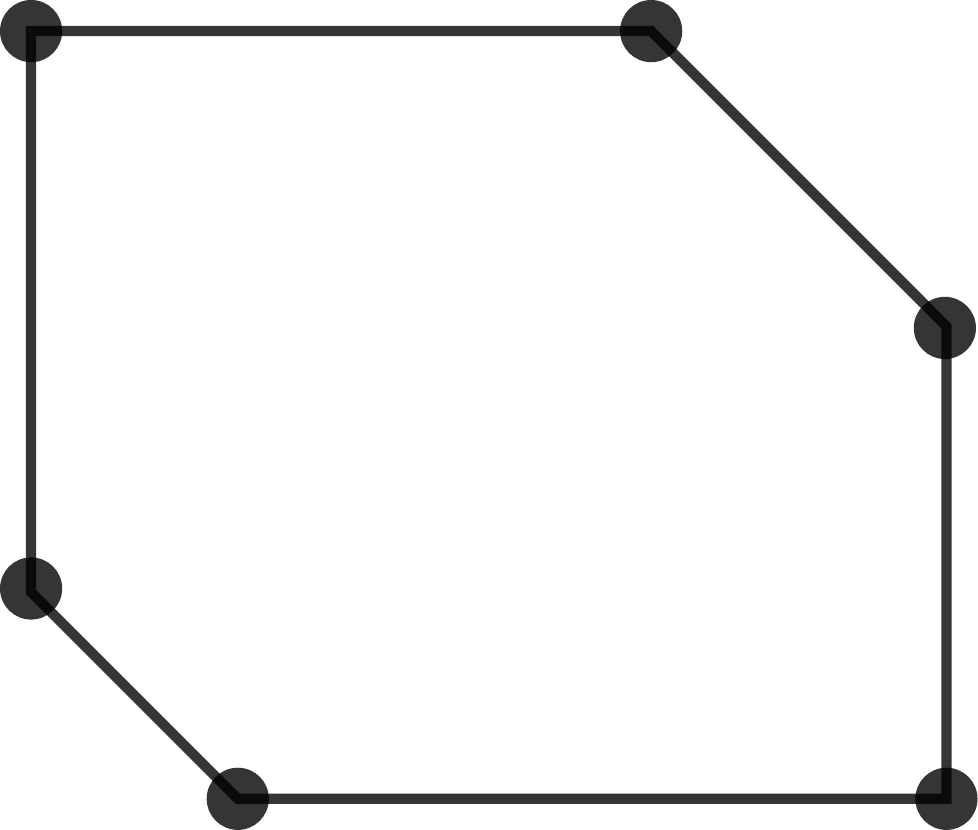}
\put(-150,25){$(0,d_3)$}
\put(-173,93){$(0,d_0-d_1)$}
\put(-105,-12){$(d_3,0)$}
\put(-30,93){$(d_1,d_0-d_1)$}
\put(-30,-12){$(d_0-d_2,0)$}
\put(3,57){$(d_0-d_2,d_2)$}
\caption{}
\label{fig:poly}
\end{figure}
This polygon corresponds to the blow-up $X = \P^2_{3,0}$ of $\P^2_\R$
in $3$ real points and the divisor $D = d_0 L - d_1 E_1- d_2 E_2 - d_3 E_3$. 
\end{remark}

Given a floor diagram $\D = (G, \theta)$, by abuse of notation
we also denote by $\D$ the disjoint
union of $\ve(G)$, $\ed(G)$, $\ed^{+\infty}(G)$ and $\ed^{-\infty}(G)$.
This is a partially ordered set of size $n = (d_0-d_1) + (d_0 - d_1 -1) + d_1 + (d_0 - d_2 - d_3) =
3d_0-d_1-d_2-d_3-1$, where $x \leq y$ if there is an oriented path in $G$ connecting
$x$ to $y$.

\begin{defi}\label{def marking}
Fix $s, r \in \N$ such that $2s +r = n$. A $s$-marking of a floor diagram $\D$ of class $(d_0,d_1,d_2,d_3)$  is an monotone map 
$\varphi  : \D \to \{1,\dots,s+r\}$ such that for all $j \in \{1, \dots, r\}$ 
\[
  |\varphi^{-1}(j)| = 
	\begin{cases}
	  2 & \text{if } j \leq s, \\
		1 & \text{if } j > s. 
	\end{cases}
\]
\end{defi}

A floor diagram enhanced with a $s$-marking is called a \textit{$s$-marked floor
  diagram}.

\begin{defi}
Two $s$-marked floor diagrams $(\D,\varphi)$ and $(\D',\varphi')$ are called
\emph{equivalent} if there exists an isomorphism of oriented trees
$\psi : \D\to\D'$ such that $w=w'\circ \psi$, $\theta=\theta'\circ
\psi$, and $\varphi=\varphi' \circ\psi$.
\end{defi}

Given a $s$-marked floor diagram $(\D, \varphi)$, 
we use the shorthand $P_j = \varphi^{-1}(j)$.
We next introduce a Witt-invariant multiplicity
for $(\D,\varphi)$ which only depends on its equivalence class.

By the monotonicity of $\varphi$, for $j \in \{1, \dots, s\}$
the pair $P_j$ consists either of a vertex and edge adjacent to each other
or two incomparable elements in $\D$. We denote by $V$ and $\Im$ the subsets
of indices of the first and second type, respectively. 
For $j \in V$, we set $\omega(j) = \omega(e)$ where $e$ is the edge in $P_j$. 

An automorphisms $\psi \in \Aut(\D, \varphi)$ necessarily satisfies $\psi(P_j) = P_j$
for all $0 \leq j \leq s+r$. Hence it is completely determined by the subset 
$J \in \{1, \dots, s\}$ labelling those $P_j$ whose elements get exchanged by $\psi$. 
We can therefore regard $\Aut(\D, \varphi)$ as a subgroup of the power set of 
$\{1, \dots, s\}$ equipped with the
operation $J \triangle J = (J \cup J')\setminus (J\cap J')$.
We denote by $\T$ the inclusion-minimal non-empty elements
in this subgroup. 
It follows that the elements of $\T$ are pairwise disjoint and
that $\Aut(\D, \varphi) \cong (\Z/2\Z)^{\T}$ canonically.

Given $T \in \T$, the associated automorphism $\psi_T$ exchanges two isomorphic (connected) 
subtrees of $G$ while keeping
the rest of $G$ fixed. We therefore call $T$ a twin tree
and $j \in T$ a twin edge/vertex/source/sink 
whenever $P_j$ consists of edges/vertices/sources/sinks of $G$.
We denote by $\omega^{\pm\infty}_T$ the number of twin sinks and twin sources in $T$, 
respectively. 
Moreover, there is a unique twin edge $j \in T$ such that 
the edges in $P_j$ are adjacent to the same vertex $v$ and $\varphi(v) \notin T$;
we call $j$ the \emph{twin root} of $T$ and set 
$\omega^\infty_T = \omega(j) + \omega^{+\infty}_T + \omega^{-\infty}_T$ where
$\omega(j) = \omega(e)$ for $e \in P_j$.

Setting $C = \Im \setminus \bigcup_{T\in \T} T$, 
we obtain a partition $\{1, \dots, s\} = (\bigsqcup_{T \in \T} T) \sqcup C \sqcup V$. 
This partition depends only on the equivalence class 
of $(\D, \varphi)$.

We denote by $t_1, \dots, t_s$ 
the elements of $\Inv(\Sq_s)$
given by
$t_j(\delta) = \Tr \E_{\delta_j} = \lra{2,2\delta_j} \in \W(K)$
for any $K \in \Fields$ and $\delta = (\delta_1,\ldots,\delta_s) \in \Sq_s(K)$.
For $\omega \in \N$, $j \in \{1, \dots, s\}$ and $J \subset \{1, \dots, s\}$, we set
\begin{align*}
  t_J &= \prod_{j\in J} t_j, & u_j &= 2 \lra{1} - t_j, & u_J &= \prod_{j\in J} u_j, 
\end{align*}
\begin{align*}
  [\omega]_j  :=
	  \begin{cases}
		  \lra{1} +\frac{\omega-1}{2} u_j& \text{ if $\omega$ is odd,} \\
		  \frac{\omega}{2}u_j & \text{ if $\omega$ is even.} 
		\end{cases}
\end{align*}
We note that $t_J, u_J, [\omega]_j$ lie in $\Z[t_1, \dots, t_s]$
and satisfy the following simple rules:
$t_j^2 = 2 t_j$, $u_j^2 = 2 u_j$ $t_j u_j = 0$, and $[\omega]_j \cdot [\omega']_j = [\omega \omega']_j$.
Furthermore we have $t_J \pm u_J \in 2\Z[t_1, \dots, t_s]$, and
we denote by  $(t_J \pm u_J)/2$ the unique elements in $\Z[t_1, \dots, t_s]$  which give 
$t_J \pm u_J$ when multiplied by $2$.

\begin{definition}\label{def:mult}
  Let $(\D, \varphi)$ be a $s$-marked floor diagram 
	and define $V$, $C$ and $\T$ as above. 
    Then $(\D, \varphi)$ 
		is \emph{essential} if for any bounded edge $e$
		of even weight we have $\varphi(e) \in T$ for some $T \in \T$
		(that is, even weights only occur in the twin trees). 
		
	If $(\D, \varphi)$ is essential, we define its 
	\emph{quadratic multiplicity} $\mu(\D, \varphi) \in \Inv(\Sq_s)$ as 
	  \[
      \mu(\D, \varphi) = t_C
			  \prod_{\substack{e \in \ed(G)\\ \varphi(e) \leq s}} [\omega(e)]_{\varphi(e)}
			  \prod_{T \in \T}  \frac{t_T + (-1)^{\omega^\infty_T} u_T}{2}.
		\]
\end{definition}

Importantly, the discussion above implies that $\mu(\D, \varphi) = \mu(\D', \varphi')$
if $(\D, \varphi)$ and $(\D', \varphi')$ are equivalent. 
Note also that the second product can be restricted to $\varphi(e) \in \{1, \dots, s\} \setminus C$
since $t_j [\omega]_j = t_j$ when $\omega$ is odd. 

We now give the following generalization of \cite[Theorem 10.13]{PMPR-QuadraticallyEnrichedPlane}, with the assumption of sufficiently large
characteristic  in \emph{loc.\ cit.} removed. 
Let us summarize the setting:
We set $S = \{2,3\}$. 
We fix $\dbar \in \N^4$ such that $d_0>d_1\ge d_2\ge d_3\in \N$ and $d_0-d_1-d_2\ge 0$
and set $n = 3d_0-d_1-d_2-d_3-1$ and $m = \lfloor n/2 \rfloor$.
We denote by $Q'_{\dbar} \in \Inv_S(\Sq_m)$ the restriction of $Q_{(1,1,1), \dbar} \in \Inv_S(n,1,1,1) = \Inv_S(n)$ to multi-quadratic algebras.
Furthermore, we fix $0 \leq s \leq m$
and denote by $Q'^{(m-s)}_{\dbar} \in \Inv_S(\Sq_s)$ the invariant
given by $Q'^{(m-s)}_{\dbar}(\delta_1, \dots, \delta_s) = Q'_{\dbar}(\delta_1, \dots, \delta_s, 1, \dots, 1)$.

\begin{thm}\label{thm:FD}
    In the setting just described, we have
\[
 Q'^{(m-s)}_{\dbar} = \sum_{[(\D,\varphi)]} \mu(\D,\varphi) \in \Inv_S(\Sq_s),
\]
  where the sum ranges over all equivalence classes of essential  $s$-marked floor diagrams $(\D,\varphi)$ of class $(d_0,d_1,d_2,d_3)$. 
\end{thm}

Before proving \cref{thm:FD}, let us explain the main consequence of interest to us.

\begin{corollary} \label{thm:qGWtoric integral}
  Given $\dbar \in \N^4$,
	the quadratic invariant $Q_{(1,1,1), \dbar}$
	is $\beta$-integral, that is, $Q_{(1,1,1), \dbar} \in \UInv_S(n)$.
\end{corollary}

\begin{proof}
    If $d_0=1$, the result is trivial. For $d_0\ge 2$, by symmetry and since otherwise $Q_{(1,1,1), \dbar}$ is trivially $0$,
		we can restrict to the case 
		$d_0>d_1\ge d_2\ge d_3\in \N$ and 
		$d_0-d_1-d_2\ge 0$.
		
		Since clearly $\mu(\D, \varphi) \in \Z[t_1, \dots, t_m]$
		for any essential marked floor diagram $(\D, \varphi)$, 
		it follows from \cref{thm:FD} 
		that $Q'_{\dbar} \in \Z[t_1, \dots, t_m] \subset \Inv_S(\Sq_m)$.
		Since $t_j^2 = 2t_j$, this is equivalent to being an integral linear combination
		of the $t_J$, $J \subset \{1, \dots, m\}$. The latter form 
		a $\W(\PP\setminus S)$-basis for $\Inv_S(\Sq_m)$
		(the canonical isomorphism $\Inv_S(\Sq_m) \cong \Inv_S(2,\dots, 2)$ with $m$ repetitions
		identifies the $t_J$ with the $\beta$-basis). 
		Finally, since the restriction map $\UInv_S(n) \to \Inv_S(\Sq_m)$
		is injective and 
		sends $\beta_j$ to $\beta'_j = \sum_{J : |J| = j} t_J$,
		see \cref{thmMultiquadraticFree} and \cref{thmBetaBasis1var}, 
		we conclude that $Q'_{\dbar} \in \Z[t_1, \dots, t_m]$
		implies $Q_{(1,1,1), \dbar} \in \Z[\beta_0, \dots, \beta_m] = \UInv_S(n)$.
	\end{proof}

In order to prove \ref{thm:FD}, we start with rewriting
the factor corresponding to a twin tree $T$ in accordance with the presentation given in \cite[Definition 10.11]{PMPR-QuadraticallyEnrichedPlane}. In next lemma, we use the notation $a\equiv b$ for two integers $a$ and $b$ equal modulo 2.

\begin{lemma} \label{lemTreeMult}
  Let $(\D, \varphi)$ be a essential $s$-marked floor diagram 
	and $T \in \T$ a twin tree. Then 
	\begin{equation}\label{eq:mut}
	  \mu(T) := \frac{t_T + (-1)^{\omega^\infty_T} u_T}{2}
		\prod_{\substack{e \in \ed(G) \\ \varphi(e) \in T}} [\omega(e)]_{\varphi(e)}
		=
		\lra{2^{|T| + 1}}
		\prod_{\substack{j \in T \\ \text{twin edge}}} [\omega(j)^2]_{j}
		\sum_{\substack{J\subset T\\ |J| \equiv \omega^\infty_T}}
		\prod_{j\in J} \lra{\delta_j}.
	\end{equation}
\end{lemma}

\begin{proof}
  First, note that
	\[
	  \prod_{\substack{e \in \ed(G) \\ \varphi(e) \in T}} [\omega(e)]_{\varphi(e)}
		= 
		\prod_{\substack{j \in T \\ \text{twin edge}}} [\omega(j)^2]_{j}
	\]
	because $\omega(e) = \omega(e')$ for $\{e,e'\}=P_j$ with $j$ a twin edge, and 
	$[\omega]^2_j = [\omega^2]_j$ for all $\omega \in \N$. 
  Since $\lra{\delta_j} = \lra{2}(t_j-\lra{2})$ (we consider $\lra{\delta_j}$
	as the obvious element of $\Inv(\Sq_m)$ here), we 
	can rewrite the remaining parts of the right hand side as
	\begin{align*}
		\lra{2^{|T| + 1}}
		  \sum_{\substack{J\subset T\\ |J| \equiv \omega^\infty_T}}
		  \prod_{j\in J} \lra{\delta_j} 
		&= 
		  \lra{2^{\omega^\infty_T + |T| + 1}}
		  \sum_{\substack{J\subset T\\ |J| \equiv \omega^\infty_T}}
		  \prod_{j\in J} (t_j-\lra{2}) \\
		&=
		  \lra{2^{\omega^\infty_T + |T| + 1}}
		  \sum_{I\subset T} t_I
			\left(
			\sum_{\substack{I \subset J\subset T\\ |J|\equiv \omega^\infty_T}}
			\lra{(-2)^{|J|-|I|}} 
			\right) \\
		&=
		  (-1)^{\omega_T^\infty} \sum_{I\subset T} (-1)^{|I|} t_I \lra{2^{|T|-|I|-1}} 
			\left|\left\{
			\begin{array}{c}
			  I \subset J\subset T\\ |J|\equiv \omega^\infty_T 
			\end{array}
			\right\}\right|.
	\end{align*}
	We now distinguish two cases. If $I \subsetneq T$, 
	then the cardinality appearing as the last term
	is $2^{|T| - |I| - 1}$ (one half of all subsets of $T \setminus I$). 
	If $I = T$, then the cardinality is either $1$ (if $\omega^\infty_T \equiv |T|$)
	or $0$ (if $\omega^\infty_T \equiv |T|+1$). 
	Denoting $t_T^? = t_T$ or $t_T^? = 0$ accordingly and using 
	$2\lra{2} = 2$, we can therefore deduce
	\begin{align*} 
		\lra{2^{|T| + 1}}
		  \sum_{\substack{J\subset T\\ |J| \equiv \omega^\infty_T (2)}}
		  \prod_{j\in J} \lra{\delta_j} 
		&= 
		  (-1)^{\omega_T^\infty}
			\Big(
			(-1)^{|T|} \lra{2}  t_T^?  +
			\sum_{I\subsetneq T} 2^{|T|-|I|-1} (-1)^{|I|} t_I 
			\Big) \\
		&= 
		  (-1)^{\omega_T^\infty}
			\Big(
			(-1)^{|T|} \lra{2}  t_T^?  +
			\frac{u_T - (-1)^{|T|} t_T}{2}
			\Big) \\
		&= 
		  (-1)^{\omega_T^\infty+|T|+1}
			\Big(
			- \lra{2}  t_T^?  +
			\frac{t_T + (-1)^{|T|+1}u_T}{2}
			\Big).
	\end{align*}
	If $\omega^\infty_T \equiv |T| + 1$  the claim follows, 
	so let us suppose $\omega^\infty_T \equiv |T|$. 
	In this case, the fact that $T$ is a tree implies
	that the weight $\omega$ of the pair of edges forming the root edge $j \in T$ 
	has to be even, say $\omega = 2k$.  
	Therefore, its contribution to the edge product is 
	$[4k^2]_j = 2 k^2 u_j$.
	Again using that $2\lra{2} = 2$ we find that the total left hand side of \eqref{eq:mut}
	gives 
	\begin{multline*}
		  (-1)^{\omega_T^\infty+|T|+1}
			\prod_{\substack{j \in T \\ \text{twin edge}}} [\omega(x_{2j})^2]_{j}
			\Big(
			- \lra{2}  t_T  +
			\frac{t_t + (-1)^{|T|+1}u_T}{2}
			\Big) 
			\\ = 
			\prod_{\substack{e \in \ed(G) \\ \varphi(e) \in T}} [\omega(e)]_{\varphi(e)}
			\Big(
			  t_T  -
			\frac{t_T + (-1)^{|T|+1}u_T}{2}
			\Big) 
			= \frac{t_T + (-1)^{\omega^\infty_T} u_T}{2}
		\prod_{\substack{e \in \ed(G) \\ \varphi(e) \in T}} [\omega(e)]_{\varphi(e)}
	\end{multline*}
	which proves the claim. 
\end{proof}

\begin{proof}[Proof of \cref{thm:FD}]
    Since both sides of 
		the equation in \cref{thm:FD} are unramified away from $S = \{2,3\}$,
		it suffices to prove the equality for characteristic 0, that is, 
		for $K \in \Fields_\Q$ and $\delta = (\delta_1, \dots, \delta_s) \in \Sq_s(K)$.
		Now, it follows from  \cite[Theorem 10.13]{PMPR-QuadraticallyEnrichedPlane}
		(together with 
		\cite[Proposition 5.4]{PMPR-QuadraticallyEnrichedPlane})
		that $Q'_{\dbar}$ can be computed 
		as a weighted count over equivalence classes of 
		$s$-marked floor diagrams (floor diagrams with $s$ merged points in \emph{loc.\ cit.})
		and it remains to compare multiplicities. 
Let us denote by $\mu'(\D, \varphi) \in \W(K)$ 
the image of the multiplicity of $(\D,\varphi)$ from \cite[Definition 10.11]{PMPR-QuadraticallyEnrichedPlane} 
under the map $\WG(K) \to \W(K)$. 
 One has $\mu'(\D, \varphi) = 0$ 
for  a non-essential $(\D,\varphi)$ and 
$\mu'(\D, \varphi) = \mu(\D, \varphi)(\delta)$ 
for essential $(\D,\varphi)$
using \cref{lemTreeMult}.
We conclude that the weighted count
mentioned above is equal to our weighted count
in \cref{thm:FD}. This proves the claim. 
\end{proof}

\begin{example}
  Note that $\mu(T)$ can be rewitten as follows. We set $\omega(T) :=
	 \prod_{e} \omega(e)$ where the product runs through all
	 $e \in \ed(G)$ such that $\varphi(e) \in T$.
	Then, using $t_j [\omega]_j = 0$ for $\omega$ even and $u_j [\omega]_j = \omega u_j$,
	one easily checks 
	\[
	  \mu(T) = 
		  \begin{cases}
			  (t_T + (-1)^{\omega_T^\infty} \omega(T) u_T)/2 & \text{if $\omega(T)$ odd}, \\
			  (-1)^{\omega_T^\infty} \omega(T) u_T/2 & \text{if $\omega(T)$ even}.				
			\end{cases}
	\]
\end{example}

\begin{exa}
   When $K=\R$, 
	the multiplicity $\mu(\D,\varphi)(-1, \dots, -1) \in \W(\R)\cong\Z$ equals  
	the sum of the corresponding real multiplicity in the sense of 
	\cite[Definition 3.8]{BruMik09} over all equivalence classes of
	conventionally marked floor diagrams 
	associated to $(\D,\varphi)$. 
	Therefore, \cite[Theorem 3.9]{BruMik09} 
	confirms that the multireal values of $Q_{\dbar}$
	are given by the Welschinger invariants $\Wel_{\P^2_{3,0}}(\dbar;s)$.
\end{exa}

\begin{remark} 
  \cref{thm:FD} also holds when the notion of $s$-marking is 
	modified as follows: Given a fixed subset $A \subset \{1, \dots, s+r\}$
	of size $s$,
	we now require that a $s$-marking $\varphi \colon \D \to \{1, \dots, s+r\}$
	satisfies $|\varphi^{-1}(j)| =2$ for $j \in A$ and $|\varphi^{-1}(j)| =1$ otherwise.
	As in the real case, the equality of the counts for different choices of $A$
	is not very transparent from a combinatorial point of view. 
\end{remark}

\begin{remark} 
  Passing from $s$ to $s+1$ 
	can be understood quite easily 
	in our setting, c.f.\ \cite[Section 5]{PMPR-QuadraticallyEnrichedPlane}.
	
	Fix a essential floor diagram $\D$. 
	First, note that the map
	$\{1, \dots, s+r\} \to \{1, \dots, s+r-1\}$
	which sends $s+1, s+2 \mapsto s+1$ and is otherwise 
	strictly monotone transforms an $s$-marking $\varphi$
	into an $s+1$-marking denoted by $\Phi(\varphi)$.
	Moreover, $\Phi$ compatible with equivalence on both sides. 
	Therefore, given an $s+1$-marking $\varphi$, it remains
	to show that $\mu(\D, \varphi)$ evaluated in $\delta_{s+1} = 1$
	is equal to the sum of $\mu(\D, \varphi)$ taken over
	the essential $\varphi$ in $\Phi^{-1}(\varphi)$ modulo equivalence. 
	We denote by $V$, $C$ and $\T$ the partition of $\{1, \dots, s+1\}$
	given by $\varphi$. 
	There are three cases:
	\begin{enumerate}
		\item 
		  If $s+1 \in V$: In this case, there is a single $s$-marking $\varphi'$ such 
			that $\Phi(\varphi') = \varphi$ since the two elements
			in $\PP_{s+1}$ are comparable in $G$. 
			The partition of $\varphi'$ only changes in $V' = V \setminus \{s+1\}$.
			Finally, since $[\omega]_{s+1}(\dots, 1) = 1$ for $\omega$ odd, 
			we get $\mu(\D, \varphi)(\dots, 1) = \mu(\D, \varphi')$.
		\item
		  If $s+1 \in C$: In this case, there are two non-equivalent $s$-markings
			$\varphi_1$ and $\varphi_2$ over $\varphi$ (which of course only differ
			by exchanging the labels $s+1$ and $s+2$). 
			They are of equal multiplicity, have the same partition except for $C' = C \setminus \{s+1\}$
			and since $t_{s+1}(\dots, 1) = 2$ we get $\mu(\D, \varphi)(\dots, 1) = \mu(\D, \varphi_1) +
			 \mu(\D, \varphi_2)$.
		\item 
		  If $s+1 \in T$ for some $T \in \T$: 
			Again, there are two  $s$-markings
			$\varphi_1$ and $\varphi_2$ over $\varphi$ 
			but they are equivalent via the automorphism $\psi_T$ of $\D$.
			The partition for $\varphi_1$ (and $\varphi_2$) is given
			by $\T' = \T \setminus \{T\}$ and $C' = C \cup T \setminus \{s+1\}$.
			Since $u_T(\dots, 1) = 0$, the factor for $T$ in the third product
			of $\mu(\D, \varphi)(\dots, 1)$ simplifies to $t_{T \setminus \{s+1\}}$.
			Moreover, for $j \leq s$ we have $[\omega]_{s+1}(\dots, 1) = 0$ or $1$ and 
			$[\omega]_j t_j = 0$ or $t_j$ according
			to whether $\omega$ is even or odd.
			It follows that $\mu(\D, \varphi)(\dots, 1) = 0$
			if $T$ contains a twin edge of even weight, and
			$\mu(\D, \varphi)(\dots, 1) = \mu(\D, \varphi_1)$
			otherwise. 
	\end{enumerate}
\end{remark}

\subsection{Proof of \cref{thm:dP6}}\label{sec:proofDP6}

 The quadratic Abramovich--Bertram formula \cite[Theorem 6.6]{Brugalle-WickelgrenABQ} generalizes to the following case, as indicated in \cite[Remark 6.7]{Brugalle-WickelgrenABQ}. We include a proof for completeness. Let $k\in\Fields_{2,3}$ be a perfect field and let $\delta$ be in $k^*$. Let $X_{\E_{\delta} \times k}$ denote the del Pezzo blow-up of $\P^2_k$ along a subscheme isomorphic to $\Spec (\E_{\delta} \times k)$ from \cref{lemdelPezzo6}.

\begin{lemma}\label{lm:generalizeABQ6.6}
We have \[
Q_{X_{\E_\delta\times k},(d_0,d_1,d_1,d_3)} = Q_{X_{k^3},(d_0,d_1,d_1,d_3)} + (2\lra{1}-\Tr(\E_\delta))\sum_{\ell\ge 1}(-1)^{\ell}Q_{X_{k^3},(d_0,d_1-\ell,d_1+\ell,d_3)}.
\]
\end{lemma}

\begin{proof}
First assume that the characteristic of $k$ is $0$. Let $x,y,z$ denote homogeneous coordinates on the projective space $\P^2_{k[[t]]}$. Let $\X \to \Spec k[[t]]$ be the blow-up of $\P^2_{k[[t]]}$ along the disjoint union of the closed subschemes determined by the ideals $(y,x^2-tz^2)$ and $(z,x-y)$. Then $\X$ is a $1$-nodal Lefschetz fibration of del Pezzo surfaces satisfying the hypotheses of \cite[Corollary 5.33]{Brugalle-WickelgrenABQ}. For any $a$ in $k^*$, let $\Sigma(a)$ denote the $k((t))$-scheme given by generic fiber of the pullback of $\X$ by $t \mapsto a t^2$. By the quadratic Abramovich--Bertram formula \cite[Corollary 5.33]{Brugalle-WickelgrenABQ}, we have 
\[
Q_{\Sigma(\delta),(d_0,d_1,d_1,d_3)} = Q_{\Sigma(1),(d_0,d_1,d_1,d_3)} + (2\lra{1}-\Tr(\E_\delta))\sum_{\ell\ge 1}(-1)^{\ell}Q_{\Sigma(1),(d_0,d_1-\ell,d_1+\ell,d_3)}.
\] Changing variables $z \mapsto tz$ shows that $\Sigma(\delta)$ is the basechange of $X_{\E_\delta\times k}$ and $\Sigma(1)$ is the basechange of $X_{\E_1 \times k}$, showing the lemma.

For $k$ perfect of positive characteristic, choose a complete discrete valuation ring $R$, with fraction field $K$. Since $\Sq{R} = \Sq(k)$, we may view $\delta$ as an element $\tilde{\delta}$ of $\Sq(R)$. As in the proof of \cref{propQuadraticUnramified}, we can choose an $\E_{\tilde{\delta}} \times R$ point of $\P^2_R$ lifting the $\E_{\delta} \times k$ point of $\P^2_k$ determining the blow-up $X_{\E_\delta\times k}$. The blow-up $\Bl_{\E_{\tilde{\delta}} \times R} \P^2_R$ is a del Pezzo surface with generic fiber $X_{\E_{\tilde{\delta}} \times K}$. By \cite[Lemma 9.3(ii)]{KLSW-relor}, we can lift $D$ to a relative Cartier divisor on $\Bl_{\E_{\tilde{\delta}} \times R} \P^2_R$, which thus determines a Cartier divisor $D_K$ on the general fiber $X_{\E_{\tilde{\delta}} \times K}$. We have
\[
n_0 = -K_{X_{\E_\delta\times k}} \cdot D -1 = -K_{X_{\E_{\tilde{\delta}}\times K}} \cdot D_K -1.
\] Choose $A_0$ in $\Et_{n_0}(k)$. The map $\Et_{n_0}(k) \cong \Et_{n_0}(R) \to  \Et_{n_0}(K) $ sends $A_0$ to an \'etale $K$ algebra we will denote by $A_0^K$. By \cite[Section 5.2]{degree} and the degree of \cite[Section 2.4]{degree}, we have 
\[
\widehat Q_{X_{\E_\delta\times k},D}(A_0) = \widehat Q_{X_{\E_{\tilde{\delta}} \times K},D_K}(A^K_0).
\] Since the lemma holds in characteristic $0$, it follows that the lemma holds for $k$ perfect of positive characteristic as well.
\end{proof}

We now prove \cref{thm:dP6}. By \cref{lemdelPezzo6} and \cref{thmQuadraticWitt}, $Q_{\n, \dbar}$ is a well-defined Witt invariant over any field of characteristic $0$ for $|\n| \leq 3$.

 We first will prove that $Q_{\n, \dbar,\Q}$ is $\beta$-integral, from which the analogous claim follows over any field of characteristic $0$,  since the Witt invariant $Q_{\n, \dbar,\Q}$ over $\Q$ determines the corresponding Witt invariants over any field of characteristic $0$ by base change. Since $Q_{\n, \dbar}$ is invariant under adding an extra $1$ to $\n$ and an extra $0$ (if $|\n| \leq 2$), we can restrict our attention to the case $|\n| = 3$. There are $3$ possibilities: $\n = (1,1,1)$, $\n = (2,1)$, $\n = (3)$. By \cref{remOddN}, the map $\Inv(n, 3) \to \Inv(n,2,1)$ is injective and preserves $\beta$-integrality. Hence, it only remains to show that $Q_{(2,1),\dbar,\Q}$ is $\beta$-integral. Choose now $\n=(2,1)$ and $\dbar=(d_0,d_1,d_2)$. By the quadratic Abramovich--Bertram formula, as in Lemma~\ref{lm:generalizeABQ6.6}
\[
Q_{X_{\E_\delta\times k},(d_0,d_1, d_1,d_3)} = Q_{X_{k^3},(d_0,d_1,d_1,d_3)} + (2\lra{1}-\Tr(\E_\delta))\sum_{\ell\ge 1}(-1)^{\ell}Q_{X_{k^3},(d_0,d_1-\ell,d_1+\ell,d_3)}.
\]
In other words
\[
  Q_{(2,1),(d_0,d_1,d_3)} = Q_{(1,1,1),(d_0,d_1,d_1,d_3)} + (2\lra{1}-\beta_1(\E_\delta))\sum_{\ell\ge 1}(-1)^{\ell}Q_{(1,1,1),(d_0,d_1-\ell,d_1+\ell,d_3)}
\]
over $\Q$. Hence, since $Q_{(1,1,1),\dbar'}$ is $\beta$-integral for any $\dbar'$, so is  $Q_{(2,1),\dbar}$ over $\Q$.

We thus have shown that $Q_{\n, \dbar,\Q}$ is $\beta$-integral. By \cref{thmUnramifiedQuasiIntegral}, the invariant $Q_{\n, \dbar,\Q}$ thus extends to a unique well-defined Witt invariant, unramified away from $S= \{2,3\}$ for $|\n| \leq 3$, which by \cref{thmWelschingerUniversal} must equal $\V_{\n,\d}$. It remains to show that whenever $Q_{\n, \dbar}$ is defined, it is equal to $\V_{\n, \dbar}$. The value of $Q_{\n, \dbar}(A_1,\ldots,A_r)$ is defined when $A_i \cong A'_i \otimes_{\kappa_0} \kappa$ for some $A_i'$ over a perfect field $\kappa_0$ of characteristic not $2$ or $3$. By \cref{lemdelPezzo6}, we may choose a general configuration of points $\Spec \prod_i A_i' \to \P^2$. By \cite[Proposition 3.1.4]{Kedlaya-Swan-conductorsI}, we may choose $R_0 \to R$ a map of complete discrete valuation rings whose associated map of residue fields is $\kappa_0 \to \kappa$ and 
whose fraction fields $K_0 \to K$ 
are of characteristic $0$. 
Since $\Et_{\n}(R_0) \cong \Et_{\n}(\kappa_0)$, we may choose finite \'etale exentions $\tilde{A}'_i$ of $R_0$ pulling back to $A'_i$. Because $\P^2$ is smooth, we may choose a general configuration $\Spec \prod_i \tilde{A}'_i \to \P^2_{R_0}$ lifting the general configuration over $\kappa_0$, as in the proof of \cref{propQuadraticUnramified}. In particular, the blow-up $X=\Bl_{\Spec \prod_i \tilde{A}'_i }\P^2_{R_0}$ of $ \P^2_{R_0}$ at $\Spec \prod_i \tilde{A}'_i $ is a del Pezzo surface. 
Let $D$ be the Cartier divisor on $X$ corresponding to $\dbar$. Note that since $n_1 + \ldots + n_r \leq 3$, Hypothesis~\ref{deghyp1-part} is satisfied. By \cref{propQuadraticUnramified} it follows that we have the commutative diagram
\[\begin{tikzcd}
	{\Et_n(K)} & {\Et_n(R)} & {\Et_n(\kappa)} \\
	{W(K)} & {W(R)} & {W(\kappa).}
	\arrow["{Q_{X_K,D_K}}"', from=1-1, to=2-1]
	\arrow[from=1-2, to=1-1]
	\arrow["\cong"{marking, allow upside down}, draw=none, from=1-2, to=1-3]
	\arrow["{Q_{X_\kappa,D_\kappa}}", from=1-3, to=2-3]
	\arrow[from=2-2, to=2-1]
	\arrow["\cong"{marking, allow upside down}, draw=none, from=2-2, to=2-3]
\end{tikzcd}\] Since $K$ is characteristic $0$, we have that $Q_{X_K,D_K} = \V_{\n, \dbar, K}$. Since $\W(R) \to \W(K)$ is injective and $ \V_{\n, \dbar}$ is unramified, it follows that $Q_{X_{\kappa},D_{\kappa}} = \V_{\n, \dbar, \kappa}$. This proves the theorem. 
\hfill$\square$

\appendix
\crefalias{section}{appendix}
\section{Change of the basis between  $\lambda$ and $\beta$}\label{sec:basechange}

In this section, we discuss in more detail the change of basis for the bases  $(\lambda_i)$ and $(\beta)_i$ introduced 
in \cref{secWittInvariants}.

We denote the multichoose coefficients (choosing $k$ among $n$ with repetitions allowed) by
\[
  \mbinom{n}{k} = \binom{n+k-1}{k}
\]
and define the two following sequences of integers.
\begin{align*}
  \gamma_{i,k}^m &=
	  \sum_{j=k}^{i}(-1)^{j-k} \mbinom{m}{j-k}  {m-j \choose i-j}, &
  \delta_{i,k}^m &= 
	  \sum_{j=k}^{i}(-1)^{j-k} {m-k \choose j-k} {m\choose i-j} . 
\end{align*}

\begin{prop}\label{prop:beta basis}
  Let $n\ge 1$ and  $m=\lfloor \frac{n}{2}\rfloor$. 
	The bases $(\beta_i^n)_{0\le i\le  m}$ and $(\lambda_i^n)_{0\le i\le  m}$ of 
	$\Inv(n)$ are related as follows. 
	\begin{align*}
		n = 2m &&
		  \beta_i^n &= \sum_{k=0}^i \lra{2^{i-k}} \gamma_{i,k}^m  \lambda_k^n, &
			\lambda_i^n &= \sum_{k=0}^i \lra{2^{i-k}} \delta_{i,k}^m  \beta_k^n, \\
		n = 2m+1 &&
		  \beta_i^n &= \sum_{k=0}^i \left(\sum_{u=k}^i (-1)^{u-k} \lra{2^{i-u}} \gamma_{i,u}^m  \right)\lambda_k^n, &
			\lambda_i^n &= \sum_{k=0}^i \lra{2^{i-k}} (\delta_{i,k}^m + \lra{2} \delta_{i-1,k}^m)  \beta_k^n.
	\end{align*}
\end{prop}

\begin{proof}
  We first treat the even case. 
	Recall from \cref{thmMultiquadraticFree}
	that it is sufficient to check the
	formulas on multiquadratic algebras or,
	in other words, for the restrictions to $\Inv(\Sq_m)$
	which by abuse of notation we denote by the same letter here.
	Setting 
	\begin{align*}
		P_i & = P^m_i = \sum_{|K|= i} \prod_{k \in K} x_k, & P &= \sum_{i=0}^m P_i t^i, \\
		L_i &= P^{2m}_i(1,\ldots, 1, x_1, \ldots, x_m), & L &= \sum_{i=0}^{2m} L_i t^i, \\
		B_i &= P_i(1+x_1, \ldots, 1+x_m), & B &= \sum_{i=0}^m B_i t^i, \\
				&\in \Z[x_1, \ldots, x_m]^{S_m}, & &\in \Z[x_1, \ldots, x_m, t]^{S_m},
	\end{align*}
	we have by definition that $\alpha_i = P_i$, $\lambda_i = \lra{2^i} L_i$
	and $\beta_i = \lra{2^i}    B_i$. It is therefore sufficient to prove the identities
	\begin{align*} 
		B_i &= \sum_{k=0}^i \gamma_{i,k}^m  L_k, &
	  L_i &= \sum_{k=0}^i \delta_{i,k}^m  B_k,		
	\end{align*}
	in $\Z[x_1, \ldots, x_m]$.
	To see this, we first note that $P = (1 + x_1 t) \cdots (1 + x_m t)$,
	and that analogously 
	\[
		L = (1+t)^m (1 + x_1 t) \cdots (1 + x_m t) = (1+t)^m P.
	\]
	Since
	\begin{align*} 
		(1+t)^m &= \sum_i \binom{m}{i} t^i, \\
		(1+t)^{-m} &= (1 - t + t^2 -t^3 \cdots )^m = \sum_i (-1)^i \mbinom{m}{i} t^i && \in \Z[[t]],
	\end{align*}
	we get 
	\begin{align*}
		L_i &= \sum_{j=0}^i \binom{m}{i-j} P_j, &
		P_i &= \sum_{j=0}^i (-1)^{i-j} \mbinom{m}{i-j} L_j.
	\end{align*}
	For $B_i$, we find
	\begin{align*} 
		B_i &= \sum_{j=0}^i \binom{m-j}{i-j} P_j, &
		P_i &= \sum_{j=0}^i (-1)^{i-j} \binom{m-j}{i-j} B_j.
	\end{align*}
	Here, the first equation follows from the definition and 
	the second equation via the change of variables
	$x_i \mapsto -1 - x_i$, this since this yields
	$B_i \mapsto (-1)^i P_i$ and $P_j \mapsto (-1)^j B_j$.
  The proof is then finished by combining the formulas. 
	
	For the odd case, we recall from \cref{remOddN} that $\Inv(2m+1) \cong \Inv(2m)$
	via the map induced by $A \mapsto A \oplus K$. 
	Furthermore under this map we have
	\begin{align*} 
		\alpha_i^{2m+1} &\mapsto \alpha_i^{2m}, & 
		\beta_i^{2m+1} &\mapsto \beta_i^{2m}, &
		\lambda_i^{2m+1} &\mapsto \lambda_i^{2m} + \lambda_{i-1}^{2m}.
	\end{align*}
	Here, the last equation follows easily from
  $P^{2m+1}(\mathbf{x}, 1) = P^{2m}_i(\mathbf{x}) +
	P^{2m}_{i-1}(\mathbf{x})$.
	This immediately proves the expression for $\lambda^n_i$ 
	 Denoting $\beta_i^{n}=\sum_{k=0}^{i}\mu_{i,k}^{n}\lambda_k^n$, we have $\mu_{i,k}^{2m}=\mu_{i,k}^{2m+1}+\mu_{i,k+1}^{2m+1}$. Hence we deduce
	 \[
		\mu_{i,k}^{2m+1}=\mu_{i,k}^{2m} - \mu_{i,k+1}^{2m}
+ \mu_{i,k+2}^{2m}	- \ldots 
\]
which proves the expression for $\beta^n_i$.
\end{proof}

\begin{remark} 
  We finish by mentioning that of course one might investigate other bases of $\Inv(n)$.
It seems that one way to produce interesting bases is the concept of power structures, 
see \cite[Definition 2.1]{PP-PowerStructuresGrothendieck}.
On $\WG(K)$, or on the semiring $\QF(K)$, we have two canonical power structures
given by symmetric and exterior powers of quadratic forms.
Pre-composing them with the trace invariant $\Tr$ produces new Witt invariants
which contain subcollections that form bases. The $\lambda$-basis is constructed
in this way using exterior powers. 

Conversely, on the semiring $\Et_*(K) = \bigcup_{n\geq 0} \Et_n(K)$,
under direct sum and tensor product,
we have analogous power structures given by symmetric and exterior powers.
Note that these operations do not commute with the trace invariant,
for example, in general 
$\Tr(\bigwedge^2 \E_a) = \Tr(K) = \lra{1} \neq \lra{a} = \bigwedge^2 \lra{2,2a} =
\bigwedge^2 \Tr(\E_a)$.
Therefore post-composing them with the trace invariant again 
produces new Witt invariants which are priori are not related to the ones from above
and again can give rise to bases for $\Inv(n)$. 
In particular, let us define $\chi^n_i = \chi_i \in \Inv(n)$ by setting
$\chi_i(A) = \Tr(\Sym^i A)$.
Note that $\Sym^i E_a \cong j E_a (\oplus K)$ with $j = \lfloor{\frac{i+1}{2}}\rfloor$
and the summand $(\oplus K)$ only appears when $i$ is even. 
Since $\Tr$ is a (semi-)ring homomorphism such that $\Tr(\E_\delta)^2 = 2 \Tr(\E_\delta)$,  
we have 
\[
  \chi_i = \beta_i + \Z\langle \beta_{i-1}, \ldots, \beta_0\rangle.
\]
Therefore, $\chi_0, \ldots, \chi_m$ form a basis of
$\Inv(n)$ and the base change to the $\beta$-basis
is defined over $\Z$. (However, it seems tricky to write down closed formulas
for the coefficients of this base change.)
In particular, a given Witt invariant $\alpha \in \Inv(n)$ 
is $\beta$-integral if and only if it is $\chi$-integral. 
In \cite[Section 11]{PMPR-QuadraticallyEnrichedPlane},
the last row of each table (labelled by $\sigma$) gives the decomposition
of $Q_{X,D}$ in terms of the $\chi$-basis (denoted $a_0, \ldots, a_m$ in \cite[Section 11]{PMPR-QuadraticallyEnrichedPlane}). 
In \cref{tab:sigmabasisP2}, \cref{tab:sigmabasisP1P1}, and \cref{tab:sigmabasisPbp}, we list the coefficients of several Welschinger--Witt invariants
with respect to the $\chi$-basis.
\begin{table}[!ht]
	\[
	  {\setlength{\extrarowheight}{3pt}
	  \begin{array}{|c|c|}
	   \hline  d & \V_{d}
	  \\  \hhline{|=|=|} 
	  4& -2\chi_0 +2\chi_1 - \chi_2+\chi_3
	  \\ \hline  5& -118\chi_0 -18\chi_1+ 18\chi_2+ \chi_3-\chi_4-\chi_5+\chi_6 
	  \\ \hline  6& -4474\chi_0 -9460\chi_1 + 6644\chi_2 + 828\chi_3-  2680\chi_4 + 836\chi_5 + 284\chi_6-236  \chi_7   +46\chi_8 
  
	  \\ \hline  \multirow{2}{*}{7}& -4519048\chi_0+  6205190 \chi_1  -448528\chi_2+  -2536404 \chi_3+  1184776\chi_4+  350017 \chi_5
	  \\ &  -437346 \chi_6+  76967\chi_7+   46664\chi_8 -23786  \chi_9+ 3336\chi{10}
	  \\ \hline  \multirow{2}{*}{8}&3.05 \cdot 10^9 \chi_0  -1.08\cdot 10^{10} \chi_1 +  3.59 \cdot 10^9\chi_2 +  2.94\cdot 10^9 \chi_3   -2.56 \cdot 10^9\chi_4 + 1.31 \cdot 10^8\chi_5 
	  \\ & +  6.25 \cdot 10^8 \chi_6  -2.64e\cdot 10^8 \chi_7 +  -1.48\cdot 10^7 \chi_8 + 4.16 \cdot 10^7\chi_9 -1.24 \cdot 10^7\chi_{10} +  1.29\cdot 10^6 \chi_{11}
	  \\ \hline  
	\end{array}
	  }
	\]
	\caption{Welschinger--Witt invariants of $\P^2$ in the $\chi$-basis. The coeffecients of $\V_8$ are only approximate.}
	\label{tab:sigmabasisP2}
  \end{table}
  \begin{table}[!h]
	\[
	{\setlength{\extrarowheight}{3pt}
	\begin{array}{|c|c|}
	\hline  (a,b) & \V_{\P^1\times \P^1,(1,1),(a,b)} 
		\\ \hhline{|=|=|}  
  (2,4) &  14 \chi_0 + 2 \chi_1 - \chi_2 +  \chi_3 
  \\  \hline  
  (3,4) &  42 \chi_0   +   18\chi_2- \chi_4  + \chi_6
  \\  \hline  
  (3,5)& -105 \chi_0  + 512\chi_1 -86  \chi_2 -168 \chi_3 + 126\chi_4 -24 \chi_6+ 8\chi_7 
\\  \hline  
		  \multirow{2}{*}{ (4,5)} &  -77528 \chi_0   -264496 \chi_1 + 163008  \chi_2 +   36272\chi_3 
		  \\& -69240 \chi_4 +  17152\chi_5 + 8128 \chi_6-5376 \chi_7 + 912 \chi_8
		  \\  \hline 
	\end{array}
	}
	\]
	\caption{
	Welschinger--Witt invariants of $\P^1 \times \P^1$  in the $\chi$-basis
	}
	\label{tab:sigmabasisP1P1}
\end{table}
\begin{table}[!ht]
	\[
	  {\setlength{\extrarowheight}{3pt}
	  \begin{array}{|c|c|}
	   \hline  (\n,\dbar) & \V_{\n,\dbar}
	  \\  \hhline{|=|=|} 
	   ((1),(3,1)) & \chi_0+ \chi_1 
	  \\ \hline  ((1),(4,2))& 3\chi_0 +  \chi_2
	  \\ \hline  ((1,1),(4,1,1))&  3\chi_0+ 3\chi_1  +  \chi_2 +  \chi_3
	  \\ \hline  ((1,1),(4,2,2))&  \chi_0+ \chi_1 
	  \\ \hline  ((1,1,1),(4,1,1,2))&  6\chi_0+ 2\chi_1 +\chi_2
	  \\ \hline  							
	\end{array}
	  }
	\]
	\caption{Welschinger--Witt invariants $\V_{\n,\dbar}$ in $\chi$-basis}
	\label{tab:sigmabasisPbp}
  \end{table}  
Note that while for small degrees the $\chi$-coefficients are often simpler than
the $\beta$-coefficients, this seems to turn around for large degrees where the former ones seem to explode.
\end{remark}

\section{Enumerative description of Witt invariance}
\label{WittInvarianceQuadratic}

In the following, we give a proof of \cref{thmQuadraticWitt}
for perfect fields $k\in\Fields_{2,3}$ using the enumerative description of $Q_{X,D}$ (whenever the latter applies).

Recall that the linear system $|D|$ is isomorphic to projective space $\P^N_k$. 
Given a curve $C \subset X$ with class $D$, we denote
by $[C] \in |D|$ the associated (scheme-theoretic closed) point. 
We define the \emph{field of moduli} 
$k_C$ of $C$ as the residue field of the point
$[C] \in |D|$. The extension $k \to k_C$ is finite. 
We define $\widetilde C \subset X \otimes_k k_C$
as the preimage of the canonical $k_C$-point associated to $[C]$ 
with respect 
to the universal curve $\mathcal{U} \to |D|$, 
that is, given by the fibre product diagram
\[\begin{tikzcd}
	{\widetilde C} & {\P^2_{k_C}} & {\Spec(k_C)} \\
	{\mathcal{U}} & {\P^2_k \times |D|} & {|D|}
	\arrow[hook, from=1-1, to=1-2]
	\arrow[from=1-1, to=2-1]
	\arrow[two heads, from=1-2, to=1-3]
	\arrow[from=1-2, to=2-2]
	\arrow[from=1-3, to=2-3]
	\arrow[hook, from=2-1, to=2-2]
	\arrow[two heads, from=2-2, to=2-3]
\end{tikzcd}\]
Note that $k_{\widetilde C} = k_C$. 
We say that $C$ is \emph{rational} and \emph{nodal}, respectively, 
if $\widetilde C$ is geometrically rational (including geometrically integral)
or geometrically nodal
(we warn the reader that this may be slightly non-standard notation). 
For a point $p \in \widetilde C$, we denote by $k_p$
the residue field of $p$ (again, the extension $k_C \to k_p$ is finite). 
If $p$ is a node of $\widetilde C$, 
the projective tangent cone $\P T_p \widetilde C$ consists of two points,
so its coordinate ring $A_p = A(\P T_p \widetilde C)$ lies in $\Et_2(k_p)$. 
We denote by $\delta_p \in \Sq(k_p) = k_p^*/(k_p^*)^2$ the unique element
such that $A_p \cong \E_{\delta_p}$. 

Recall that the norm map $\Nm_{L/K} \colon L^* \to K^*$
for a finite field extension $K \to L$ is multiplicative and hence
descends to $\Nm_{L/K} \colon \Sq(L) \to \Sq(K)$.
We define the \emph{quadratic weight} of $\widetilde C$ as
\[
  \omega(\widetilde C) = 
	  \prod_{p \in \widetilde C  \text{ node}}  \Nm_{k_p/k_C}(\delta_p) 
		\in \Sq(k_C).
\]
Next, let us denote by $\tr_{L/K} \colon L \to K$ the trace map of a finite separable
extension $K \to L$.
This gives rise to a so-called \emph{trace transfer map} 
\begin{align*}
  (\tr_{L/K})_* \colon \Sq(L) &\to \QF(K), \\
	  [a] &\mapsto (q_a \colon L \to K, x \mapsto \tr_{L/K}(ax^2)).
\end{align*}
Since $k$ is perfect, the extension $k \to k_C$ is separable
and we define the \emph{quadratic weight} of $C$ as 
\[
  \omega(C) = (\tr_{k_C/k})_* (\omega(\widetilde C)) \in \QF(K). 
\]
We summarize the construction in the following diagram.
\[\begin{tikzcd}
	k & {k_C} & {k_p} & {A_p} \\
	{\omega(C) \in \QF(k)} & {\omega(\widetilde C)  \in \Sq(k_C)} & {\delta_p \in \Sq(k_p)}
	\arrow[from=1-1, to=1-2]
	\arrow[from=1-2, to=1-3]
	\arrow[from=1-3, to=1-4]
	\arrow["{\Sq = \Et_2}"{description}, dotted, from=1-4, to=2-3]
	\arrow["{\tr_*}", dotted, from=2-2, to=2-1]
	\arrow["\Nm", dotted, from=2-3, to=2-2]
\end{tikzcd}\]

Let $n=-K_X\cdot D-1$, and
fix $A \in \Et_n(K)$. 
Let us assume that  
there exists a generic point configuration $\PP \subset X$ such that $\PP = \Spec(A)$
and for which \cite[Theorem 3]{degree} applies.
Any rational curve $C$ with class $D$ and such that  $\PP \subset C \subset X$ is nodal
and the enumerative definition of $\widehat Q_{X,D}(A)$ is given by
\begin{equation} \label{eq:EnumerativeDefinition} 
  \widehat Q_{X,D}(A) = \sum_{\substack{\PP \subset C \subset X \\ C \text{ rational} \\ C \in |D|}} \omega(C),
\end{equation}
see \cite[Theorem 3]{degree}.

\begin{example} \label{exRealQuadratic}
  Let us assume $k = \R$ (cf.\ \cite[Remark 2.5]{Levine-Welschinger}). 
	Since $\R$ and $\C$ are the only finite extensions
	of $\R$, any étale algebra $A \in \Et_n(\R)$ is of the form
	$A = R_s$ for some $0 \leq s \leq m$. 
	A generic point configuration $\PP \subset X$ satisfies $\PP = \Spec(R_s)$
	if it contains $n-s$ real points and $s$ pairs of complex conjugated points. 
	Given $\PP \subset C \subset X$, there 
	are two options. 
	\begin{itemize}
		\item 
		If $k_C = \C$, then $\omega(\widetilde C) = 1 \in \{1\} = \Sq(\C)$. 
		  Since $(\tr_{\C/\R})_*(1) = \Tr_\R(\C) = h \in \QF(\R)$, hence
			$\omega(C) = h$. 
		\item 
		If $k_C = \R$, we have $C = \widetilde C$ and the trace step is trivial.
		For any node $p \in C$, there are two options:
		\begin{itemize}
			\item 
			  If $k_p = \C$, we have $\delta_p = 1 \in \{1\} = \Sq(\C)$
				and therefore $\nm_{\C/\R}(\delta_P) = 1 \in \{\pm 1\} = \Sq(\R)$. 
				Hence such a node can be disregarded. 
			\item 
			  If $k_p = \R$, there are again two options: Either
				$\delta_p = 1 \in \Sq(\R)$, a \emph{hyperbolic} node with two 
				real tangent lines, or $\delta_p = -1 \in \Sq(\R)$, an \emph{elliptic}
				node with a pair of complex conjugated tangent lines. 
		\end{itemize}
		We conclude that if $k_C = \R$ we have $\omega(C) = \lra{(-1)^{m(C)}} \in \QF(\R)$ where
		$m(C)$ denotes the number of real elliptic nodes in $C$. 
	\end{itemize}
	In summary, only looking at the class $Q_{X,D}(A) \in \Z = \W(\R)$,
	we see from \cref{eq:EnumerativeDefinition}
	that $Q_{X,D}( A)$ is equal to the \emph{signed} sum of real (that is, $k_C = \R$) rational curves $C$
	through $\PP$, where the sign of $C$ is given by $(-1)^{m(C)}$. 
	In other words  $Q_{X,D}( A) = \Wel_X(D;s)$, the Welschinger invariant defined
	in \cref{secWelschinger}.
\end{example}

We summarize a few elementary facts about traces and norms
that will be used below (we leave the proofs to the reader).

\begin{lemma} \label{lemNormTraceProps}
  Let $A$ be a  étale algebra over $K$ of finite degree,
	and fix $a \in A$. 
	\begin{enumerate}
		\item 
		Given a field extension $K \to L$, 
		we have 
		\begin{gather*} 
			\nm_{(A\otimes L)/L}(a \otimes 1) = \nm_{A/K}(a), \;\;\;
			\tr_{(A\otimes L)/L}(a \otimes 1) = \tr_{A/K}(a), \\
			(\tr_{(A\otimes L)/L})_*(a \otimes 1) = (\tr_{A/K})_*(a) \otimes L 
			\in \WG(L). 
		\end{gather*}
		\item 
		Given a (algebra) decomposition $A = A_1 \times \ldots \times A_l$ 
		such that $a = (a_1, \ldots, a_l)$, 
		we have 
		\begin{gather*} 
			\nm_{A/K}(a) = \prod_{i=1}^l \nm_{A_i/K}(a_i), \;\;\;
			\tr_{A/K}(a) = \sum_{i=1}^l \tr_{A_i/K}(a_i), \\
			(\tr_{A/K})_*(a) = \sum_{i=1}^l (\tr_{A_i/K})_*(a_i)
			\in \WG(K). 
		\end{gather*}
	\end{enumerate}
\end{lemma}

  Suppose that 
	for any perfect extension
	$k \to K$ and $A \in \Et_n(K)$ there exists
	a generic point configuration $\PP \subset X_K$
	such that $\PP \cong \Spec(A)$.
	This ensures that $Q_{X,D}$ can be computed
	using the enumerative definition from 
	\cref{eq:EnumerativeDefinition}.
	This is the case for example 
	if $k$ is infinite 
	and $X$ is $k$-rational.

\begin{proof}[Proof of \cref{thmQuadraticWitt} under the above assumption]
	Given $k \to K \to L$ and $A \in \Et_n(K)$, 
	we need to show $Q_{X_K,D_K}( A) \otimes L = Q_{X_L,D_L}(A \otimes L)$.
	Without loss of generality we may assume $k=K$, which saves us writing
	a few indices. 
	To use \cref{eq:EnumerativeDefinition}, 
	we fix a generic point configuration $\PP \subset X$ such that $\PP = \Spec(A)$. 
	Let $C$ be a rational curve with class $D$ such that $\PP \subset C \subset X$.
	Then the irreducible decomposition $C \otimes L = C_1 \cup \ldots \cup C_\ell$
	consists of rational curves in $X_L$ containing $\PP_L$, and all such 
	curves occur in this way for a unique $C$. Therefore it remains
	to prove
	$\omega(C) \otimes L = \omega(C_1) + \ldots + \omega(C_\ell) \in \WG(L)$.
	We first consider the case $k_C = k$ (and hence $\ell = 1$). 
		
	\begin{lemma} \label{lemBaseCase}
		Let $C \subset \P^2_K$ be a rational nodal curve such that $k_C = K$ 
		and $K \to L$	a field extension. Then the induced map 
		$\Sq(K) \to \Sq(L)$ satisfies
		\[
		  \omega(C) \mapsto \omega(C \otimes L). 
		\]
	\end{lemma}
	
	\begin{proof}
		Let $p \in C$ be a node. We proceed similar as above:
		Given the irreducible decomposition $p \otimes L = p_1 \cup \ldots \cup p_\ell$, 
		it suffices to show that 
		\[
		  \Sq(k) \ni \nm(\delta_p) \mapsto \prod_{j=1}^\ell\nm(\delta_{p_j}) 
		  \in \Sq(L). 
		\]
		We have $k_p \otimes_K L = L_1 \times \ldots \times L_\ell$  such that
		$k_{p_i} = L_i$ for all $i$. Given $\delta \in k_p^*$ such that $[\delta] = \delta_p$
		and $\delta \otimes 1 = (\delta_1, \ldots, \delta_\ell)$,
		we have 
		\[
		  A_{p_i} = A_p \otimes_{k_p} L_i = E_\delta \otimes_{k_p} L_i = E_{\delta_i}, 
		\]
		hence $[\delta_i] = \delta_{p_i}$ for all $i$. 
		Then the statement follows from the norm formulas of 
		\cref{lemNormTraceProps}.
	\end{proof}
	
	We return to the general case, which we finish in a similar way.
	We have $k_C \otimes L = L_1 \times \ldots \times L_\ell$
	such that $L_i$ is the field of moduli of $C_i$. 
	One checks directly that $\widetilde{C} \otimes_{k_C} L_i = \widetilde{C_i}$
	in a canonical way. 
	Fix $w \in k_C$ such that $\omega(\widetilde{C}) = [w]  \in \Sq(k_C)$ and assume 
	$w \otimes 1 = (w_1, \ldots w_l)$. 
	By \cref{lemBaseCase}, we have 
	$\omega(\widetilde{C_i}) = [w_i] \in \Sq(L_i)$ for all $i$.
	Now the statement follows from the trace formulas of 
	\cref{lemNormTraceProps}.
\end{proof}

 \bibliographystyle{alpha}

 \bibliography{Bibli}

\end{document}